\documentclass[12pt]{amsart}
\usepackage[margin=1in]{geometry}

\usepackage{graphicx,amssymb}
\usepackage{hyperref}
\hypersetup{
    bookmarks=true,         % show bookmarks bar?
    pdffitwindow=false,     % window fit to page when opened
    pdfstartview={FitH},    % fits the width of the page to the window
    colorlinks=ture,       % false: boxed links; true: colored links
}

\usepackage{setspace}
\newtheorem{thm}{Theorem}[section]
\newtheorem{lem}[thm]{Lemma}
\newtheorem{prop}[thm]{Proposition}

\theoremstyle{definition}
\newtheorem{defn}[thm]{Definition}

\theoremstyle{remark}
\newtheorem{rem}[thm]{Remark}

\theoremstyle{Fact}

\theoremstyle{Claim}
\newtheorem{claim}[thm]{Claim}

\theoremstyle{Example}

\numberwithin{equation}{section}

 \makeatletter
\def\tagform@#1{\maketag@@@{\ignorespaces#1\unskip\@@italiccorr}}
\let\orgtheequation\theequation
\def\theequation{(\orgtheequation)}
\makeatother

  \allowdisplaybreaks
  \everymath{\displaystyle}

\begin{document}

\title[Existence and Uniqueness to the Prandtl Equations]{Local-in-Time Existence and Uniqueness
of Solutions to the Prandtl Equations by Energy Methods}

%    author one information
% \author[short version for running head]{name for top of paper}

\author{Nader Masmoudi}
\address{Courant Institute, 251 Mercer Street,
New York, NY 10012-1185,
USA,}
\email{masmoudi@courant.nyu.edu  }

\author{Tak Kwong WONG}
\address{Department of Mathematics\\ University of California, Berkeley \\
970 Evans Hall $\#$3840\\ Berkeley, CA 94720, USA}
%\curraddr{}
\email{takkwong@math.berkeley.edu}
%\thanks{}

%    author two information
%\author{}
%\address{}
%\curraddr{}
%\email{}
%\thanks{}

%    \subjclass is required.
%\subjclass[2010]{Primary }

\date{\today}

%\dedicatory{}

%    "Communicated by" -- provide editor's name; required.
%\commby{}

\begin{abstract}
   We prove local existence and uniqueness
 for  the two-dimensional  Prandtl system in weighted Sobolev spaces
under the Oleinik's monotonicity assumption. In particular we do not
use the Crocco transform.  Our proof is based on a {\it new nonlinear
energy estimate}  for the  Prandtl system. This new energy estimate is
based on a cancellation property which is valid under the  monotonicity assumption.
To construct the solution,
we  use a  regularization of  the system  that  preserves  this nonlinear
structure. This new nonlinear
structure may give some insight on the convergence properties from
Navier-Stokes system to Euler system when the viscosity goes to zero.
\end{abstract}

\maketitle
\tableofcontents

% MATH -------------------------------------------------------------------
 \newcommand{\R}{\mathbb{R}}
 \newcommand{\T}{\mathbb{T}}
 \newcommand{\TR}{\mathbb{T} \times \mathbb{R}^+}
 \newcommand\N{{\mathbb N}}
 \newcommand\Z{{\mathbb Z}}
 \newcommand\e{{\epsilon}}
 \newcommand\w{\omega}
 \newcommand\dx{\partial_x}
 \newcommand\dy{\partial_y}
 \newcommand\dt{\partial_t}
 \newcommand\dxx{\partial^2_x}
 \newcommand\dyy{\partial^2_y}
 \newcommand\Dalpha{D^\alpha}
 \newcommand\Dawe{\Dalpha \wep}
 \newcommand\Hs{H^{s,\gamma}_{\sigma,\delta}}
 \newcommand\Hsthree{H^{s+3,\gamma}_{\sigma, \delta}}
 \newcommand\Hsfour{H^{s+4,\gamma}_{\sigma, \delta}}
 \newcommand\Hg{H^{s,\gamma}_g}
 \newcommand\aep{a^\e}
 \newcommand\gep{g^\e}
 \newcommand\uep{u^\e}
 \newcommand\vep{v^\e}
 \newcommand\wep{\omega^\e}
 \newcommand\pep{p^\e}
 \newcommand\constC{C_{s, \gamma, \sigma, \delta}}
 \newcommand\constCw{C_{s',\gamma,\delta,\|\w\|_{H^{s'+4,\gamma}}}}
 %\newcommand{\abs}[1]{\left\vert#1\right\vert}
 %\renewcommand\bar{\overline}

%%% ----------------------------------------------------------------------

%==========================================================================
%   Section 1 : Introduction
%==========================================================================
\section{Introduction}\label{s:intro}

The zero-viscosity limit of  the incompressible Navier-Stokes system
in a bounded domain, with Dirichlet boundary conditions,   is one
of the most
challenging open  problems  in Fluid Mechanics.
This is  due to the formation of
a boundary layer which appears because  we can not impose
the same  Dirichlet boundary condition for the Euler equation.
 This boundary layer satisfies formally
 the Prandtl system.  Indeed, in 1904,
Prandtl \cite{Prandtl04}   suggested that there  exists a thin layer called
boundary layer, where the solution $\vec{u}$ undergoes a sharp
transition from a solution to the  Euler system  to the no-slip
boundary condition $ \vec{u}  = \vec{0} $  on  $\partial \Omega $ of the
Navier-Stokes system. In other words, Prandtl proved
formally that the solution $\vec{u}$ of the Navier-Stokes system
can be written as   $\vec{u} = \vec{U} + \vec{u}_{BL}$ where $\vec{U}$ solves
the Euler system with $\vec{U} \cdot {\bf n} =0$ on the boundary  and
$\vec{u}_{BL}$ is small except near the boundary. In rescaled variables
$\vec{U} + \vec{u}_{BL}$ solves the Prandtl system.  When studying this problem,
there are at least 3 main questions:

\begin{itemize}
\item[(a)] The local well-posedness of  the Prandtl system;

 \item[(b)] Proving the convergence of solutions of the  Navier-Stokes system towards
a solution of the Euler system;

\item[(c)] The justification of the boundary layer expansion.
\end{itemize}

In full generality, these questions are still open except in
the analytic case where (a)-(c) can be proved \cite{SC98a,SC98b,LCS03,KV11}.

Concerning (a), the main existence result is due to  Oleinik who
 proved the local existence for the Prandtl system
\cite{Oleinik63,Oleinik66}  under a monotonicity assumption and using
the Crocco transform
 (see also \cite{OS99}).
These solutions can be extended as global weak solutions if the
pressure gradient is favorable $(\partial_x p \leq 0) $
\cite{XZ04,XZZ}.
However, E and Engquist \cite{EE97} proved a
blow up
result for the Prandtl system for some special
type of initial data.
More recently,  G\'erad-Varet and Dormy \cite{GVD10}   proved ill-posedness
for the linearized Prandtl equation around a nonmonotonic
shear flow (see also \cite{GN11,GVN12}).

Concerning (b), the main result is a convergence criterion due
to Kato  \cite{Kato84}  that basically says that convergence
is equivalent to the fact that there is no dissipation in
a very  thin layer (of size $\nu$).  This criterion was
extended in different directions  (see \cite{TW97a,Wan01,Kel07}).
Also, in \cite{Masmoudi98arma}, it is proved that the convergence
holds if the horizontal viscosity goes to zero slower than
the vertical one. It is worth noting that the Prandtl system
is the same in this case.

Concerning (c),   there is a negative result by Grenier \cite{Grenier00cpam}
who proves that the expansion does not hold in $W^{1,\infty}$.
Of course this does not prevent (b) from holding.

There are many review papers about
 the inviscid limit of the Navier-Stokes in a bounded
domain  and the
Prandtl system from different aspects (see \cite{CS00,E00,Masmoudi07hand}).
Let us also mention that when considered in the whole space \cite{Swann71,Kato72,Masmoudi07cmp}
or with other boundary conditions such as Navier boundary condition \cite{Xin,Iftimie-Sueur,BC11,MR11prep}
 or incoming flow \cite{TW02}, the convergence problem becomes simpler
since there is no boundary layer or the boundary layer is stable.

The prime objective of this paper  is to prove the local existence and uniqueness for the two-dimensional
Prandtl system  under the Oleinik's monotonicity assumption in certain weighted energy spaces
without using  the Crocco transform.
Precise statement will be provided in section  \ref{s:mainresults}.
In addition to giving a very simple understanding of the monotonicity assumptions,
our result may give us a better understanding about the questions (b) and (c)  since
it is given in physical space.
Nevertheless,  we are still not able to use our new nonlinear energy to study the
convergence problem (b) or (c).
  In spirit this paper is similar to our
previous paper about the Hydrostatic Euler equations \cite{MW11prep} where we gave
a proof of existence and uniqueness  in physical space under some
convexity assumption of the profile. The previous known proof of Brenier \cite{Brenier99}
uses Lagrangian coordinates and requires more assumptions on the initial data.

Let us end this introduction by outlining the structure of this paper. In section
\ref{s:mainresults} we will state our main result, that is, theorem \ref{t:Hs exist PS}.
Explanations of our approach and approximate scheme will be provided in sections
\ref{s:outline} and \ref{s:approx}. Assuming the solvability of approximate systems,
we will derive our new weighted a priori estimates in section \ref{s:RPE}. Using these
weighted estimates, we will complete the proof of our main theorem \ref{t:Hs exist PS}
in section \ref{s:pf main thm}. In section \ref{s:exist RPE} we will solve the
approximate systems. For the sake of self-containedness, we will also provide
several elementary proofs and computations in appendices \ref{s:appendixA} -
\ref{s:appendixE}. Finally, let us mention that the new  preprint \cite{AWXY12}  also
considers the existence for the  Prandtl system in physical space.
The methods of proof are very different.

%=================================================================
%
%    2.    Main Results
%
%==================================================================
%
\section{Main Result}\label{s:mainresults}
In this section we will first introduce the Prandtl equations, and then describe
our solution spaces as well as our main result. Main difficulties and brief explanation
of our approach will be given in section \ref{s:outline}.

Throughout this paper, we are concerned with the two-dimensional Prandtl equations in a
periodic domain $\TR := \{(x,y);\, x \in \R/\Z, 0 \le y < + \infty\}$:
%
%   eqn 2.1
\begin{equation}\label{e:PE}
\left\{\begin{aligned}
    \partial_t u + u\partial_x u +v\partial_y u & = \partial^2_y u - \partial_x p
    && \text{ in } [0,T] \times \TR\\
    \partial_x u + \partial_y v & = 0 && \text{ in } [0,T] \times \TR\\
     u|_{t=0} & = u_0 && \text{ on } \TR\\
    u|_{y=0} = v|_{y=0} & = 0 && \text{ on } [0,T] \times \T\\
    \lim_{y \to + \infty} u(t,x,y) & = U  && \text{ for all } (t,x) \in [0,T] \times \T,
\end{aligned}\right.
\end{equation}
where the velocity field $(u,v) := (u(t,x,y), v(t,x,y))$ is an unknown,
the initial data $u_0 := u_0 (x,y)$ and the outer flow $U := U(t,x)$ are given and
satisfy the compatibility conditions:
%
%
%  eqn 2.2
\begin{equation} \label{e:u0 conditions}
    u_0 |_{y=0} = 0 \qquad \text{ and } \qquad \lim_{y \to +\infty} u_0 = U.
\end{equation}
Furthermore, the given scalar pressure $p := p(t,x)$ and the outer flow $U$ satisfy
the well-known Bernoulli's law:
%
%  eqn 2.3
\begin{equation}\label{e:Ber}
    \partial_t U + U \partial_x U = - \partial_x p.
\end{equation}
In this work, we will consider system \eqref{e:PE} under the Oleinik's monotonicity assumption:
%
%  eqn 2.4
\begin{equation}\label{e:MA}
    \omega :=\partial_y u >0 .
\end{equation}
Under this hypothesis, one must further assume $U > 0$.

Let us first introduce the function space in which the Prandtl equations \eqref{e:PE}
will be solved. Denoting the vorticity $\omega := \partial_y u$, we define the  space
$\Hs$   for $\omega$ by
\begin{equation*}
    \Hs := \left\{ \omega : \TR \to \R; \, \| \omega \|_{H^{s,\gamma}} < + \infty,
    (1+y)^\sigma \omega \ge \delta \text{ and } \sum_{|\alpha| \le 2} |(1+y)^{\sigma + \alpha_2}
     \Dalpha \omega |^2 \le \frac{1}{\delta^2} \right\}
\end{equation*}
where $s\ge 4$, $\gamma \ge 1$, $\sigma > \gamma + \frac{1}{2}$, $\delta \in (0,1)$,
$\Dalpha := \partial^{\alpha_1}_x \partial^{\alpha_2}_y$ and the weighted $H^s$ norm
$\| \cdot \|_{H^{s,\gamma}}$ is defined by
%
%  eqn 2.5
\begin{equation}\label{e:w2Hsgamma}
    \|\omega\|^2_{H^{s, \gamma} (\TR)} := \sum_{|\alpha| \le s} \|(1+y)^{\gamma + \alpha_2}
    \Dalpha \omega \|^2_{L^2 (\TR)}.
\end{equation}
Here, the main idea is adding an extra weight $(1+y)$ for each $y$-derivative. This corresponds
to the weight $\frac{1}{y}$ in the Hardy type inequality. Furthermore, we also denote
$H^{s,\gamma} := \{ \w :\TR \to \R;\, \|\w\|_{H^{s,\gamma}} < + \infty\}$.
\begin{rem}[Requirement: $\sigma > \gamma + \frac{1}{2}$]
If $\sigma \le \gamma + \frac{1}{2}$, then one may check that $\Hs (\TR)$ is an empty set. Thus, we
must have the hypothesis $\sigma > \gamma + \frac{1}{2}$.
\end{rem}

Now, we can state our main result:

\begin{thm}[Local $\Hs$ Existence and Uniqueness to the Prandtl Equations \eqref{e:PE}] \label{t:Hs exist PS}
Let $s \ge 4$ be an even integer, $\gamma \ge 1$, $\sigma > \gamma + \frac{1}{2}$ and
$\delta \in (0,\frac{1}{2})$. For simplicity\footnote{The regularity hypothesis on the outer
flow $U$ is obviously  not optimal in the viewpoint of our a priori weighted energy estimates.
One may further loosen the regularity requirement on $U$ by applying other approximate schemes.
We leave this for the interested reader.}, we suppose that the outer flow $U$ satisfies
%
%  eqn 2.6
\begin{equation} \label{e:supU}
    \sup_t \||U|\|_{s+9, \infty} := \sup_t \sum^{[\frac{s+9}{2}]}_{l=0}
    \|\dt^l U\|_{W^{s-2l+9, \infty} (\T)} < + \infty.
\end{equation}
Assume that $u_0 - U \in H^{s, \gamma-1}$ and the initial vorticity
$\omega_0 := \partial_y u_0 \in H^{s,\gamma}_{\sigma, 2\delta}$. In addition, when $s = 4$, we
further assume that $\delta > 0$ is chosen small enough such that
%
%
% eqn 2.7
\begin{equation} \label{e:mainthm con}
    \|\w_0\|_{\Hg} \le C \delta^{-1}
\end{equation}
where the norm $\|\cdot\|_{\Hg}$ will be defined by \eqref{e:Hgsgamma} and $C$ is a universal
constant. Then there exist a time $T := T(s,\gamma,\sigma,\delta,\|\omega_0 \|_{H^{s, \gamma}},
U)>0$ and a unique classical solution $(u,v)$ to the Prandtl equations \eqref{e:PE} such that
$u-U \in L^\infty([0,T]; H^{s, \gamma-1}) \cap
% \bigcap_{s'<s}
C([0,T];H^{s}-w)$ and the vorticity
$\omega := \partial_y u \in L^\infty([0,T];\, \Hs) \cap
% \bigcap_{s'<s}
C([0,T];H^{s}-w)$,  where $ H^{s}-w $ is the space $ H^{s} $ endowed with
its weak topology.
\end{thm}
%
%
%  remark 2.3

\begin{rem}[$U \equiv$ constant]
When the outer flow $U$ is a constant, one may show that the life-span $T$ stated in
theorem \ref{t:Hs exist PS} is independent of $U$. For the reasoning, see
remark \ref{r:remark}.
\end{rem}

The proof of theorem \ref{t:Hs exist PS} is based on our new weighted energy estimates, which relies
on a nonlinear cancelation property that holds under the Oleinik's monotonicity assumption \eqref{e:MA}.
An outline of our proof will be given in sections \ref{s:outline} and \ref{s:approx}, and the
detailed analysis will be provided in sections \ref{s:RPE} - \ref{s:exist RPE}.

Before we proceed, let us comment on our notation. Throughout this paper, all constants
$C$ may be different in different lines. Subscript(s) of a constant illustrates the
dependence of the constant, for example, $C_s$ is a constant depending on $s$ only.

%==============================================================
%
%   3. Difficulties and Outline of Our Approach
%
%===============================================================

\section{Difficulties and Outline of Our Approach}\label{s:outline}
The aim of this section is to explain main difficulties of proving theorem
\ref{t:Hs exist PS} as well as our strategies for overcoming them. Let us begin by stating
the main difficulties as follows.

In order to solve the Prandtl equations \eqref{e:PE} in certain $H^s$ spaces, we have to overcome the following three difficulties:
\begin{itemize}
\item[(i)] the vertical velocity $v := - \partial^{-1}_y \partial_x  u$ creates a loss of $x$-derivative, so the standard energy estimates do not apply;
\item[(ii)] the unboundedness of the underlying physical domain $\TR$ allows certain quantities growth at $y = + \infty$ even if the solution is smooth or bounded in $H^s$;
\item[(iii)] the lack of higher order boundary conditions at $y=0$ prevents us to apply the integration by parts in the variable $y$, but it is a standard and crucial step to deal with the operator $\partial_t - \partial^2_y$.
\end{itemize}

Indeed, difficulty (i) is the major problem for the Prandtl equations \eqref{e:PE}, and it explains why there are just a few existence results in the literature. The key ingredient of the current work is to develop a $H^s$ control by considering a special $H^s$ norm (see \eqref{e:Hgsgamma} below) which can avoid the regularity loss created by $v$. Difficulty (ii) is somewhat based on the fact that Poincar\'{e} inequality does not hold for the unbounded domain $\TR$. However, one may overcome this technical problem by  replacing the Poincar\'{e} type inequalities by  Hardy type inequalities. This is our main reason for adding a weight $(1+y)$ for each $y$-derivative to our $H^s$ energies \eqref{e:w2Hsgamma} and \eqref{e:Hgsgamma}. Difficulty (iii) seems to be an obstacle, but it is not. A reconstruction argument for the higher order boundary conditions can fix this technical difficulty when $s$ is even, see lemma \ref{l:reduce bdry} for more details.

Now, let us explain our new weighted energy, which is the main novelty in this paper.

Judging from nonlinear cancelations, the weighted norm \eqref{e:w2Hsgamma} is not suitable for estimating solutions of the Prandtl equations \eqref{e:PE}. Thus, we introduce another weighted norm for the vorticity $\omega$, namely
%
%  eqn 3.1
\begin{equation}\label{e:Hgsgamma}
    \|\omega\|^2_{H^{s, \gamma}_g (\TR)}
    := \|(1+y)^\gamma g_s \|^2_{L^2 (\TR)} + \sum_{\substack{|\alpha| \le s\\ \alpha_1 \le s-1}}
    \|(1+y)^{\gamma + \alpha_2} D^\alpha \omega \|^2_{L^2 (\TR)}
\end{equation}
where
\[
     g_s := \dx^s \omega - \frac{\dy\w}{\w} \dx^s (u-U) \qquad \text{ and }
     \qquad u(t,x,y) := \int^y_0 \w (t,x,\tilde{y})\, d\tilde{y}
\]
provided that $\omega := \partial_y u > 0$. The difference between norms \eqref{e:w2Hsgamma} and \eqref{e:Hgsgamma} is that we replace the weighted $L^2$ norm of $\dx^s \omega$ by that of $g_s$, which is a better quantity because $g_s$ can avoid the loss of $x$-derivative (i.e., difficulty (i) above), see subsubsection \ref{sss:geps} for further explanation.

The first important observation is that as long as $\omega\in\Hs$, we can show that the new weighted norm \eqref{e:Hgsgamma} is almost equivalent to the weighed $H^s$ norm \eqref{e:w2Hsgamma}, that is,
%
% eqn  3.2
\begin{equation} \label{e:wHg}
    \|\w\|_{\Hg} \lesssim \|\w\|_{H^{s,\gamma}} + \|u-U\|_{H^{s,\gamma-1}}
    \lesssim \|\w\|_{\Hg} + \|\dx^s U\|_{L^2}
\end{equation}
provided that $\w = \dy u$, $u|_{y=0}=0$ and $ \lim_{y \to +\infty} u =U$. The proof of \eqref{e:wHg} is elementary, and will be given in
appendix \ref{s:appendixA}. In spirit of \ref{e:wHg}, we will
estimate $\|\w\|_{\Hg}$ instead of $\|\w\|_{H^{s,\gamma}}$.

The second important
observation is that due to the nonlinear
cancelation, the loss of $x$-derivative is avoided by the norm $\|\cdot\|_{\Hg}$,
so one can simply derive a priori energy estimates on $\w$ by applying the standard
energy methods. These estimates indeed can be extended to $\wep := \dy \uep$, which
is the regularized vorticity of the regularized Prandtl equations \eqref{e:RPE} below,
because the regularization \eqref{e:RPE} preserves the nonlinear structure of the
original Prandtl equations \eqref{e:PE}. See subsection \ref{ss:weighted est} for the
detailed analysis.

Once we have obtained the weighted energy estimates, it remains to derive weighted
$L^\infty$ controls on the lower order derivatives of $\w$ so that we can close our
estimates in the function space $\Hs$. The derivations of these $L^\infty$ estimates
are standard: ``viewing'' the evolution equations of the lower order derivatives as
``linear'' parabolic equations with coefficients involving higher order terms that
can be bounded by the weighted energies, we can obtain our desired estimates by
the classical maximum principle since we have already controlled the weighted energies.
These weighted $L^\infty$ estimates are also extendable to the regularized vorticity
$\wep$, see subsection \ref{ss:Linfty lower} for further details.

In order to prove the existence, we will construct an approximate scheme which
keeps the a priori estimates described above. Due to the nonlinear cancelation, our
a priori estimates are complicated in certain sense, so the construction of the
approximate scheme is tricky. An outline of this construction will be given
in section \ref{s:approx}.

For the uniqueness, it is an immediate consequence of a $L^2$ comparison principle
(see proposition \ref{p:L2 compare}), whose proof relies on a nonlinear
cancelation that is similar to the one applied in the energy estimates.

%=======================================================================
%
%   4.    Approximate Scheme
%
%========================================================================

\section{Approximate Scheme} \label{s:approx}
The main purpose of this section is outlining the approximate systems which we apply to
prove the existence. Since our weighted $H^s$ a priori estimates are somewhat more
nonlinear than  usual,
the approximate scheme is slightly more complicated.

Our approximate scheme has three different levels and will be explained as follows.

The first approximation of \eqref{e:PE} is the regularized Prandtl equations:
for any $\e > 0$,
%
%   eqn 4.1
\begin{equation}\label{e:RPE}
\left\{\begin{aligned}
    \dt\uep+\uep\dx\uep+\vep\dy\uep& = \e^2 \dxx\uep + \dyy\uep -\dx\pep && \text{ in }
    [0,T] \times \TR\\
    \dx\uep + \dy\vep & = 0 && \text{ in } [0,T] \times \TR\\
    \uep|_{t=0} & = u_0 && \text{ on } \TR\\
    \uep|_{y=0} = \vep|_{y=0} & = 0 && \text{ on } [0,T] \times \T\\
    \lim_{y \to +\infty} \uep(t,x,y) & = U && \text{ for all } (t,x) \in [0,T] \times \T,
\end{aligned}\right.
\end{equation}
where $\pep$ and $U$ satisfy a regularized Bernoulli's law:
%
% eqn 4.2
\begin{equation}\label{e:RBer}
    \dt U + U \dx U = \e^2 \dxx U - \dx \pep.
\end{equation}
Or equivalently, the regularized vorticity $\wep := \dy \uep$ satisfies the following
regularized vorticity system: for any $\e > 0$,
%
% eqn 4.3
\begin{equation}\label{e:Rvort}
\left\{\begin{aligned}
    \dt\wep + \uep\dx\wep+\vep\dy\wep & = \e^2 \dxx\wep+\dyy\wep && \text{ in } [0,T] \times \TR\\
    \wep|_{t=0} & = \w_0 := \dy u_0 && \text{ on } \TR\\
    \dy \wep|_{y=0} & = \dx \pep && \text{ on } [0,T] \times \T
\end{aligned}\right.
\end{equation}
where the velocity field $(\uep,\vep)$ is given by
%
% eqn 4.4
\begin{equation}\label{e:Ruv}
%\left\{\begin{aligned}
    \uep(t,x,y) := U - \int^{+\infty}_y \wep (t,x,\tilde{y})\,d\tilde{y} \quad \text{ and }
    \quad \vep (t,x,y) := - \int^y_0 \dx \uep (t,x,\tilde{y})\, d\tilde{y}.
%\end{aligned}\right.
\end{equation}

The main idea of this approximation is adding the viscous terms $\e^2\dxx \uep$ and
 $\e^2 \dxx \wep$ to avoid the loss of $x$-derivative. The advantage of this regularization is
 that our new weighted $H^s$ and $L^\infty$ a priori estimates also hold for $\wep$, and it is
 the main reason why we can derive the uniform (in $\e$) estimates in section \ref{s:RPE}.
 The price that we pay is the appearance of extra terms $\frac{\dx \wep}{\wep},
 \frac{\dxx \wep}{\wep}, \frac{\dx\dy \wep}{\wep}$ and $\frac{\dyy \wep}{\wep}$ during the
 estimation, but these terms can be controlled in the function space $C([0,T]; \Hs)$. Before
 going to the next level, we should also emphasize that replacing the
 Bernoulli's law \eqref{e:Ber} by the regularized Bernoulli's law \eqref{e:RBer}
 is crucial here, otherwise the conditions $u|_{y=0} = 0$ and $\lim_{y \to +\infty} u = U$
 cannot be satisfied simultaneously. Although the approximate system \eqref{e:Rvort} -
 \eqref{e:Ruv} seems to be nice, its existence in the function space $\Hs$ is not obvious,
 so we will further approximate it by the next approximate system.

The second level of approximation is the truncated and regularized vorticity system:
for any $\e > 0$ and $R \ge 1$,
%
%  eqn 4.5
\begin{equation} \label{e:TRPE}
\left\{\begin{aligned}
    \dt \w_R + \chi_R \{u_R \dx \w_R + v_R \dy \w_R \} & = \e^2 \dxx \w_R +
    \dyy \w_R && \text{ in } [0,T] \times \TR\\
    \w_R |_{t=0} & = \w_0 := \dy u_0 && \text{ on } \TR\\
    \dy \w_R |_{y=0} & = \dx \pep && \text{ on } [0,T] \times \T\\
\end{aligned}\right.
\end{equation}
where the velocity field $(u_R, v_R)$ is given by
%
%
%  eqn 4.6
\begin{equation} \label{e:TRuv}
%\left\{\begin{aligned}
    u_R (t,x,y) := U - \int^{+\infty}_y \w_R (t,x,\tilde{y}) \, d\tilde{y} \quad \text{ and }
    \quad v_R (t,x,y) := - \int^y_0 \dx u_R (t,x,\tilde{y})\, d\tilde{y}.
%\end{aligned}\right.
\end{equation}
Here, $\pep$ and $U$ still satisfy the regularized Bernoulli's law \eqref{e:RBer}. The
cutoff function $\chi_R$ is defined by $\chi_R (y) := \chi (\frac{y}{R})$ where
$\chi \in C^\infty_c ([0, +\infty))$ satisfies the following properties:
%
%
%  eqn 4.7
\begin{equation} \label{e:chi properties}
%\left\{\begin{aligned}
     0 \le \chi \le 1, \qquad \chi |_{[0,1]} \equiv 1, \qquad
    supp\; \chi \subseteq [0,2] \quad \text{ and } \quad
     -2 \le \chi' \le 0.
%\end{aligned} \right.
\end{equation}

The main disadvantage of approximate system \eqref{e:TRPE} - \eqref{e:TRuv} is that the truncation
on the convection term $u_R \dx \w_R + v_R \dy \w_R$ destroys the boundary condition
$u_R|_{y=0} = 0$ as well as our weighted $H^s$ a priori estimate. However, it still
keeps the weighted $L^\infty$ controls.

To compensate for the lack of our new weighted $H^s$ estimates, one may apply the
standard $H^s$ energy estimates because the system \eqref{e:TRPE} - \eqref{e:TRuv}
does not have the problem of $x$-derivative loss. These estimates depend on $\e$, but
not on $R$. Thus, passing to the limit $R \to + \infty$ for the solution of \eqref{e:TRPE} -
\eqref{e:TRuv} to that of \eqref{e:Rvort} - \eqref{e:Ruv} should be no doubt.
The reason of doing this approximation is to prepare for our next approximate system.

The third level of approximation is the linearized, truncated and regularized
vorticity system: for any $\e > 0, R \ge 1$ and $n \in \N$,
%
%
%  eqn 4.8
\begin{equation} \label{e:LTRPE}
\left\{\begin{aligned}
    \dt \w^{n+1} + \chi_R \{u^n \dx \w^n +v^n \dy \w^n \} & = \e^2 \dxx \w^{n+1}
    + \dyy \w^{n+1} && \text{ in } [0,T] \times \TR\\
    \w^{n+1} |_{t=0} & = \w_0 := \dy u_0 && \text{ on } \TR\\
    \dy \w^{n+1} |_{y=0} & = \dx \pep && \text{ on } [0,T] \times \T
\end{aligned}\right.
\end{equation}
where the velocity field $(u^n, v^n)$ is given by
%
%
% eqn 4.9
\begin{equation} \label{e:LTRuv}
%\left\{\begin{aligned}
    u^n (t,x,y) := U - \int^{+\infty}_y \w^n (t,x,\tilde{y}) \, d\tilde{y} \quad
    \text{ and } \quad v^n (t,x,y) := - \int^y_0 \dx u^n (t,x,\tilde{y}) \, d\tilde{y}.
%\end{aligned}\right.
\end{equation}
In other words, \eqref{e:LTRPE} - \eqref{e:LTRuv} is a linearization of
\eqref{e:TRPE} - \eqref{e:TRuv}.

The main advantage of the iterative scheme  \eqref{e:LTRPE} - \eqref{e:LTRuv}  is that
its explicit solution formula can be obtained by the method of reflection.
Using the explicit solution formula and the fact that $\chi_R
\{u^n \partial_x u^n + v^n \partial_y u^n\}$ has compact
support, one may prove that there exists a uniform (in $n$) life-span $T >0$ for the
approximate sequence $\{\w^n\}_{n \in \N}$ provided that
$\omega_0 \in H^{s,\gamma}_{\sigma, 2\delta}$. This gives us a starting
point so that we can solve the approximate systems and derive estimates.

Solving the above approximate systems in a reverse order and deriving appropriate
estimates, we can prove the existence to the
Prandtl equations \eqref{e:PE}. Detailed
analysis for solving the regularization \eqref{e:RPE} as well as other approximate
systems \eqref{e:Rvort} - \eqref{e:Ruv}, \eqref{e:TRPE} - \eqref{e:TRuv} and
\eqref{e:LTRPE} - \eqref{e:LTRuv} will be given in section \ref{s:exist RPE}. Assuming that
$(\uep, \vep, \wep)$ solves \eqref{e:RPE} - \eqref{e:Ruv}, we will derive uniform
(in $\e$) weighted estimates in section \ref{s:RPE}. Based on these uniform estimates,
we will complete the proof of our main
theorem \ref{t:Hs exist PS} in section \ref{s:pf main thm}.

%=====================================================================
%
%  5. Uniform Estimates on the Regularized Prandtl Equations
%
%===================================================================
\section{Uniform Estimates on the Regularized Prandtl Equations}\label{s:RPE}
In this section and the next we are going to complete the proof of our main theorem
\ref{t:Hs exist PS} provided that we have a solution of the regularized Prandtl equations
\eqref{e:RPE}. In this section we will derive uniform estimates for the regularized
Prandtl equations  \eqref{e:RPE} by using the new weighted energy \eqref{e:Hgsgamma}
introduced in section \ref{s:outline}. These estimates are the main novelty of
this paper. Then we will finish the proof of theorem \ref{t:Hs exist PS} in
section \ref{s:pf main thm}. After that, an outline for solving the
regularized Prandtl equations \eqref{e:RPE} will be provided in section \ref{s:exist RPE}.

Our starting point is that we can solve the velocity field $(\uep, \vep)$ from
the regularized Prandtl equations \eqref{e:RPE}. More precisely, let us assume proposition
\ref{p:loc sol RPE}, which will be shown in section \ref{s:exist RPE}, below for the moment.
%
% prop 5.1
\begin{prop}[Local Existence of the Regularized Prandtl Equations]\label{p:loc sol RPE}
Let $s\ge 4$ be an even integer, $\gamma \ge 1, \sigma > \gamma + \frac{1}{2},
\delta \in (0, \frac{1}{2})$ and $\e \in (0,1]$. If
$\w_0 \in H^{s+12,\gamma}_{\sigma, 2\delta}, U$ and $\pep$ are given and satisfy the
regularized Bernoulli's law \eqref{e:RBer} and the regularity assumption \eqref{e:supU},
then there exist a time
$T := T(s,\gamma,\sigma,\delta, \e, \|\w_0\|_{H^{s+4,\gamma}},U) > 0$ and a solution
$\wep \in C([0,T]; H^{s+4,\gamma}_{\sigma,\delta}) \cap C^1 ([0,T]; H^{s+2,\gamma})$
to the regularized vorticity system \eqref{e:Rvort}-\eqref{e:Ruv}.

Furthermore, the velocity $(\uep,\vep)$ defined by \eqref{e:Ruv} satisfies the regularized
Prandtl equations \eqref{e:RPE} as well.
\end{prop}
%
% rmk 5.2
\begin{rem}[Initial Data]
The $H^{s,\gamma}_{\sigma,2\delta}$ functions can be approximated by
$H^{s+12,\gamma}_{\sigma,2\delta}$ functions in the norm
$\|\cdot\|_{H^{s,\gamma}}$, so by the standard density argument, the hypothesis $\w_0 \in
H^{s+12,\gamma}_{\sigma, 2\delta}$ can be reduced to be $\w_0 \in
H^{s,\gamma}_{\sigma, 2\delta}$ in our final result.
% We will leave these technicalities to the interested reader.
\end{rem}

According to proposition \ref{p:loc sol RPE}, the life-span
$T_{s,\gamma,\sigma,\delta, \e, \w_0, U}$ of $\wep$ depends on $\e$,
so our aim in  this section is to remove the $\e$-dependence by deriving uniform (in $\e$)
estimates on $\wep$. In other words, we will prove the following
%
%  prop 5.3
\begin{prop}[Uniform Estimates on the Regularized Prandtl Equations]\label{p:Uniform RPE}
Let $s\ge4$ be an even integer, $\gamma \ge 1, \sigma > \gamma + \frac{1}{2}, \delta \in
(0,1)$ and $\e \in [0,1]$. If $\wep \in C([0,T]; H^{s+4,\gamma}_{\sigma, \delta}) \cap C^1 ([0,T];
H^{s+2,\gamma}_{\sigma,\delta})$ and $(\uep, \vep, \wep)$ solves \eqref{e:RPE} - \eqref{e:Ruv},
 then we have the following uniform (in $\e$) estimates:
\begin{itemize}
\item[(i)] (Weighted $H^s$ Estimates)
%
%   eqn 5.1
    \begin{equation} \label{e:Hsgamma}
    \begin{split}
        & \; \|\wep(t)\|_{\Hg} \\
         \le & \; \left\{\|\w_0\|_{\Hg}^2 + \int^t_0 F(\tau)\,d\tau \right\}^{\frac{1}{2}}
        \left\{1 - \constC \left( \|\w_0\|_{\Hg}^2 + \int^t_0 F(\tau)\, d\tau \right)^{\frac{s-2}{2}} t
        \right\}^{-\frac{1}{s-2}}
    \end{split}
    \end{equation}
    as long as the second braces on the right hand side of \ref{e:Hsgamma} is positive, where
    the positive constant $\constC$ depends on $s$, $\gamma$, $\sigma$, $\delta$ only and
    $F :[0,T] \to \R^+$ is defined by
%
%   eqn 5.2
    \begin{equation} \label{e:F eqn}
        F := \constC \{1+\|\dx^{s+1} U\|_{L^\infty}^4\} + C_{s} \sum^{\frac{s}{2}}_{l=0}
        \|\dt^l \dx \pep\|^2_{H^{s-2l} (\T)}.
    \end{equation}
\item[(ii)] (Weighted $L^\infty$ Estimates)
Define $I(t) := \sum_{|\alpha| \le 2} |(1+y)^{\sigma + \alpha_2} \Dawe (t)|^2$. For
any $s \ge 4$,
%
%
%  eqn 5.3
\begin{equation} \label{e:I s4}
    \|I(t)\|_{L^\infty (\TR)} \le \max \{\|I(0)\|_{L^\infty (\TR)}, 6C^2 \Omega (t)^2\}
    e^{\constC \{1+G(t)\}t}
\end{equation}
where the universal constant $C$ is the same as the one in inequality \eqref{e:Sobolev}, $\Omega$
and $G :[0,T] \to \R^+$ are defined by
%
%
% eqn 5.4
\begin{equation} \label{e:G eqn}
    \Omega (t) := \sup_{[0,t]}\|\wep\|_{\Hg} \qquad \text{ and } \qquad
    G(t) := \sup_{[0,t]} \|\wep\|_{\Hg} + \sup_{[0,t]} \|\dx^s U\|_{L^2 (\T)}.
\end{equation}
In addition, if $s \ge 6$, then we also have
%
%
% eqn 5.5
\begin{equation} \label{e:weighted Linfty s6}
    \|I(t)\|_{L^\infty (\TR)}
    \le (\|I(0)\|_{L^\infty (\TR)} + C_{s, \gamma} \{1+\Omega(t)\} \Omega(t)^2 t)
    e^{\constC \{1+G(t)\}t}.
\end{equation}
For $s \ge 4$, we have the following lower bound estimate:
%
%
% eqn 5.6
 \begin{equation} \label{e:weighted Linfty s4}
 \begin{split}
    & \; \min_{\TR} (1+y)^\sigma \wep (t) \\
    \ge & \; \left(1-\constC \{1+ G(t)\} t e^{\constC \{1+G(t)\}t} \right) \cdot
    \left( \min_{\TR} (1+y)^\sigma \w_0 -C_{s, \gamma} \Omega(t) t \right)
 \end{split}
\end{equation}
provided that $\min_{\TR} (1+y)^\sigma \w_0 - C_{s, \gamma} \Omega(t) t \ge 0$, where $\constC$
is a positive constant depending on $s, \gamma, \sigma$ and $\delta$ only.
\end{itemize}
\end{prop}
%
%
%  remark 5.4
\begin{rem}[Two $L^\infty$ Estimates on $I$]
In proposition \ref{p:Uniform RPE}, we stated two $L^\infty$ controls on the quantity $I(t)$, namely, estimates \eqref{e:I s4} and \eqref{e:weighted Linfty s6}. Indeed, \eqref{e:weighted Linfty s6} is a better estimate within a short time, but it only holds for $s\geq 6$. Thanks to this better estimate, we can derive the uniform weighted $L^\infty$ bound \eqref{e:weighted Linfty bound} without any additional assumption when $s\geq 6$. In contrast, we are required to impose an extra initial hypothesis \eqref{e:mainthm con} for the case $s=4$ since we only have the weaker estimate \eqref{e:I s4} in this case. See proposition \ref{p:lifespan} for the details.
\end{rem}
%
%  remark 5.5
\begin{rem}[A Priori Estimates on the Prandtl Equations]
When $\e = 0$, proposition \ref{p:Uniform RPE} provides a priori estimates for the
Prandtl equations \eqref{e:PE}. Similar situation occurs in proposition
\ref{p:lifespan} as well.
\end{rem}

The proof of proposition \ref{p:Uniform RPE} will be given in the subsections \ref{ss:weighted est}
and \ref{ss:Linfty lower} as follows.

%------------------------------------------------------------------------
%
%   5.1   Weighted Energy Estimates
%
%------------------------------------------------------------------------
\subsection{Weighted Energy Estimates}\label{ss:weighted est}
The objective of this subsection is to derive uniform (in $\e$) weighted $H^s$ estimates on
$\wep$. These estimates, which are the main novelty of this paper, include: (i) weighted $L^2$
estimates on $\Dawe$ for $|\alpha| \le s$ and $\alpha_1 \le s-1$ in subsubsection
\ref{sss:Dwep}, and (ii) weighted $L^2$ estimate on $g_s$ in subsubsection \ref{sss:geps}.
We will combine these two estimates in subsubsection \ref{sss:Hs wep} to obtain
the uniform weighted energy estimates \eqref{e:Hsgamma}. This will complete the proof of part (i) of
proposition \ref{p:Uniform RPE}.

%----------------------------------------------------------------------
%
%   5.1.1   Weighted L^2 Estimates on $\Dalpha \wep$
%
%-----------------------------------------------------------------------
\subsubsection{Weighted $L^2$ Estimates on $\Dawe$}\label{sss:Dwep}
Using the standard energy method, we will derive weighted $L^2$ estimates on $\Dawe$
for $|\alpha| \le s$ and $\alpha_1 \le s-1$ in this subsubsection. It works because we
are allowed to loss at least one $x$-regularity in these cases.

More specifically, we will prove
%
%   prop  5.5
\begin{prop}[$L^2$ Controls on $(1+y)^{\gamma+\alpha_2} \Dawe$ for $|\alpha| \le s$
and $\alpha_1 \le s-1$] \label{p:L2 Dwep}
Under the hypotheses of proposition \ref{p:Uniform RPE}, we have the following estimates:
%
%  eqn 5.7
\begin{equation}\label{e:L2 Dwep}
\begin{split}
    & \; \frac{1}{2} \dfrac{d}{dt} \sum_{\substack{|\alpha| \le s \\
    \alpha_1 \le s-1}} \|(1+y)^{\gamma+\alpha_2} \Dawe \|^2_{L^2}\\
    \le & \; -\e^2 \sum_{\substack{|\alpha| \le s \\
    \alpha_1 \le s-1}} \|(1+y)^{\gamma+\alpha_2} \dx \Dawe\|^2_{L^2}
    - \frac{1}{2} \sum_{\substack{|\alpha| \le s \\
    \alpha_1 \le s-1}} \|(1+y)^{\gamma+\alpha_2} \dy \Dawe\|^2_{L^2}\\
    & \; + \constC \{\|\wep\|_{\Hg} + \|\dx^s U\|_{L^\infty} \} \|\wep\|^2_{\Hg} \\
    & \; + \constC \{1+\|\wep\|_{\Hg} \}^{s-2} \|\wep\|^2_{\Hg}
    + C_s \sum^{\frac{s}{2}}_{l=0} \|\dt^l \dx \pep\|^2_{H^{s-2l}(\T)},
\end{split}
\end{equation}
where the positive constants $C_s$ and $\constC$ are independent of $\e$.
\end{prop}
%
% rmk 5.6
\begin{rem}[Boundary Terms at $y=+\infty$] \label{r:bdry infty}
In the proof of proposition \ref{p:L2 Dwep} and that of proposition \ref{p:L2 gs} below, we
will ignore the boundary terms at $y = + \infty$ while we are integrating by parts in the
 variable $y$. Skipping these boundary terms is just for the presentation convenience, and
ignoring these technicalities is harmless. Indeed, one may deal with these boundary terms
by any one of the following two methods:
\begin{itemize}
\item[(i)] Since $\wep \in \Hsfour$, by proposition \ref{p:decay Hs+3}, we have nice
pointwise decays \eqref{e:decay Hs+3} for $\wep$ and its spatial derivatives.
Therefore, when $\sigma$ is much larger than $\gamma$, one may easily check that
those terms which we will omit actually vanish;
\item[(ii)] As long as $\wep(t) \in \Hs$, the norm $\|\wep\|_{\Hg} <+\infty$ provides
certain integrability of the underlying quantities. Thus, one may overcome the technical
difficulty by first multiplying by a  nice cutoff function $\chi_R (y) := \chi (\frac{y}{R})$
during the estimation, and then passing to the limit $R$ $\to$ $+\infty$. The main advantage of
this approach is that it only requires $\sigma > \gamma + \frac{1}{2}$ and $\wep(t)$ in
$\Hs$, but not in $\Hsfour$. As a demonstration, we will apply this argument in the
proof of proposition \ref{p:L2 compare} for the reader's convenience.
\end{itemize}
In conclusion, the proofs of proposition \ref{p:L2 Dwep} and \ref{p:L2 gs} are absolutely
correct, even if we ignore the boundary terms at $y = + \infty$.
\end{rem}
%
%  proof of prop 5.5
\begin{proof}[Proof of proposition \ref{p:L2 Dwep}]
Differentiating the vorticity equation $\eqref{e:Rvort}_1$ with respect to $x\; \alpha_1$
times and $y\; \alpha_2$ times, we obtain the evolution equation for $\Dawe$:
%
% eqn 5.8
\begin{equation}\label{e:DRvort alpha}
%\begin{split}
     \; \{\dt + \uep\dx + \vep \dy - \e^2 \dxx  - \dyy \} \Dawe
    =  \; - \sum_{0<\beta \le \alpha} \binom{\alpha}{\beta}
    \{D^\beta \uep \dx D^{\alpha-\beta} \wep + D^\beta \vep \dy D^{\alpha-\beta}\wep\}.
%\end{split}
\end{equation}
Multiplying \eqref{e:DRvort alpha} by $(1+y)^{2\gamma+2\alpha_2} \Dawe$, and then
integrating over $\TR$, we have
%
% eqn 5.9
\begin{equation}\label{e:MI DRvort alpha}
\begin{split}
    & \; \frac{1}{2} \dfrac{d}{dt} \|(1+y)^{\gamma+\alpha_2} \Dawe\|^2_{L^2} \\
    = & \; \e^2 \iint (1+y)^{2\gamma+2\alpha_2} D^\alpha \wep \dxx \Dawe
    + \iint (1+y)^{2\gamma+2\alpha_2} D^\alpha \wep \dyy \Dawe\\
    & \; - \iint (1+y)^{2\gamma+2\alpha_2} \Dawe
    \{\uep \dx \Dawe + \vep \dy \Dawe\} \\
    & \; - \sum_{0 < \beta \le \alpha} \binom{\alpha}{\beta}
    \iint (1+y)^{2\gamma+2\alpha_2} \Dawe
    \{D^\beta \uep \dx D^{\alpha-\beta} \wep + D^\beta \vep \dy D^{\alpha-\beta} \wep\}.
\end{split}
\end{equation}

Now that we can apply integration by parts and the standard Sobolev's type estimates
on trilinear forms to control the right hand side of \eqref{e:MI DRvort alpha} as follows.
%
%  claim 5.7
\begin{claim} \label{claim:rhs MI DRvort alpha}
There exist constants $C_{s,\gamma}$ and $\constC > 0$ such that for any
$|\alpha| \le s$ and $\alpha_1 \le s-1$,
%
%  eqn 5.10
\begin{equation} \label{e:rhs MDa1}
    \e^2 \iint (1+y)^{2\gamma+2\alpha_2} \Dawe \dxx \Dawe
    = - \e^2 \| (1+y)^{\gamma+\alpha_2} \dx \Dawe \|^2_{L^2}.
\end{equation}
%
%  eqn 5.11
\begin{equation} \label{e:rhs MDa2}
\begin{split}
    & \; \iint (1+y)^{2\gamma+2\alpha_2} \Dawe \dyy \Dawe\\
    & \quad \quad \quad   \le  \; - \frac{3}{4} \|(1+y)^{\gamma+\alpha_2} \dy \Dawe\|^2_{L^2}
    - \int_{\T} \Dawe \dy \Dawe \; dx \big|_{y=0} + C_{s,\gamma} \|\wep\|^2_{\Hg}.
\end{split}
\end{equation}
%
% eqn 5.12
\begin{equation} \label{e:rhs MDa3}
\begin{split}
    & \; \left|\iint (1+y)^{2\gamma+2\alpha_2} \Dawe \{\uep\dx \Dawe
    +\vep \dy \Dawe\}\right|\\
    & \quad \quad \quad   \le  \; \constC \{\|\wep\|_{\Hg} + \|\dx^s U\|_{L^2}\} \|\wep\|^2_{\Hg}
\end{split}
\end{equation}
%
%  eqn 5.13
\begin{equation} \label{e:rhs MDa4}
\begin{split}
    &\; \left| \sum_{0<\beta \le \alpha} \binom{\alpha}{\beta}
    \iint (1+y)^{2\gamma+2\alpha_2} \Dawe \{D^\beta \uep \dx
    D^{\alpha-\beta} \wep + D^\beta \vep \dy D^{\alpha-\beta} \wep \} \right|\\
    & \quad \quad \quad   \le  \; \constC \{\|\wep\|_{\Hg} + \|\dx^s U\|_{L^\infty} \}
    \|\wep\|_{\Hg}^2.
\end{split}
\end{equation}

\end{claim}

Assuming claim \ref{claim:rhs MI DRvort alpha}, which will be shown later in this
section, for the moment, we can apply inequalities \eqref{e:rhs MDa1} - \eqref{e:rhs MDa4}
to the equality \eqref{e:MI DRvort alpha}, and obtain
%
% eqn  5.14
\begin{equation}\label{e:L2 dt Dwep}
\begin{split}
    & \; \frac{1}{2} \dfrac{d}{dt} \|(1+y)^{\gamma+\alpha_2} \Dawe\|^2_{L^2}\\
    \le & \; - \e^2 \|(1+y)^{\gamma+\alpha_2} \dx \Dawe\|_{L^2}
    - \frac{3}{4} \|(1+y)^{\gamma+\alpha_2} \dy \Dawe\|^2_{L^2}\\
    & \; - \int_{\T} \Dawe \dy \Dawe \; dx \big|_{y=0}
    +C_{s,\gamma} \|\wep\|^2_{\Hg} + \constC \{\|\wep\|_{\Hg} + \|\dx^s U \|_{L^\infty}\}
     \|\wep\|^2_{\Hg}.
\end{split}
\end{equation}

When $|\alpha| \le s-1$, we can apply the simple trace estimate
%
%  eqn 5.15
\begin{equation} \label{e:trace est}
    \int_{\T} |f| \; dx \Big|_{y=0} \le C \left\{\int^1_0 \int_{\T} |f| \;dxdy
    + \int^1_0 \int_{\T} |\dy f| \; dxdy \right\}
\end{equation}
to control the boundary integral $\int_{\T} \Dawe\, \dy \Dawe\, dx \big|_{y=0}$ as
follows:
%
%
%  eqn 5.16
\begin{equation} \label{e:bdry int Dawe}
    \left| \int_\T \Dawe \dy \Dawe\, dx |_{y=0} \right|
    \le \frac{1}{12} \|(1+y)^{\gamma + \alpha_2 +1} \dyy \Dawe\|^2_{L^2}
    + C \|\wep\|^2_{\Hg}.
\end{equation}
However, when $|\alpha| =s$, a main difficulty arises: the order of
$\dy \Dawe|_{y=0}$ is too high so that we cannot control the boundary integral
$\int_{\T} \Dawe \dy \Dawe\; dx|_{y=0}$ by the simple trace estimate
\eqref{e:trace est}. In order to make use of \eqref{e:trace est}, we must reduce the
order of the problematic term $\dy \Dawe|_{y=0}$. When $s$ is even, this
can be done by a boundary reduction argument as follows.

At this moment, let us state without proof the following boundary reduction lemma,
which will be proven at the end of this subsubsection.
%
%  lemma 5.8
\begin{lem}[Reduction of Boundary Data] \label{l:reduce bdry}
Under the hypotheses of proposition \ref{p:Uniform RPE}, we have at the boundary
$y = 0$,
%
%
%  eqn 5.17
\begin{equation}\label{e:reduce bdry 13}
\left\{\begin{aligned}
    \dy\wep|_{y=0} & = \dx \pep\\
    \dy^3 \wep|_{y=0} & = (\dt - \e^2 \dxx) \dx \pep + \wep \dx \wep|_{y=0}.
\end{aligned}\right.
\end{equation}
For any $2 \le k \le [\frac{s}{2}]$, there are some constants
$C_{k,l,\rho^1,\rho^2, \cdots, \rho^j}$'s,
which do not depend on $\e$ or $(\uep,\vep,\wep)$, such that
%
%  eqn 5.18
\begin{equation}\label{e:reduce bdry}
    \dy^{2k+1} \wep|_{y=0} = (\dt-\e^2\dxx)^k \dx\pep
    + \sum^{k-1}_{l=0} \e^{2l} \sum^{\max\{2,k-l\}}_{j=2}\sum_{\rho \in A^j_{k,l}}
    C_{k,l,\rho^1,\rho^2, \cdots, \rho^j}  \prod^j_{i=1} D^{\rho^i} \wep|_{y=0}
\end{equation}
where $A^j_{k,l} := \{\rho:=(\rho^1,\rho^2,\cdots,\rho^j)\in\N^{2j};\,
3\sum^j_{i=1} \rho^i_1 + \sum^j_{i=1} \rho^i_2 = 2k+4l+1$, $\sum^j_{i=1} \rho^i_1 \le k + 2l -1$, $\sum^j_{i=1} \rho^i_2 \le 2k-2l-2$ and
$|\rho^i| \le 2k-l-1$ for all $i=1,2,...,j\}$.

\end{lem}

Now, we can apply lemma \ref{l:reduce bdry} to control the boundary integral
$\int_{\T} \Dawe \dy \Dawe \; dx |_{y=0}$ for $|\alpha|=s$ with
$0 \le \alpha_1 \le s-1$ in the following two cases:\\
\hfill\\
\underline{Case I:} ($\alpha_2$ is even)

When $\alpha_2 := 2k$ for some $k \in \N$, we can apply boundary reduction lemma
\ref{l:reduce bdry} to $\dy \Dawe |_{y=0}$, and obtain
%
%   eqn 5.19
\begin{equation} \label{e:reduce I}
\begin{split}
    & \; \int_{\T} \Dawe\dy \Dawe\, dx|_{y=0}\\
    = &\; \int \Dawe (\dt - \e^2 \dxx)^k \dx^{\alpha_1 +1} \pep\, dx |_{y=0}\\
    & \;+ \sum^{k-1}_{l=0} \e^{2l} \sum^{\max\{2,k-l\}}_{j=2} \sum_{\rho \in A^j_{k,l}}
    C_{k,l,\rho^1,\rho^2, \cdots, \rho^j} \int \Dawe \dx^{\alpha_1}
    \left( \prod^j_{i=1} D^{\rho^i} \wep \right) \, dx |_{y=0}.
\end{split}
\end{equation}
According to the definition of $A^j_{k,l}$, one may check by using the indices restrictions that the largest
possible order for $\dx^{\alpha_1} D^{\rho^i} \wep$ is $\le s-1$ and at most one of
$\dx^{\alpha_1} D^{\rho^i} \wep$ can attain the order $s-1$, namely, the orders of
other terms are $\le s-2$. Therefore, we can apply the simple trace estimate \eqref{e:trace est}
and proposition \ref{p:L2Linfty uvwgk} to the identity \eqref{e:reduce I}
to obtain
%
%  eqn 5.20
\begin{equation}\label{e:reduce even}
\begin{split}
    & \; \left| \int_{\T} \Dawe \dy \Dawe\;dx \Big|_{y=0} \right|\\
    \le & \; \frac{1}{12} \| (1+y)^{\gamma+\alpha_2} \dy \Dawe \|^2_{L^2}
    + C_s \sum^{\frac{s}{2}}_{l=0} \|\dt^l \dx \pep\|^2_{H^{s-2l} (\T)}
    + \constC \{1+\|\wep\|_{\Hg} \}^{s-2} \|\wep\|^2_{\Hg}.
\end{split}
\end{equation}
\hfill\\
\underline{Case II:} ($\alpha_2$ is odd)

When $\alpha_2 := 2k+1$ for some $k \in \N$, since $\alpha_1 + \alpha_2 = s$ is
assumed to be even, we know that $\alpha_1 \ge 1$. Using integration by parts in
$x$, we have
%
%  eqn 5.21
\begin{equation}\label{e:reduce II}
    \int_{\T} \Dawe \dy \Dawe \; dx \Big|_{y=0}
    = - \int_{\T} \dx \Dawe \dx^{\alpha_1-1} \dy^{\alpha_2+1}
    \wep \; dx \Big|_{y=0}.
\end{equation}
Now, the term $\dx \Dawe|_{y=0} = \dx^{\alpha_1+1} \dy^{2k+1} \wep|_{y=0}$
has an odd number of $y$ derivatives, and hence, we can apply the boundary reduction
lemma \ref{l:reduce bdry} to reduce the order of the right hand side of
\eqref{e:reduce II}. Similar to the Case I, we can further apply the simple trace
estimates \eqref{e:trace est} and proposition \ref{p:L2Linfty uvwgk} to eventually
obtain the following estimates:
%
% eqn 5.22
\begin{equation}\label{e:reduce odd}
\begin{split}
     \left|\int_{\T} \Dawe \dy \Dawe \, dx \Big|_{y=0} \right|
    & \le \; \frac{1}{12} \|(1+y)^{\gamma+\alpha_2+1} \dx^{\alpha_1-1}
    \dy^{\alpha_2+2} \wep \|^2_{L^2} + C_s \sum^{\frac{s}{2}}_{l=0}
    \| \dt^l \dx \pep\|^2_{H^{s-2l}(\T)} \\
    &  \qquad + \constC \{1+\|\wep\|_{\Hg} \}^{s-2} \|\wep\|^2_{\Hg}.
\end{split}
\end{equation}

Finally, combining estimates \eqref{e:L2 dt Dwep}, \eqref{e:bdry int Dawe},
\eqref{e:reduce even} and \eqref{e:reduce odd} and summing over $\alpha$,
 we prove \eqref{e:L2 Dwep}.

\end{proof}

In order to complete the proof of proposition \ref{p:L2 Dwep}, it remains to
show claim \ref{claim:rhs MI DRvort alpha} and the boundary reduction lemma
\ref{l:reduce bdry}. Let us first prove the claim \ref{claim:rhs MI DRvort alpha}
as follows.
%
%
%   proof of claim 5.7
\begin{proof}[Proof of claim \ref{claim:rhs MI DRvort alpha}]
\hfill\\
\underline{Proof of \eqref{e:rhs MDa1}:}

The equality \eqref{e:rhs MDa1} follows immediately from an integration by parts in
the variable $x$.\\
\hfill\\
\underline{Proof of \eqref{e:rhs MDa2}:}

Integrating by parts in $y$ (cf. remark \ref{r:bdry infty}), we have
\begin{align*}
    & \; \iint (1+y)^{2\gamma+2\alpha_2} \Dawe \dyy \Dawe\\
    = & \; - \| (1+y)^{\gamma+\alpha_2} \dy \Dawe\|^2_{L^2}
    - \int_{\T} \Dawe \dy \Dawe\,dx|_{y=0}\\
    & \; - 2 (\gamma + \alpha_2) \iint (1+y)^{2\gamma+2\alpha_2-1} \Dawe \dy \Dawe\\
    \le & \; - \frac{3}{4} \|(1+y)^{\gamma+\alpha_2} \dy \Dawe\|^2_{L^2}
    - \int_{\T} \Dawe \dy \Dawe\, dx |_{y=0} + C_{s,\gamma} \|\wep\|^2_{\Hg},
\end{align*}
which is inequality \eqref{e:rhs MDa2}.\\
\hfill\\
\underline{Proof of \eqref{e:rhs MDa3}:}

Integrating by parts (cf. remark \ref{r:bdry infty}), and using
$\dx \uep + \dy \vep = 0$, we have
%
%  eqn 5.23
\begin{equation}  \label{e:pf rhs MDa3}
    \iint (1+y)^{2\gamma+2\alpha_2} \Dawe\{\uep\dx\Dawe+\vep\dy\Dawe\}
    = (\gamma + \alpha_2) \iint (1+y)^{2\gamma+2\alpha_2-1} \vep |\Dawe|^2.
\end{equation}
which and inequality \eqref{e:B3 Linfty v} imply inequality \eqref{e:rhs MDa3}.\\
\hfill\\
\underline{Proof of \eqref{e:rhs MDa4}:}

Using the facts that $\dy\vep = - \dx\uep$ and $\dy\uep = \wep$, one may check
that all terms on the left hand side of \eqref{e:rhs MDa4} are one of the
following three types: denoting $e_1 := (1,0)$ and $e_2 := (0,1)$,
for $\eta \in \N$ and $\kappa, \theta\in \N^2$,\\
\hfill\\
\textbf{Type I:}
\[
    J_1 := \iint (1+y)^{2\gamma+2\alpha_2} \Dawe \dx^\eta \vep D^\kappa \wep
\]
where $1 \le \eta \le s-1$ and $\eta e_1 + \kappa = \alpha + e_2$,\\
\hfill\\
\textbf{Type II:}
\[
    J_2 := \iint (1+y)^{2\gamma+2\alpha_2} \Dawe \dx^\eta \uep D^\kappa \wep
\]
where $1 \le \eta \le s$ and $\eta e_1 + \kappa = \alpha + e_1$,\\
\hfill\\
\textbf{Type III:}
\[
    J_3 := \iint (1+y)^{2\gamma+2\alpha_2} \Dawe D^\theta \wep D^\kappa \wep
\]
where $|\theta| \le s-1$ and $\theta + \kappa = \alpha + e_1 - e_2$.

Thus, it suffices to control $J_1, J_2$ and $J_3$ by the right hand side of
\eqref{e:rhs MDa4} as follows.\\
\hfill\\
\textbf{Estimates for Type I:}\\
When $1 \le \eta \le s-2$, applying proposition \ref{p:L2Linfty uvwgk}, we have,
since $\kappa_2 = \alpha_2 +1$,
\begin{align*}
    |J_1| & \le \; \|(1+y)^{\gamma+\alpha_2} \Dawe\|_{L^2} \,
    \left\| \frac{\dx^\eta \vep}{1+y} \right\|_{L^\infty}\,
    \|(1+y)^{\gamma+\kappa_2} D^\kappa \wep\|_{L^2}\\
    & \le  \; \constC \{\|\wep\|_{\Hg} + \|\dx^s U \|_{L^2}\} \|\wep\|^2_{H^{s,\gamma}}.
\end{align*}
When $\eta = s-1$, by triangle inequality and proposition \ref{p:L2Linfty uvwgk},
we have, since $\kappa_2 = \alpha_2 +1$,
\begin{align*}
    |J_1| & \le \; \|(1+y)^{\gamma+\alpha_2} \Dawe\|_{L^2} \,
    \left\| \frac{\dx^{s-1} \vep + y \dx^s U}{1+y} \right\|_{L^2}\,
    \|(1+y)^{\gamma+\kappa_2} D^\kappa \wep\|_{L^\infty}\\
    & \qquad + \|(1+y)^{\gamma+\alpha_2} \Dawe\|_{L^2} \,
    \|\dx^s U\|_{L^\infty} \, \|(1+y)^{\gamma+\kappa_2} D^\kappa \wep\|_{L^2}\\
    & \le  \; \constC \{\|\wep\|_{\Hg} + \|\dx^s U \|_{L^\infty}\} \|\wep\|^2_{\Hg}.
\end{align*}
In conclusion, $J_1$ can be controlled by the right hand side of
\eqref{e:rhs MDa4}.\\
\hfill\\
\textbf{Estimates for Type II:}\\
When $1 \le \eta \le s-1$, applying proposition \ref{p:L2Linfty uvwgk}, we have,
since $\kappa_2 = \alpha_2$,
\begin{align*}
    |J_2| & \le \; \|(1+y)^{\gamma+\alpha_2} \Dawe\|_{L^2}  \|\dx^\eta \uep\|_{L^\infty}
    \| (1+y)^{\gamma+\kappa_2} D^\kappa \wep\|_{L^2} \\
    & \le \; \constC \{ \|\wep\|_{\Hg} + \|\dx^s U\|_{L^2} \} \|\wep\|^2_{\Hg}.
\end{align*}
When $\eta = s$, by triangle inequality and proposition \ref{p:L2Linfty uvwgk}, we have,
since $\kappa = (0, \alpha_2)$ and $0 \le \alpha_2 \le 1$,
\begin{align*}
    |J_2| & \le \; \|(1+y)^{\gamma+\alpha_2} \Dawe\|_{L^2}  \|\dx^s (u-U) \|_{L^2}
    \|(1+y)^{\gamma+\kappa_2} D^\kappa \wep \|_{L^\infty} \\
    & \qquad + \| (1+y)^{\gamma + \alpha_2} \Dawe \|_{L^2} \|\dx^s U\|_{L^\infty}
    \| (1+y)^{\gamma+\kappa_2}  D^\kappa \wep\|_{L^2} \\
    & \le \; \constC \{ \|\wep\|_{\Hg} + \|\dx^s U\|_{L^\infty} \} \|\wep\|^2_{\Hg}.
\end{align*}
In conclusion, $J_2$ can also be controlled by the right hand
side of \eqref{e:rhs MDa4}.\\
\hfill\\
\textbf{Estimates for Type III:}\\
Applying proposition \ref{p:L2Linfty uvwgk}, we have, since $\theta_2 + \kappa_2
= \alpha_2 -1$,
\begin{align*}
    |J_3| & \le \;
    \begin{cases}
        \|(1+y)^{\gamma+\alpha_2} \Dawe\|_{L^2} \|(1+y)^{1+\theta_2} D^\theta \wep\|_{L^\infty}
        \|(1+y)^{\gamma+\kappa_2} D^\kappa \wep\|_{L^2} & \text{ if } |\theta| \le s-2\\
        \|(1+y)^{\gamma+\alpha_2} \Dawe\|_{L^2} \|(1+y)^{1+\theta_2} D^\theta \wep\|_{L^2}
        \|(1+y)^{\gamma+\kappa_2} D^\kappa \wep\|_{L^\infty} & \text{ if } |\theta| = s-1
    \end{cases}\\
    & \le \; \constC \{\|\wep\|_{\Hg} + \|\dx^s U\|_{L^2}\} \|\wep\|^2_{\Hg}.
\end{align*}
Combining all estimates for type I - III, we prove inequality \eqref{e:rhs MDa4}.

\end{proof}

Lastly, we will prove the boundary reduction lemma \ref{l:reduce bdry} as follows.
%
% proof of lemma 5.9
\begin{proof}[Proof of lemma \ref{l:reduce bdry}]
First of all, let us mention that equality $\eqref{e:reduce bdry 13}_1$ is exactly the
same as the given boundary condition $\eqref{e:Rvort}_3$. Furthermore, differentiating
the vorticity $\eqref{e:Rvort}_1$ with respect to $y$, and then evaluating at $y=0$,
we obtain equality $\eqref{e:reduce bdry 13}_2$ by using $\eqref{e:reduce bdry 13}_1$ and
$\uep|_{y=0}=\vep|_{y=0}\equiv 0$. Thus, it remains to prove the formula
\eqref{e:reduce bdry}.

In order to illustrate the idea, let us derive the formula \eqref{e:reduce bdry} for
the case $k=2$ as follows.

Differentiating the vorticity equation $\eqref{e:Rvort}_1$ with respect to $y$ thrice,
and then evaluating at $y=0$, we obtain, by using $\eqref{e:reduce bdry 13}_2$
and $\uep|_{y=0}=\vep|_{y=0} \equiv 0$,
%
%  eqn 5.24
\begin{equation} \label{e:bdry5}
    \dy^5 \wep|_{y=0} = (\dt-\e^2\dxx)^2 \dx\pep+ (\dt-\e^2\dxx)(\wep\dx\wep)
    + 3 \wep\dx\dyy\wep + 2 \dy\wep\dx\dy\wep - 2 \dx\wep\dyy\wep|_{y=0}.
\end{equation}
Since the last three terms on the right hand side are our desired forms, we only
need to deal with the terms $(\dt-\e^2\dxx)(\wep\dx\wep)|_{y=0}$. Using the evolution
equations for $\wep$ and $\dx\wep$ as well as $\uep|_{y=0}=\vep|_{y=0} \equiv 0$, one
may check that
%
%  eqn 5.25
\begin{equation} \label{e:rhs bdry5}
    (\dt-\e^2\dxx)(\wep\dx\wep)|_{y=0} = \wep \dx\dyy\wep +\dx\wep\dyy\wep
    -2\e^2\dx\wep\dxx\wep|_{y=0},
\end{equation}
where all terms on the right hand side of \eqref{e:rhs bdry5} are also our desired forms. Substituting \eqref{e:rhs bdry5} into \eqref{e:bdry5},
we justify the formula \eqref{e:reduce bdry} for $k=2$.

Now, using the same algorithm, we are going to prove the formula \eqref{e:reduce bdry}
by induction on $k$.
For the notational convenience, we denote
\[
    \mathcal{A}_k := \{\sum^{k-1}_{l=0} \e^{2l} \sum^{\max\{2,k-l\}}_{j=2} \sum_{\rho \in
    A^j_{k,l}} C_{k,l,\rho^1,\rho^2,\cdots,\rho^j} \prod^j_{i=1} D^{\rho^i} \wep|_{y=0} \}.
\]
Under this notation, we will prove $\dy^{2k+1} \wep|_{y=0} - (\dt-\e^2\dxx)^k \dx\pep \in
\mathcal{A}_k$.

Assuming that the formula \eqref{e:reduce bdry} holds for $k=n$, we will show that
it also holds for $k=n+1$ as follows.

In order to reduce the order of $\dy^{2n+3} \wep|_{y=0}$, we first differentiate the
vorticity $\eqref{e:Rvort}_1$ with respect to $y\; 2n+1$ times, and then evaluate
the resulting equation at $y=0$ to obtain
%
% eqn 5.26
\begin{equation} \label{e:bdry2n+3}
\begin{split}
    \dy^{2n+3}\wep|_{y=0} = & \; (\dt-\e^2\dxx) \dy^{2n+1}\wep +
    \sum^{2n+1}_{m=1} \binom{2n+1}{m} \dy^{m-1} \wep \dx\dy^{2n-m+1} \wep \\
    & \; - \sum^{2n+1}_{m=2} \binom{2n+1}{m} \dx\dy^{m-2} \wep \dy^{2n-m+2} \wep|_{y=0}.
\end{split}
\end{equation}
By a routine checking, one may show that the last two terms of \eqref{e:bdry2n+3}
belong to $\mathcal{A}_{n+1}$, so it remains to deal with $(\dt-\e^2\dxx) \dy^{2n+1}
\wep|_{y=0}$ only.

Thanks to the induction hypothesis, there exist constants
$C_{n,l,\rho^1,\rho^2,\cdots,\rho^j}$'s such that
\[
    \dy^{2n+1} \wep|_{y=0} = (\dt-\e^2 \dxx)^n \dx\pep
    + \sum^{n-1}_{l=0} \e^{2l} \sum^{\max\{2,n-l\}}_{j=2} \sum_{\rho \in A^j_{n,l}}
    C_{n,l,\rho^1,\rho^2,\cdots,\rho^j} \prod^j_{i=1} D^{\rho^i} \wep|_{y=0},
\]
so we have, up to a relabeling of the indices $\rho^i $'s,
%
%  eqn 5.27
\begin{equation} \label{e:rhs bdry2n+3}
\begin{split}
    & \; (\dt-\e^2\dxx)\dy^{2n+1} \wep|_{y=0}\\
    = & \; (\dt-\e^2\dxx)^{n+1} \dx\pep + \sum^{n-1}_{l=0}\e^{2l}
    \sum^{\max\{2,n-l\}}_{j=2} \sum_{\rho \in A^j_{n,l}}
    \tilde{C}_{n,l,\rho^1,\rho^2,\cdots,\rho^j} (\dt-\e^2\dxx) D^{\rho^1} \wep
    \prod^j_{i=2} D^{\rho^i} \wep\\
    & \; - \sum^{n-1}_{l=0} \e^{2l+2} \sum^{\max\{2,n-l\}}_{j=2} \sum_{\rho \in A^j_{n,l}}
    \tilde{\tilde{C}}_{n,l,\rho^1,\rho^2,\cdots,\rho^j} \dx D^{\rho^1} \wep
    \dx D^{\rho^2} \wep \prod^j_{i=3} D^{\rho^i}\wep |_{y=0}
\end{split}
\end{equation}
where $\tilde{C}_{n,l,\rho^1,\rho^2,\cdots,\rho^j}$'s and
$\tilde{\tilde{C}}_{n,l,\rho^1,\rho^2,\cdots,\rho^j}$'s are some new constants depending
on $C_{n,l,\rho^1,\rho^2,\cdots,\rho^j}$'s. It is worth to note that the last term on the
right hand side of \eqref{e:rhs bdry2n+3} belongs to $\mathcal{A}_{n+1}$, so it remains
to check whether the second term on right hand side of \eqref{e:rhs bdry2n+3} also belongs
to $\mathcal{A}_{n+1}$.

Differentiating the vorticity equation $\eqref{e:Rvort}_1$ with respect to
$x\; \rho^1_1$  times and $y\;\rho^1_2$ times, and then evaluating at $y=0$,
we have, by using $\uep|_{y=0} = \vep|_{y=0} \equiv 0$ and denoting $e_2 := (0,1)$,
%
%  eqn 5.28
\begin{equation} \label{e:rhs bdry2n+3 2}
\begin{split}
    & \; (\dt-\e^2\dxx)D^{\rho^1} \wep|_{y=0}\\
    = & \; - \sum_{\substack{\beta \le \rho^1 \\ \beta_2 \ge 1}} \binom{\rho^1}{\beta}
    D^{\beta-e_2} \wep \dx D^{\rho^1-\beta}\wep
    + \sum_{\substack{\beta \le \rho^1 \\ \beta_2 \ge 2}} \binom{\rho^1}{\beta}
    \dx D^{\beta-2e_2} \wep \dy D^{\rho^1-\beta}\wep + \dyy D^{\rho^1}\wep|_{y=0}.
\end{split}
\end{equation}
Using \eqref{e:rhs bdry2n+3 2}, one may justify by a routine
counting of indices that the second term on the right hand side of
\eqref{e:rhs bdry2n+3} belongs to $\mathcal{A}_{n+1}$. This completes the proof
of lemma \ref{l:reduce bdry}.

\end{proof}

%----------------------------------------------------------------
%
%    5.1.2  Weighted L^2 Estimates on g_s^\e
%
%------------------------------------------------------------------
\subsubsection{Weighted $L^2$ Estimate on $g_s^\e$} \label{sss:geps}
In this subsubsection we will derive the $L^2$ estimate on $(1+y)^\gamma \gep_s$ by the
standard energy method. This can be done because the quantity $\gep_s := \dx^s \wep
- \frac{\dy\wep}{\wep} \dx^s (\uep-U)$ avoids the loss of $x$-derivative by a
nonlinear cancellation, which is one of the key observations in this paper and will
be explained as follows.

Let us begin by writing down the evolution equations for $\wep$ and $\uep-U$:
%
%  eqn 5.29
\begin{equation} \label{e:evlt eqn}
\left\{\begin{aligned}
    (\dt + \uep \dx + \vep \dy) \wep & = \e^2 \dxx\wep + \dyy \wep\\
    (\dt + \uep \dx + \vep \dy) (\uep-U)
    & = \e^2 \dxx (\uep-U) + \dyy (\uep-U) - (\uep-U) \dx U
\end{aligned}\right.
\end{equation}
where we applied the regularized Bernoulli's law \eqref{e:RBer} in the derivation of $\eqref{e:evlt eqn}_2$.
Since our aim is to control the $H^s$ norm of $\wep$ (or $\uep-U$), let
us differentiate \eqref{e:evlt eqn} with respect to $x \; s $ times. Then we have
%
%  eqn 5.30
\begin{equation}\label{e:Dxs evlt}
\left\{\begin{aligned}
     (\dt + \uep \dx + \vep \dy) \dx^s \wep + \dx^s \vep \dy \wep
      & =  \; \e^2 \dx^{s+2} \wep + \dx^s  \dyy \wep + \cdots\\
     (\dt + \uep \dx + \vep \dy) \dx^s (\uep-U) +
     \dx^s \vep \,\wep & = \; \e^2 \dx^{s+2} (\uep - U) + \dx^s \dyy (\uep-U) + \cdots,
\end{aligned}\right.
\end{equation}
where we applied the fact that $\wep = \dy (\uep-U)$ and the symbol $\cdots$ represents
the lower order terms which we want the reader to ignore at this moment.

The main obstacle in \eqref{e:Dxs evlt} is the term $\dx^s \vep = - \dx^{s+1} \dy^{-1} \uep$,
which has $s+1$ $x$-derivatives so that standard energy estimates cannot be closed.
However, since there are two equations in \eqref{e:Dxs evlt}, we can eliminate the problematic
term $\dx^s \vep$ by subtracting them in an appropriate way.

Subtracting $\frac{\dy\wep}{\wep} \times \eqref{e:Dxs evlt}_2$ from $\eqref{e:Dxs evlt}_1$,
we obtain
%
%  eqn 5.31
\begin{equation}\label{e:gs evlt}
\begin{split}
    &\; (\dt + \uep \dx + \vep \dy - \e^2 \dxx-\dyy) \gep_s \\
    = &\; 2\e^2 \{\dx^{s+1} (\uep-U) - \frac{\dx\wep}{\wep} \dx^s (\uep-U)\} \dx\aep
    + 2 \gep_s \dy\aep - \gep_1 \dx^s U - \sum^{s-1}_{j=1} \binom{s}{j}
    \gep_{j+1} \dx^{s-j} \uep\\
    &\;- \sum^{s-1}_{j=1} \binom{s}{j} \dx^{s-j} \vep \{\dx^j \dy \wep - \aep \dx^j \wep\}
    + \aep \sum^{s-1}_{j=0} \binom{s}{j} \dx^j (\uep-U) \dx^{s-j+1} U.
\end{split}
\end{equation}
where $\gep_k := \dx^k \wep - \aep \dx^k (\uep-U)$ and $\aep := \frac{\dy \wep}{\wep}$.
Here, the main reason that we can apply this nonlinear cancelation is the Oleinik's
monotonicity assumption (i.e., $\omega > 0$), which is ensured in our solution class
$\Hs$. For the justification of \eqref{e:gs evlt}, see appendix \ref{s:appendixD}.

Now, we are going to derive the following weighted energy estimate on $\gep_s$:
%
%  prop 5.10
\begin{prop}[$L^2$ Control on $(1+y)^\gamma \gep_s$] \label{p:L2 gs}
Under the hypotheses of proposition \ref{p:Uniform RPE}, we have the following estimate:
%
%  eqn 5.32
\begin{equation}\label{e:L2 gs}
\begin{split}
    & \; \frac{1}{2}\dfrac{d}{dt} \|(1+y)^\gamma \gep_s\|^2_{L^2}\\
    \le & \; - \frac{1}{2} \e^2 \|(1+y)^\gamma \dx \gep_s\|^2_{L^2}
    - \frac{1}{2} \|(1+y)^\gamma \dy \gep_s\|^2_{L^2}
    + C \|\dx^{s+1} \pep\|^2_{L^2(\T)} + C_{\gamma, \delta} \|\dx^s U\|^2_{L^\infty (\T)}
    \|\wep\|^2_{\Hg}\\
    & \; + \constC \{1+\|\wep\|_{\Hg} + \|\dx^s U\|_{L^\infty (\T)}\}\,
    \{\|\wep\|_{\Hg} + \|\dx^{s+1} U\|_{L^\infty (\T)} \}\, \|\wep\|_{\Hg},
\end{split}
\end{equation}
where the positive constants $C, C_{\gamma, \delta}$ and $\constC$ are independent of $\e$.

\end{prop}

The proof of proposition \ref{p:L2 gs} is almost a straight forward application of
 energy methods except the estimation \eqref{e:claim gs 4} below is slightly tricky.
%
%   proof of prop 5.10
\begin{proof}[Proof of proposition \ref{p:L2 gs}]
Multiplying the evolution equation \eqref{e:gs evlt} by $(1+y)^{2\gamma} \gep_s$, and
then integrating over $\TR$, we have
%
%  eqn 5.33
\begin{equation}\label{e:pf L2 gs 1}
\begin{split}
    & \; \frac{1}{2}\dfrac{d}{dt} \|(1+y)^\gamma \gep_s\|^2_{L^2}\\
    = & \; \e^2 \iint (1+y)^{2\gamma} \gep_s \dxx \gep_s
    + \iint (1+y)^{2\gamma} \gep_s \dyy \gep_s
    - \iint (1+y)^{2\gamma} \gep_s \{\uep \dx \gep_s + \vep \dy \gep_s \}\\
    & \; + 2 \e^2 \iint (1+y)^{2\gamma} \gep_s  \{ \dx^{s+1} (\uep-U)
    - \frac{\dx\wep}{\wep} \dx^s (\uep - U) \}\dx \aep\\
    & \; + 2 \iint (1+y)^{2\gamma} |\gep_s|^2 \dy \aep
    - \iint (1+y)^{2\gamma} \gep_1 \gep_s \dx^s U\\
    & \; - \sum^{s-1}_{j=1} \binom{s}{j} \iint (1+y)^{2\gamma}
    \gep_{j+1} \gep_s \dx^{s-j} \uep  - \sum^{s-1}_{j=1} \binom{s}{j}
    \iint (1+y)^{2\gamma} \dx^{s-j} \vep \{ \dx^j \dy\wep - \aep \dx^j \wep \}  \gep_s \\
    & \; + \sum^{s-1}_{j=0} \binom{s}{j} \iint (1+y)^{2\gamma} \gep_s \aep
    \dx^j (\uep-U) \dx^{s-j+1} U.
\end{split}
\end{equation}

Indeed, all terms on the right hand side of \eqref{e:pf L2 gs 1} can be controlled
by using integration by parts and the standard Sobolev's type estimate on multilinear
forms. Precisely, we have
%
%   Claim 5.11
\begin{claim}\label{claim:rhs pf L2 gs 1}
There exist constants $C, C_\delta, C_{\gamma, \delta}$ and $\constC > 0$ such that
%
%   eqn 5.34
\begin{equation} \label{e:claim gs 1}
    \e^2 \iint (1+y)^{2\gamma} \gep_s \dxx \gep_s
    = - \e^2 \|(1+y)^\gamma \dx \gep_s \|^2_{L^2},
\end{equation}
%
%  eqn 5.35
\begin{equation} \label{e:claim gs 2}
    \iint (1+y)^{2\gamma} \gep_s \dyy \gep_s
    \le - \frac{1}{2} \|(1+y)^{\gamma} \dy \gep_s\|^2_{L^2}
    + C_{\gamma, \delta} \{1+\|\dx^s U\|^2_{L^\infty(\T)}\} \|\wep\|^2_{\Hg}
    + C \|\dx^{s+1} \pep\|^2_{L^2(\T)},
\end{equation}
%
%  eqn 5.36
\begin{equation} \label{e:claim gs 3}
    \left| \iint (1+y)^{2\gamma} \gep_s \{\uep \dx \gep_s + \vep \dy \gep_s\}\right|
    \le \constC \{\|\wep\|_{\Hg} + \|\dx^s U\|_{L^\infty}\} \|\wep\|^2_{\Hg},
\end{equation}
%
%  eqn 5.37
\begin{equation} \label{e:claim gs 4}
\begin{split}
    & \; \left| 2\e^2 \iint (1+y)^{2\gamma} \gep_s
    \{\dx^{s+1} (\uep-U) - \frac{\dx\wep}{\wep} \dx^s (\uep-U)\}\dx\aep \right| \\
    &\quad \quad  \le \;  \frac{1}{2} \e^2 \|(1+y)^{\gamma} \dx\gep_s \|^2_{L^2}
    + \e^2 \constC \{\|\wep\|_{\Hg} + \|\dx^{s+1}U\|_{L^2(\T)}\}\|\wep\|_{\Hg},
\end{split}
\end{equation}
%
%  eqn 5.38
\begin{equation} \label{e:claim gs 5}
    \left| 2 \iint (1+y)^{2\gamma} |\gep_s|^2 \dy \aep\right|
    \le C_\delta \|\wep\|^2_{\Hg},
\end{equation}
%
%  eqn 5.39
\begin{equation} \label{e:claim gs 6}
    \left| \iint (1+y)^{2\gamma} \gep_1 \gep_s \dx^s U\right|
    \le \constC \|\dx^s U\|_{L^\infty(\T)}
    \{\|\wep\|_{\Hg} + \|\dx^s U\|_{L^2 (\T)} \}\|\wep\|_{\Hg},
\end{equation}
%
%  eqn 5.40
\begin{equation} \label{e:claim gs 7}
    \left| \sum^{s-1}_{j=1} \binom{s}{j} \iint (1+y)^{2\gamma}
    \gep_{j+1} \gep_s \dx^{s-j} \uep\right|
    \le \constC \{\|\wep\|_{\Hg} + \|\dx^s U\|_{L^2 (\T)} \}^2 \|\wep\|_{\Hg},
\end{equation}
%
%  eqn 5.41
\begin{equation} \label{e:claim gs 8}
\begin{split}
    & \; \left| \sum^{s-1}_{j=1} \binom{s}{j} \iint (1+y)^{2\gamma} \dx^{s-j} \vep
    \{\dx^j \dy \wep - \aep \dx^j \wep\} \gep_s \right| \\
   &\quad \quad  \le  \; \constC \{\|\wep\|_{\Hg} + \|\dx^s U\|_{L^\infty(\T)}\} \|\wep\|^2_{\Hg},
\end{split}
\end{equation}
%
%  eqn 5.42
\begin{equation} \label{e:claim gs 9}
\begin{split}
    & \; \left| \sum^{s-1}_{j=0} \binom{s}{j} \iint (1+y)^{2\gamma} \gep_s \aep
    \dx^j (\uep-U) \dx^{s-j+1} U \right| \\
    &\quad \quad    \le  \; \constC \|\dx^{s+1} U\|_{L^\infty(\T)}
    \{\|\wep\|_{\Hg} + \|\dx^s U\|_{L^2 (\T)} \} \|\wep\|_{\Hg}.
\end{split}
\end{equation}

\end{claim}

Assuming claim \ref{claim:rhs pf L2 gs 1}, which will be proven at the end of this
subsection, for the moment, we can apply \eqref{e:claim gs 1} - \eqref{e:claim gs 9}
to \eqref{e:pf L2 gs 1} to obtain our desired inequality \eqref{e:L2 gs} because $\e \in [0,1]$ and
$\|\dx^s U\|_{L^2(\T)} \le \|\dx^s U\|_{L^\infty (\T)} \le \|\dx^{s+1}U\|_{L^2 (\T)}
\le \|\dx^{s+1} U\|_{L^\infty (\T)}$.

\end{proof}

To complete the proof of proposition \ref{p:L2 gs}, we will show claim
\ref{claim:rhs pf L2 gs 1} as follows.
%
%
%  proof of claim 5.11
\begin{proof}[Proof of Claim \ref{claim:rhs pf L2 gs 1}]
\hfill\\
\underline{Proof of \eqref{e:claim gs 1}:}

The equality \eqref{e:claim gs 1} follows directly from an integration by parts in
the variable $x$.\\
\hfill\\
\underline{Proof of \eqref{e:claim gs 2}:}

Integrating by parts in the variable $y$ (cf. remark \ref{r:bdry infty}),
we have
%
%  eqn 5.43
\begin{align}
    &\; \iint (1+y)^{2\gamma} \gep_s \dyy \gep_s \notag\\
    = & \; - \|(1+y)^\gamma \dy \gep_s\|^2_{L^2} - \int_{\T} \gep_s \dy \gep_s \,
    dx |_{y=0} - 2\gamma \iint (1+y)^{2\gamma-1} \gep_s \dy \gep_s\notag\\
    \le & \; - \frac{3}{4} \|(1+y)^\gamma \dy \gep_s\|^2_{L^2}
    - \int_{\T} \gep_s \dy \gep_s \, dx |_{y=0} + C_{\gamma}\|\wep\|^2_{\Hg}.
    \label{e:pf claim gs 2}
\end{align}

In order to deal with the boundary integral $\int_{\T} \gep_s \dy \gep_s |_{y=0}$, one may apply the
boundary conditions $\eqref{e:RPE}_4$ and $\eqref{e:Rvort}_3$ to justify that
\[
    \dy \gep_s |_{y=0} = \dx^{s+1} \pep + \frac{\dyy \wep}{\wep} \dx^s U - \aep\gep_s
    |_{y=0}.
\]
This boundary condition allows us to reduce the order of $\dy \gep_s$, and hence,
using the simple trace estimate \eqref{e:trace est} and the facts that $\wep|_{y=0}
\ge \delta$ and $\|\aep\|_{L^\infty} \le \delta^{-2}$, one may prove that
%
%   eqn 5.44
\begin{equation} \label{e:pf claim gs 2 2}
\begin{split}
    & \; \left| \int_{\T} \gep_s \dy \gep_s \, dx |_{y=0} \right|\\
    \le & \; \frac{1}{4} \|(1+y)^\gamma \dy \gep_s\|^2_{L^2}
    + C_\delta \{1+\|\dx^s U\|^2_{L^\infty} \} \|\wep\|^2_{\Hg}
    + C \|\dx^{s+1} \pep\|^2_{L^2 (\T)}.
\end{split}
\end{equation}
Substituting \eqref{e:pf claim gs 2 2} into \eqref{e:pf claim gs 2}, we obtain
\eqref{e:claim gs 2}.\\
\hfill\\
\underline{Proof of \eqref{e:claim gs 3}:}

Integrating by parts (cf. remark \ref{r:bdry infty}), and using $\dx \uep + \dy \vep = 0$, we have
%
%  eqn 5.45
\begin{equation} \label{e:pf claim gs 3}
    \iint (1+y)^{2\gamma} \gep_s \{\uep \dx \gep_s + \vep \dy \gep_s\}
    = \gamma \iint (1+y)^{2\gamma -1} \vep |\gep_s|^2.
\end{equation}
which and inequality \eqref{e:B3 Linfty v} imply inequality \eqref{e:claim gs 3}.\\
\hfill\\
\underline{Proof of \eqref{e:claim gs 4}:}

Since $\wep \in C([0,T]; H^{s+4,\gamma}_{\sigma, \delta})$, it follows from the definition of $H^{s+4,\gamma}_{\sigma, \delta}$ that
$(1+y)^\sigma \wep \ge \delta$ and $|(1+y)^{\sigma + \alpha_2} \Dalpha \wep| \le \delta^{-1}$ for all $| \alpha | \le 2$. Thus, we have
$\|(1+y)\dx \aep\|_{L^\infty} \le \delta^{-2} + \delta^{-4}$ and
$\left\| \frac{\dx\wep}{\wep}\right\|_{L^\infty} \le \delta^{-2}$, and hence,
%
%  eqn  5.46
\begin{equation}\label{e:pf claim gs 4}
\begin{split}
    & \; \left| 2\e^2 \iint (1+y)^{2\gamma} \gep_s \{ \dx^{s+1} (\uep-U)
    - \frac{\dx\wep}{\wep} \dx^s (\uep -U) \}\dx \aep \right| \\
    \le & \; 2\e^2 C_\delta \|(1+y)^\gamma \gep_s \|_{L^2}
    \{\|(1+y)^{\gamma-1} \dx^{s+1} (\uep-U)\|_{L^2}
    + \| (1+y)^{\gamma-1} \dx^s (\uep - U) \|_{L^2} \}.
\end{split}
\end{equation}

Now, we require the following inequality:
%
%  eqn 5.47
\begin{equation} \label{e:pf claim gs 4 2}
    \|(1+y)^{\gamma-1} \dx^{s+1} (\uep-U)\|_{L^2}
    \le C_{\gamma,\sigma,\delta} \{\|\dx^{s+1} U\|_{L^2 (\T)} + \|(1+y)^\gamma \dx \gep_s
    \|_{L^2} + \|(1+y)^{\gamma-1} \dx^s (\uep-U)\|_{L^2} \}.
\end{equation}
Assuming \eqref{e:pf claim gs 4 2} for the moment, we can apply it and proposition
\ref{p:L2Linfty uvwgk} to \eqref{e:pf claim gs 4}, and obtain
\[\begin{split}
    & \; \left| 2 \e^2 \iint (1+y)^{2\gamma} \gep_s
    \{\dx^{s+1} (\uep -U) - \frac{\dx \wep}{\wep} \dx^s (\uep-U)\}\dx \aep \right| \\
    \le & \; \e^2 \constC \{ \|\wep\|_{\Hg} + \|\dx^s U\|_{L^2 (\T)}
    + \|\dx^{s+1} U\|_{L^2 (\T)} + \|(1+y)^\gamma \dx \gep_s\|_{L^2} \} \|\wep\|_{\Hg}
\end{split}\]
which implies \eqref{e:claim gs 4} by Cauchy's inequality and the inequality
$ \|\dx^s U\|_{L^2 (\T)} \le \frac{1}{2\pi} \|\dx^{s+1} U\|_{L^2 (\T)}$.

To complete the justification of \eqref{e:claim gs 4}, we have to verify
\eqref{e:pf claim gs 4 2} as follows.

Since $\delta \le (1+y)^\sigma \wep \le \delta^{-1}$ and $\uep |_{y=0} = 0$, applying part (ii) of lemma \ref{l:Hardy},
we have
%
%   eqn 5.48
\begin{align}
    \|(1+y)^{\gamma-1} \dx^{s+1} (\uep-U)\|_{L^2}
    & \le \; \delta^{-1} \left\|(1+y)^{\gamma-\sigma-1} \frac{\dx^{s+1} (\uep-U)}{\wep}
    \right\|_{L^2} \notag\\
    & \le \; C_{\gamma, \sigma, \delta} \left\{\|\dx^{s+1} U\|_{L^2 (\T)}
    + \left\| (1+y)^\gamma \wep \dy \left(\frac{\dx^{s+1} (\uep-U)}{\wep}\right)
    \right\|_{L^2} \right\}. \label{e:pf claim gs 4 3}
\end{align}
It is worth to note that
\[
    \wep \dy \left(\frac{\dx^{s+1} (\uep-U)}{\wep}\right) = \gep_{s+1}
    = \dx \gep_s + \dx\aep \dx^s (\uep - U),
\]
so by \eqref{e:pf claim gs 4 3}, we have
\[\begin{split}
    & \; \|(1+y)^{\gamma-1} \dx^{s+1} (\uep-U) \|_{L^2}\\
    \le & \; C_{\gamma, \sigma, \delta} \{\|\dx^{s+1} U\|_{L^2 (\T)}
    + \|(1+y)^\gamma \dx \gep_s\|_{L^2} + \|(1+y)\dx \aep\|_{L^\infty}
    \| (1+y)^{\gamma-1} \dx^s (\uep - U)\|_{L^2} \}
\end{split}\]
which implies inequality \eqref{e:pf claim gs 4 2} because
$\|(1+y) \dx\aep \|_{L^\infty} \le \delta^{-2} + \delta^{-4}$.\\
\hfill\\
\underline{Proof of \eqref{e:claim gs 5}:}

The inequality \eqref{e:claim gs 5} follows from
$\|\dy \aep\|_{L^\infty} \le \delta^{-2} + \delta^{-4}$ and the definition of
$\|\wep\|_{\Hg}$.\\
\hfill\\
\underline{Proof of \eqref{e:claim gs 6} and \eqref{e:claim gs 7}:}

Both inequalities \eqref{e:claim gs 6} and \eqref{e:claim gs 7} follows from the
H\"{o}lder inequality and proposition \ref{p:L2Linfty uvwgk}.\\
\hfill\\
\underline{Proof of \eqref{e:claim gs 8}:}

For $j=2,3,\cdots, s-1$, using proposition \ref{p:L2Linfty uvwgk} and
$\|(1+y)\aep\|_{L^\infty} \le \delta^{-2}$, we have
%
%  eqn 5.49
\begin{align}
    & \; \left| \iint (1+y)^{2\gamma} \dx^{s-j} \vep \{\dx^j\dy \wep
    - \aep \dx^j \wep\} \gep_s \right| \notag\\
    \le & \; \left\| \frac{\dx^{s-j} \vep}{1+y} \right\|_{L^\infty}
    \{ \|(1+y)^{\gamma+1} \dx^j \dy \wep\|_{L^2}
    + \| (1+y) \aep \|_{L^\infty} \| (1+y)^\gamma \dx^j \wep\|_{L^2} \}
    \|(1+y)^\gamma \gep_s\|_{L^2} \notag\\
    \le & \; \constC \{\|\wep\|_{\Hg} + \|\dx^s U\|_{L^2 (\T)} \}
    \|\wep\|^2_{\Hg}. \label{e:pf claim gs 8}
\end{align}
When $j=1$, using proposition \ref{p:L2Linfty uvwgk} and
$\|(1+y)\aep\|_{L^\infty} \le \delta^{-2}$ again, we have
%
%  eqn 5.50
\begin{align}
    & \; \left| \iint (1+y)^{2\gamma} \dx^{s-1} \vep \{\dx\dy \wep
    - \aep \dx \wep\} \gep_s \right| \notag\\
    \le & \; \Big\{ \left\| \frac{\dx^{s-1} \vep + y \dx^s U}{1+y} \right\|_{L^2}
    \left( \|(1+y)^{\gamma+1} \dx\dy \wep \|_{L^\infty}
     + \|(1+y) \aep \|_{L^\infty} \| (1+y)^\gamma \dx \wep\|_{L^\infty} \right)
     \notag\\
     & \;  + \|\dx^s U\|_{L^\infty (\T)} \left( \|(1+y)^{\gamma+1} \dx\dy\wep\|_{L^2}
    + \|(1+y)\aep\|_{L^\infty} \|(1+y)^\gamma \dx \wep\|_{L^2} \right) \Big\}
    \|(1+y)^\gamma \gep_s\|_{L^2} \notag\\
    \le & \; \constC \{\|\wep\|_{\Hg} + \|\dx^s U\|_{L^\infty} \}
    \|\wep\|^2_{\Hg}. \label{e:pf claim gs 8 2}
\end{align}
Combining estimates \eqref{e:pf claim gs 8} and \eqref{e:pf claim gs 8 2}, we
prove inequality \eqref{e:claim gs 8}.\\
\hfill\\
\underline{Proof of \eqref{e:claim gs 9}:}

For any $j=0,1,\cdots,s-1$, by proposition \ref{p:L2Linfty uvwgk} and
$\|(1+y)\aep\|_{L^\infty} \le \delta^{-2}$, we have
\begin{align*}
    & \; \left| \iint (1+y)^{2\gamma} \gep_s \aep \dx^j (\uep-U) \dx^{s-j+1}U\right|\\
    \le & \; \|(1+y)^\gamma \gep_s\|_{L^2} \|(1+y)\aep\|_{L^\infty} \|(1+y)^{\gamma-1}
    \dx^j (\uep-U)\|_{L^2} \|\dx^{s-j+1} U\|_{L^\infty (\T)}\\
    \le & \; \constC \|\dx^{s+1} U\|_{L^\infty(\T)} \{\|\wep\|_{\Hg} + \|\dx^sU\|_{L^2(\T)}\}
    \|\wep\|_{\Hg},
\end{align*}
which implies inequality \eqref{e:claim gs 9}.

\end{proof}

%---------------------------------------------------------------------------------\
%
%   5.1.3   Weighted H^s Estimates on \wep
%
%-----------------------------------------------------------------------------------
\subsubsection{Weighted $H^s$ Estimate on $\wep$}\label{sss:Hs wep}
The aim of this subsubsection is to combine the estimates in proposition \ref{p:L2 Dwep}
and proposition \ref{p:L2 gs} to derive the growth rate control \eqref{e:Hsgamma} on
the weighted $H^s$ energy of $\wep$.

According to propositions \ref{p:L2 Dwep} and \ref{p:L2 gs}, we know from the
definition of $\|\cdot\|_{\Hg}$ that
\begin{align*}
    & \dfrac{d}{dt} \|\wep\|^2_{\Hg}\\
    \le & \;\constC \{1+\|\wep\|_{\Hg} + \|\dx^s U\|_{L^\infty(\T)} \}
    \{\|\wep\|_{\Hg} + \|\dx^{s+1} U\|_{L^\infty (\T)} \} \|\wep\|_{\Hg} \\
    & \; + C_{\gamma, \delta} \|\dx^s U\|^2_{L^\infty (\T)} \|\wep\|^2_{\Hg}
    + \constC \{1+ \|\wep\|_{\Hg} \}^{s-2} \|\wep\|^2_{\Hg}
    + C_{s} \sum^{\frac{s}{2}}_{l=0} \|\dt^l \dx \pep\|^2_{H^{s-2l}(\T)}\\
    \le & \; \constC \|\wep\|^s_{\Hg} + \constC \{1+ \|\dx^{s+1} U\|^4_{L^\infty (\T)} \}
    + C_{s} \sum^{\frac{s}{2}}_{l=0} \|\dt^l \dx \pep\|^2_{H^{s-2l}(\T)},
\end{align*}
and hence, it follows from the comparison principle of ordinary differential equations that
\[\begin{split}
    & \; \|\wep(t)\|^2_{\Hg} \\
    \le & \; \left\{\|\w_0\|^2_{\Hg} + \int^t_0 F(\tau)\,d\tau\right\}
    \left\{1 - (\frac{s}{2}-1) \constC \left(\|\w_0\|^2_{\Hg} + \int^t_0 F(\tau)\,d\tau
    \right)^{\frac{s-2}{2}} t \right\}^{-\frac{2}{s-2}}
\end{split}\]
as long as the second braces on the right hand side of the above inequality is positive,
where $F : [0,T] \longrightarrow \R^+$ is defined by \eqref{e:F eqn}.
This proves inequality \eqref{e:Hsgamma}.

%---------------------------------------------------------------------------\
%
%         5.2 Weighted L^\infty Estimates on Lower Order Terms
%
%-----------------------------------------------------------------------------
%
\subsection{Weighted $L^\infty$ Estimates on Lower Order Terms} \label{ss:Linfty lower}
In this subsection we will derive uniform (in $\e$) weighted $L^\infty$ estimates
on $\Dawe$ for $|\alpha| \le 2$ by using the classical maximum principle. The key
idea is to ``view" the evolution equation of $\Dawe$ as a ``linear" parabolic
equation with coefficients involving higher order terms of $\uep, \vep$ and $\wep$,
which can be controlled by proposition \ref{p:L2Linfty uvwgk} provided that
$\|\wep\|_{\Hg} < + \infty$.

More precisely, we will prove part (ii) of proposition \ref{p:Uniform RPE}
as follows.
%
%
%   proof of part (ii) of prop 5.3
\begin{proof}[Proof of part (ii) of proposition \ref{p:Uniform RPE}]
This proof, based on a simple application of the classical maximum principle for
parabolic equations, will be divided into two steps. In the first step, we will
derive weighted $L^\infty$ controls on $\Dawe$ by using the maximum principle stated in
appendix \ref{s:appendixE}. These controls will rely on the boundary values of
$\Dawe$ at $y=0$, so we will also derive estimates on the boundary values of $\Dawe$
by using Sobolev embedding or growth rate control argument in the
second step.\\
\hfill\\
\underline{Step 1:} (Maximum Principle Argument)

First of all, let us derive a $L^\infty$ estimate on the $I := \sum_{|\alpha| \le 2}
|(1+y)^{\sigma+\alpha_2} \Dawe|^2$ as follows.

For notational convenience, let us denote, for $|\alpha| \le 2$,
\[
    B_\alpha := (1+y)^{\sigma+\alpha_2} \Dawe
\]
which is our concerned quantities. By a direct computation, $B_\alpha$ satisfies
%
%  eqn 5.51
\begin{equation} \label{e:Balpha}
    \{\dt+\uep\dx+\vep\dy-\e^2 \dxx-\dyy\} B_\alpha = Q_\alpha \dy B_\alpha
    + R_\alpha B_\alpha + S_\alpha,
\end{equation}
where the quantities $Q_\alpha$, $R_\alpha$ and $S_\alpha$ are given explicitly by
\[\begin{aligned}
    Q_\alpha & := - \frac{2(\sigma + \alpha_2)}{1+y},  \qquad
    R_\alpha  := \frac{\sigma+\alpha_2}{1+y} \vep + \frac{(\sigma+\alpha_2)
    (\sigma+\alpha_2 +1)}{(1+y)^2},\\
    S_\alpha & :=
    \begin{cases}
        0 & \text{ if } \alpha = (0,0),\\
        - \sum_{0 < \beta \le \alpha} \binom{\alpha}{\beta}
        \{(1+y)^{\beta_2} D^\beta \uep B_{\alpha-\beta+e_1}
        + (1+y)^{\beta_2 -1} D^\beta \vep B_{\alpha-\beta +e_2} \}
        & \text{ if } |\alpha| \ge 1.
    \end{cases}
\end{aligned}\]
Here, $e_1 := (1,0)$ and $e_2 := (0,1)$. Using proposition \ref{p:L2Linfty uvwgk},
we have the following pointwise controls on $Q_\alpha,
R_\alpha$ and $S_\alpha$: for $|\alpha| \le 2$,
%
%  eqn 5.52
\begin{equation} \label{e:QRS}
\left\{\begin{aligned}
    |Q_\alpha| & \le C_\sigma, \qquad
    |R_\alpha|  \le \constC \{1 + \|\w\|_{\Hg} + \|\dx^s U\|_{L^2 (\T)} \},\\
    |S_\alpha| & \le \constC \{ \|\w\|_{\Hg} + \|\dx^s U\|_{L^2}\}
    \sum_{0 < \beta \le \alpha} \{ |B_{\alpha-\beta+e_1}| + |B_{\alpha-\beta+e_2}| \},
\end{aligned}\right.
\end{equation}
where $C_\sigma$ and $\constC$ are some universal constants which are
independent of the solution $\wep$.

Let us recall from the definition that $I := \sum_{|\alpha| \le 2} |B_\alpha|^2$, so using
\eqref{e:Balpha} and \eqref{e:QRS}, we have
\begin{align*}
    & \; \{ \dt  + \uep\dx+ \vep\dy-\e^2 \dxx - \dyy\} I \\
    = & \; - 2 \sum_{|\alpha|\le 2} \{\e^2  |\dx B_\alpha|^2 + |\dy B_\alpha|^2\}
    + 2 \sum_{|\alpha|\le2} \{Q_\alpha B_\alpha \dy B_\alpha + R_\alpha |B_\alpha|^2
    +S_\alpha B_\alpha \} \\
    \le & \; \constC \{1 + \|\wep\|_{H^{s,\gamma}_g} + \|\dx^s U\|_{L^2}\} I.
\end{align*}

Applying the classical maximum principle for parabolic equations (see lemma \ref{l:max parabolic}
for instance) to the quantity $I$, we have, after using the definition \eqref{e:G eqn} of $G$,
%
%  eqn 5.53
\begin{equation}\label{e:Imax}
\begin{split}
    & \; \|I(t)\|_{L^\infty (\TR)} \\
    \le & \; \max \{ e^{\constC \{1+ G(t)\}t } \|I(0)\|_{L^\infty (\TR)},
    \max_{\tau \in [o,t]} \{ e^{\constC \{1+ G(t)\} (t-\tau) } \|I(\tau) |_{y=0}
    \|_{L^\infty (\T)} \} \}.
\end{split}
\end{equation}

Next, we are going to derive a lower bound estimate on $B_{(0,0)} := (1+y)^\sigma \wep$.
To do this, let us recall that $B_{(0,0)}$ satisfies
\[
    \{\dt+\uep\dx+(\vep-Q_{(0,0)}) \dy - \e^2 \dxx - \dyy\} B_{(0,0)} = R_{(0,0)} B_{(0,0)}.
\]
Using the classical maximum principle (see lemma \ref{l:min parabolic} for instance) and
\eqref{e:QRS}, we obtain
%
% eqn 5.54
\begin{equation} \label{e:min sigma wep}
\begin{split}
    &\; \min_{\TR} (1+y)^\sigma \wep(t)\\
    \ge & \; \left(1 - \constC \{1 + G(t)\}t e^{\constC \{1+ G(t)\} t } \right)
    \min \{\min_{\TR} (1+y)^\sigma \w_0, \min_{[0,t]\times \T} \wep|_{y=0} \}.
\end{split}
\end{equation}
\hfill\\
\underline{Step 2:} (Controls on Boundary Values)

According to inequalities \eqref{e:Imax} and \eqref{e:min sigma wep},
we have already controlled the underlying quantities $I$ and
$B_{(0,0)} := (1+y)^\sigma \wep$ by their initial and boundary values. However, their
boundary values are not given in the problem, so we will estimate them in this step.

In order to control $I|_{y=0} := \sum_{|\alpha| \le 2} |\Dawe |_{y=0}|^2$, we will apply
Sobolev embedding argument and growth rate control argument in the cases $s\geq 4$ and $s\geq 6$ respectively.
Combining the boundary estimates on $I$ with \eqref{e:Imax}, we will finally obtain inequalities
\eqref{e:I s4} and \eqref{e:weighted Linfty s6}.

To derive inequality \eqref{e:I s4}, we first apply lemma \ref{l:Sobolev} to obtain
\begin{align*}
    \|I|_{y=0} \|_{L^\infty (\T)}
    & \le \; \sum_{|\alpha| \le 2} \|I\|_{L^\infty (\TR)}
    \le 3 C^2 \sum_{|\alpha| \le 2} \{\|\Dawe\|^2_{L^2} + \|\dx \Dawe\|^2_{L^2}
    + \|\dyy \Dawe\|^2_{L^2} \}\\
    & \le \; 6 C^2 \|\wep\|^2_{\Hg}
\end{align*}
which and inequality \eqref{e:Imax} imply inequality \eqref{e:I s4}.

To derive inequality \eqref{e:weighted Linfty s6}, let us begin by
writing down the evolution equation for $\Dawe$ at $y=0$: for any $|\alpha| \le 2$,
%
%  eqn 5.55
\begin{equation}\label{e:evlt Dawe y=0}
    \dt\Dawe \Big|_{y=0} = (\e^2 \dxx+\dyy) \Dawe + E_\alpha  \Big|_{y=0}
\end{equation}
where the term $E_\alpha$ is given explicitly by
\[
    E_\alpha :=
    \begin{cases}
        0 & \text{ if } \alpha_2 = 0\\
        -\wep \dx \dy^{\alpha_2-1} \wep & \text{ if } \alpha_1 = 0 \text{ and } \alpha_2 \ge 1\\
        - |\dx\wep|^2 - \wep \dxx\wep & \text{ if } \alpha = (1,1).
    \end{cases}
\]
Here, the derivation of \eqref{e:evlt Dawe y=0} is just a direct differentiation on the
vorticity equation $\eqref{e:Rvort}_1$ as well as using the boundary condition $\uep |_{y=0}
= \vep|_{y=0} \equiv 0$. Furthermore, by proposition \ref{p:L2Linfty uvwgk}, if $s\ge 4$,
then
%
%  eqn 5.56
\begin{equation} \label{e:Galpha y=0}
    \|E_\alpha|_{y=0} \|_{L^\infty (\T)} \le C_{s, \gamma} \|\wep\|^2_{\Hg}.
\end{equation}
In addition, if $s \ge |\alpha| + 4$, then by proposition \ref{p:L2Linfty uvwgk} again,
for $\e \in [0,1]$,
%
%  eqn 5.57
\begin{equation} \label{e:Linfty Dawe y=0}
    \|(\e^2 \dxx+\dyy) \Dawe |_{y=0} \|_{L^\infty (\T)} \le C_{s, \gamma} \|\wep\|_{\Hg}.
\end{equation}
Therefore, using \eqref{e:evlt Dawe y=0} - \eqref{e:Linfty Dawe y=0} and inequality
\eqref{e:B3 Linfty w}, we have, for $s \ge 6$,
\[
    \|\dt I|_{y=0}\|_{L^\infty (\T)}
    \le C_{s, \gamma} \{1+\|\wep\|_{\Hg}\} \|\wep\|^2_{\Hg},
\]
and hence, by direction integration and definition \eqref{e:G eqn} of $\Omega$, we obtain
%
%  eqn 5.58
\begin{equation} \label{e:I y=0}
    \| I(t) |_{y=0} \|_{L^\infty (\T)}
    \le \| I(0)|_{y=0} \|_{L^\infty (\T)} + C_{s, \gamma} \{1 + \Omega(t)\} \Omega(t)^2 t.
\end{equation}
Combining estimates \eqref{e:Imax} and \eqref{e:I y=0}, we prove our desired estimate
\eqref{e:weighted Linfty s6}.

Lastly, it remains to show inequality \eqref{e:weighted Linfty s4}. Let us begin by deriving an estimate on $\wep|_{y=0}$.

Using identity \eqref{e:evlt Dawe y=0}, inequality \eqref{e:Linfty Dawe y=0} and
definition \eqref{e:G eqn} of $\Omega$, we have,
for any $s \ge 4$,
\[
    \|\dt\wep |_{y=0} \|_{L^\infty (\T)} \le C_{s, \gamma} \Omega(t),
\]
so a direct integration yields
%
%  eqn 5.59
\begin{equation} \label{e:min wep y=0}
    \min_{\T} \wep (t) |_{y=0}
    \ge \min_{\T} \w_0 |_{y=0} - C_{s, \gamma} \Omega(t)t.
\end{equation}
Combining estimates \eqref{e:min sigma wep} and \eqref{e:min wep y=0}, we prove
\eqref{e:weighted Linfty s4}.

\end{proof}
%
%
%

%=================================================================
%
%    6.    Proof of the Main Theorem
%
%==================================================================
%
\section{Proof of the Main Theorem} \label{s:pf main thm}
The purpose of this section is to complete the proof of our main theorem
\ref{t:Hs exist PS}. In other words, we will prove existence and uniqueness to the
Prandtl equations \eqref{e:PE} in the subsection \ref{ss:exist PE} and
\ref{ss:unique PE} respectively.

%-----------------------------------------------------------------------------
%
%          6.1   Existence for the Prandtl Equations
%
%-----------------------------------------------------------------------------
%
\subsection{Existence for the Prandtl Equations} \label{ss:exist PE}
In this subsection we will construct the solution to the Prandtl equations \eqref{e:PE}
by passing to the limit $\e$ goes to $0^+$ in the regularized Prandtl equations \eqref{e:RPE}.
Our proof will be based on the uniform (in $\e$) weighted estimates derived in proposition
\ref{p:Uniform RPE}. Using these estimates, we will derive uniform bounds and lifespan
on $\wep$, and then prove convergence of $\wep$ and consistency of the limit $\w$
as follows.\\
\\
\underline{\textbf{Uniform Bounds and Life-span on $\wep$}}

According to proposition \ref{p:loc sol RPE}, a solution $\wep$ to the
regularized Prandtl equations \eqref{e:RPE} exists up to a time interval
$[0,T_{s,\gamma,\sigma, \delta, \e, \w_0, U}]$, which may depend on
$\e$ as well. However, in proposition
\ref{p:Uniform RPE} we have already derived uniform (in $\e$) estimates on
$\wep$, so one may apply the standard continuous induction argument to further
solve the Prandtl equations \eqref{e:PE} up to a time
interval which is independent of $\e$. As a result, we have
%
%   prop 6.1
\begin{prop}[Uniform Life-span and Estimates for $\wep$] \label{p:lifespan}
In addition to the hypotheses of proposition \ref{p:loc sol RPE}, when $s=4$, we further assume that $\delta > 0$ is chosen small
enough such that the initial hypothesis \eqref{e:mainthm con} holds. Then there exists a uniform
life-span $T := T (s,\gamma, \sigma, \delta, \|\w_0\|_{H^{s,\gamma}}, U) > 0$, which is independent of $\e$, such that the
regularized vorticity system \eqref{e:Rvort} - \eqref{e:Ruv} has a solution
$\wep \in C([0,T]; \Hs) \cap C^1 ([0,T]; H^{s-2,\gamma})$ with the following uniform
(in $\e$) estimates:
\begin{itemize}
\item[(i)] (Uniform  Weighted $H^s$ Estimate) For any $\e \in [0,1]$ and any
$t \in [0,T]$,
%
%  eqn 6.1
\begin{equation} \label{e:uniform weighted Hs wep}
    \|\wep (t)\|_{H^{s,\gamma}_g} \le 4 \|\w_0\|_{H^{s,\gamma}_g}.
\end{equation}
\item[(ii)] (Uniform Weighted $L^\infty$ Bound) For any $\e \in [0,1]$ and $t \in [0,T]$,
%
%  eqn 6.2
\begin{equation} \label{e:weighted Linfty bound}
    \left\| \sum_{|\alpha| \le 2} |(1+y)^{\sigma+\alpha_2} \Dawe (t) | \right\|^2_{L^\infty (\TR)}
    \le \frac{1}{\delta^2}.
\end{equation}
\item[(iii)] (Uniform Weighted $L^\infty$ Lower Bound) For any $\e \in [0,1]$ and $t \in
[0,T]$,
%
%  eqn 6.3
\begin{equation} \label{e:weighted Linfty lower}
    \min_{\TR} (1+y)^\sigma \wep (t)  \ge \delta.
\end{equation}

\end{itemize}
\end{prop}
%
%  proof of prop 6.1
\begin{proof}
The uniform life-span $T := T(s,\gamma, \sigma, \delta, \|\w_0\|_{H^{s, \gamma}}, U)$
can be guaranteed by the uniform estimates \eqref{e:uniform weighted Hs wep} -
\eqref{e:weighted Linfty lower}, so it suffices to justify them. Indeed, the life-span
$T$ can be taken as $\min \{T_1, T_2, T_3 \}$ where $T_1, T_2$ and $T_3$ will be
defined below.\\
\\
(i) According to the definition \eqref{e:F eqn} of $F$ and the regularized Bernoulli's
law \eqref{e:RBer},
\[
    \|F\|_{L^\infty} \le \constC \{1 + \sum^{\frac{s}{2}+1}_{l=0}
    \|\dt^l U \|^2_{H^{s-2l+2} (\T)}  \}^2  \le  \constC M_U
\]
where $M_U := \sup_t \{1 + \sum^{\frac{s}{2}+1}_{l=0} \|\dt^l U\|^2_{H^{s-2l+2} (\T)} \}^2
< + \infty$, so if we take
$T_1 := \min \Big\{ \frac{3 \|\w_0\|^2_{\Hg}}{\constC M_U}\,$,
    $\frac{1-2^{-s+2}}{2^{s-2} \constC \|\w_0\|^{s-2}_{\Hg}} \Big\}$,
then by inequality \eqref{e:Hsgamma}, estimate \eqref{e:uniform weighted Hs wep} holds
for all $t \in [0,T_1]$.\\
\\
(ii) When $s \ge 6$, using part (i) of proposition \ref{p:lifespan}, we know from the
definition \eqref{e:G eqn} of $\Omega$ and $G$ that for any $t \in [0,T_1]$ where $T_1$ is
defined in part (i),
%
%  eqn 6.4
\begin{equation} \label{e:K eqn}
    \Omega(t) \le 4 \|\w_0\|_{\Hg}  \qquad \text{ and } \qquad
    G(t) \le 4 \|\w_0\|_{\Hg} + M_U = : K.
\end{equation}
Thus, if we take $T_2 := \min \left\{T_1, \frac{1}{64\delta^2 C_{s, \gamma}
(1+ 4 \|\w_0\|_{\Hg}) \|\w_0\|^2_{\Hg}},
\frac{\ln 2}{\constC (1+K)} \right\}$, then using inequality \eqref{e:weighted Linfty s6}
and the initial assumption $\sum_{|\alpha| \le 2} |(1+y)^{\sigma + \alpha_2} \Dalpha
\w_0 |^2 \le \frac{1}{4\delta^2}$, we have the upper bound \eqref{e:weighted Linfty bound}
for all $t \in [0,T_2]$.

When $s = 4$, using inequality \eqref{e:I s4}, estimate \eqref{e:K eqn} and the initial hypothesis
\eqref{e:mainthm con}, we also have the upper bound \eqref{e:weighted Linfty bound} for all
$t \in [0, T_2]$.\\
\\
(iii) Let us take $T_3 := \min \left\{T_1, \frac{\delta}{8 C_{s, \gamma} \|\w_0\|_{\Hg}},
\frac{1}{6\constC (1+K)}, \frac{\ln 2}{\constC (1+K)} \right\}$. Then using inequalities
\eqref{e:weighted Linfty s4} and \eqref{e:K eqn}, we know that the lower bound
\eqref{e:weighted Linfty lower} holds for all $t \in [0,T_3]$.

\end{proof}
\hfill\\
\underline{\textbf{Convergence and Consistency}}

Using almost equivalence relation \eqref{e:almost equiv} and uniform weighted
$H^s$ estimate \eqref{e:uniform weighted Hs wep}, we have
%
%  eqn 6.5
\begin{equation}  \label{e:wep uep-U bdds}
  \sup_{0\leq t \leq T}  ( \|\wep\|_{H^{s,\gamma}} + \|\uep -U\|_{H^{s,\gamma-1}} )
    \le \constC \{4 \|\w_0\|_{\Hg} + \|\dx^s U\|_{L^2} \} < + \infty.
\end{equation}
Furthermore, using evolution equations  $\eqref{e:evlt eqn}$, uniform $H^s$ bound
\eqref{e:wep uep-U bdds} and proposition \ref{p:L2Linfty uvwgk}, one  also find
that $\dt \wep $ and $\dt (\uep-U)$ are uniformly (in $\e$) bounded in
$L^\infty([0,T];  H^{s-2, \gamma})$ and $L^\infty([0,T];  H^{s-2, \gamma-1})$ respectively.
By the Lions-Aubin Lemma  and the compact embedding of  $ H^{s,\gamma} $ in $H^{s'}_{loc}$
stated in  lemma \ref{l:f eqn}, we  have, taking a subsequence if necessary,
as $\e_k \to 0^+$,
%
%  eqn 6.6
\begin{equation} \label{e:converge wep uep-U}
\left\{\begin{aligned}
    &\w^{\e_k} \stackrel{*}{\rightharpoonup} \w \text{ in } L^\infty([0,T]; H^{s,\gamma})
     & \text{ and }\qquad
    \w^{\e_k} \rightarrow \w \text{ in } C([0,T]; H^{s'}_{loc}),\\
    &u^{\e_k} - U \stackrel{*}{\rightharpoonup} u-U \text{ in } L^\infty([0,T]; H^{s,\gamma-1})
    & \text{ and } \qquad
    u^{\e_k} \rightarrow u \text{ in } C([0,T]; H^{s'}_{loc}),
\end{aligned}\right.
\end{equation}
for all $s' < s$, where $\w \in L^\infty([0,T];H^{s,\gamma})\cap \bigcap_{s'<s}C([0,T];H^{s'}_{loc})$, $u-U \in    L^\infty([0,T]; H^{s,\gamma-1} ) \cap \bigcap_{s'<s}C([0,T];H^{s'}_{loc})$ and $\w = \dy u$.
Using the local uniform convergence of $\dx u^{\e_k}$, we also have the
pointwise convergence of $v^{\e_k}$: as $\e_k \to 0^+$,
%
%  eqn 6.7
\begin{equation}  \label{e:converge vep}
    v^{\e_k} = - \int^y_0 \dx u^{\e_k}\, dy \rightarrow - \int^y_0 \dx u \, dy =: v.
\end{equation}

Combining \eqref{e:converge wep uep-U} - \eqref{e:converge vep}, one
may justify the pointwise convergences of all terms in the regularized Prandtl equations
$\eqref{e:RPE}_1 - \eqref{e:RPE}_4$. Thus, passing to the limit $\e_k \to 0^+$ in
$\eqref{e:RPE}_1 - \eqref{e:RPE}_4$ and the regularized Bernoulli's law \eqref{e:RBer},
we know that the limit $(u,v)$ solves the Prandtl equations $\eqref{e:PE}_1 -
\eqref{e:PE}_4$ with the Bernoulli's law \eqref{e:Ber} in the classical sense.

Lastly, in order to complete the proof of consistency, it remains to justify that
$\w \in L^\infty([0,T];\Hs)$
and the matching condition
$\eqref{e:PE}_5$. Since $\Dalpha   \w^{\e_k}$ converges to $\Dalpha  \w$ pointwisely
for all $|\alpha| \le 2$ and that  $\w^{\e_k} $ satisfies
%
% eqn 6.8
\begin{equation}  \label{e:pf f in Hs}
%\left\{\begin{aligned}
    \sum_{|\alpha| \le 2} |(1+y)^{\sigma + \alpha_2} \Dalpha  \w^{\e_k} |^2  \le \frac{1}{\delta^2}
    \qquad \text{ and } \qquad (1+y)^\sigma \w^{\e_k} \ge \delta,
%\end{aligned}\right.
\end{equation}
we deduce that \eqref{e:pf f in Hs} still holds for $\w$, and hence,
$\w \in  L^\infty(0,T;\Hs) $.

Also, by the Lebesgue's dominated convergence theorem,
\[
    \int^{+\infty}_0 \w \, dy = \lim_{\e_k \to 0^+} \int^{+\infty}_0 \w^{\e_k} \, dy = U
\]
which is equivalent to the matching condition $\eqref{e:PE}_5$ because $\w = \dy u
> 0$.

To complete the proof of existence, let us state and prove the following

%
%  lemma 6.2
\begin{lem}  \label{l:f eqn}
Let $s$ be a positive integer, $\gamma' \ge 0$ and $M < +\infty$. Assume
%
%  eqn 6.9
\begin{equation} \label{e:M eqn}
    \|f^\e\|_{H^{s,\gamma'}} \le M
\end{equation}
for all $\e \in (0,1]$. Then there exist a function $f \in H^{s, \gamma'}$
and a sequence $\{\e_k\}_{k\in\N} \subseteq (0,1]$ with $\lim_{k \to + \infty}
\e_k = 0^+$ such that as $\e_k \to 0^+$,
%
%  eqn 6.10
\begin{equation} \label{e:converge f}
%\left\{\begin{aligned}
    f^{\e_k} \stackrel{H^{s,\gamma'}}{\rightharpoonup} f \quad \text{ and } \quad
    f^{\e_k} \stackrel{H^{s'}_{loc}}{\rightarrow} f \qquad \text{ for all } s' < s.
%\end{aligned}\right.
\end{equation}

\end{lem}

%  proof of lemma 6.2
\begin{proof}[Proof of lemma \ref{l:f eqn}]
First of all, let us mention that $H^{s,\gamma'}$ has an inner product structure:
\[
    <\phi,\psi>_{H^{s,\gamma'}} := \sum_{|\alpha| \le s} \int^{+\infty}_0 \int_{\T}
    (1+y)^{2\gamma' + 2 \alpha_2} \Dalpha \phi \Dalpha \psi,
\]
so the uniform bound \eqref{e:M eqn} implies the weak convergence of $f^{\e_k}$ in \eqref{e:converge f}
via the Banach-Alaoglu theorem.

Next, by the definition of $\|\cdot\|_{H^{s,\gamma'}},\;
\|f^\e\|_{H^s} \le \|f^\e\|_{H^{s,\gamma'}} \le M$. This implies the local $H^{s'}$ norm
convergence in \eqref{e:converge f} for all
$s' < s$ because of the standard compactness of $H^s (\TR)$.

\end{proof}

Finally, let us end this subsection by giving the following
%
% remark 6.3
%
\begin{rem}[Life-span for $U\equiv$ constant]\label{r:remark}
In the special case that $U > 0$ is a constant, one may prove that the life-span $T$ stated in our main theorem
\ref{t:Hs exist PS} is independent of the constant value of $U$, that is,
$T := T(s, \gamma, \sigma, \delta, \|\w_0\|_{H^{s, \gamma}})$. The reasoning is as follows:

When $U \equiv$ constant, it follows from the regularized Bernoulli's law
\eqref{e:RBer} that $\dx \pep \equiv 0$, so by definitions \eqref{e:F eqn}
and \eqref{e:G eqn}, we have $F \equiv \constC$ and $G(t) = \Omega(t) = \sup_{[0,t]} \|\wep\|_{\Hg}$ where all of
$F$, $G$ and $\Omega$ are independent of $U$. As a result, all of our
weighted estimates \eqref{e:Hsgamma}, \eqref{e:I s4}, \eqref{e:weighted Linfty s6} and
\eqref{e:weighted Linfty s4} are independent of $U$. Therefore, one may
slightly modify the proof of proposition \ref{p:lifespan} to show that
uniform weighted estimates \eqref{e:uniform weighted Hs wep} -
\eqref{e:weighted Linfty lower} hold in a time interval
$[0, T_{s, \gamma, \sigma, \delta, \|\w_0\|_{H^{s, \gamma}}}]$, which is
independent of $U$. According to our proof of convergence and
consistency in subsection \ref{ss:exist PE}, we can solve the solution
$(u, v)$ of the Prandtl equations \eqref{e:PE} in the same time interval,
and hence, the life-span $T$ stated in the main theorem \ref{t:Hs exist PS}
is also independent of $U$.
\end{rem}

%-----------------------------------------------------------------------------
%
%   6.2  Uniqueness for the Prandtl Equations
%
%----------------------------------------------------------------------------------
%
\subsection{Uniqueness for the Prandtl Equations} \label{ss:unique PE}
The aim of this section is to prove the uniqueness of $\Hs$ solutions constructed in
subsection \ref{ss:exist PE}. To show the uniqueness, we will generalize the nonlinear
cancelation applied in subsubsection \ref{sss:geps} to the $L^2$ comparison of two
$\Hs$ solutions. This motivates us to consider the quantity $\tilde{g}$ below.

Specifically, the uniqueness of $\Hs$ solutions to the Prandtl equations \eqref{e:PE}
is a direct consequence of the following $L^2$ comparison principle.
%
%  prop 6.4
\begin{prop}[$L^2$ Comparison Principle] \label{p:L2 compare}
For any $s \ge 4, \gamma \ge 1, \sigma > \gamma + \frac{1}{2}$ and $\delta \in (0,1)$,
let $(u_i, v_i)$ solve the Prandtl equations \eqref{e:PE} with the vorticity $\w_i
:= \dy u_i \in C([0,T]; \Hs) \cap C^1 ([0,T]; H^{s-2,\gamma})$ for $i = 1, 2$. Define
$\tilde{g} := \w_1 - \w_2 + \frac{\dy\w_2}{\w_2} (u_1 - u_2)$. Then
we have
%
%  eqn 6.11
\begin{equation} \label{e:tilde g est}
    \|\tilde{g} (t)\|^2_{L^2} + \int^t_0 \|\dy \tilde{g} \|^2_{L^2}
    \le \|\tilde{g} (0) \|^2_{L^2} + C_{\gamma, \sigma, \delta, \w, U}
    \int^t_0 \|\tilde{g} \|^2_{L^2}
\end{equation}
where the constant $ C_{\gamma, \sigma, \delta, \w, U}$ depends on $\gamma, \sigma,
\delta, \|\w_1\|_{H^{4,\gamma}_g}, \|\w_2\|_{H^{4,\gamma}_g}$ and
$\|\dx^4 U\|_{L^2 (\T)}$ only.

\end{prop}

Applying the Gronwall's lemma to \eqref{e:tilde g est}, we obtain
\[
    \|\tilde{g}(t)\|^2_{L^2} \le \| \tilde{g} (0) \|^2_{L^2}
    e^{C_{\gamma, \sigma, \delta, \w, U}t}
\]
which implies $\tilde{g} \equiv 0$ provided that $u_1|_{t=0} = u_2 |_{t=0}$. Since $\w_2 \dy \left( \frac{u_1-u_2}{\w_2} \right) = \tilde{g} \equiv 0$,
we have
%
%  eqn 6.12
\begin{equation}  \label{e:u1-u2}
    u_1 - u_2 = q \w_2
\end{equation}
for some function $q := q(t,x)$. Using the Oleinik's monotonicity assumption $\w_2 > 0$
and the Dirichlet boundary condition $u_i |_{y=0} \equiv 0$ for $i = 1, 2$,
we know via \eqref{e:u1-u2} that $q \equiv 0$, and hence, $u_1 \equiv u_2$.
Since $v_i$ can be uniquely determined by $u_i$, we also have $v_1 \equiv
v_2$. This proves the uniqueness of $\Hs$ solutions.

In the rest of this subsection, we will prove proposition \ref{p:L2 compare} as follows.
%
%
%  proof prop 6.4
\begin{proof}[Proof of proposition \ref{p:L2 compare}]
Let us denote $(\tilde{u}, \tilde{v}) = (u_1, v_1) - (u_2, v_2)$ and
$a_2 := \frac{\dy\w_2}{\w_2}$. Then one may check that $\tilde{g} = \tilde{\w}
-a_2 \tilde{u} = \w_2 \dy \left(\frac{\tilde{u}}{\w_2} \right)$ and satisfies
%
%  eqn 6.13
\begin{equation} \label{e:tilde g eqn}
    (\dt+u_1 \dx + v_1 \dy - \dyy) \tilde{g}
    = - 2 \tilde{\w} \dy a_2 - \tilde{u} \{\tilde{u} \dx a_2 + \tilde{v} \dy a_2
    + 2 a_2 \dy a_2 \}.
\end{equation}

To derive the $L^2$ estimates on $\tilde{g}$, let us first recall that we define the cutoff function
$\chi_R (y) := \chi (\frac{y}{R} )$ for any $R \ge 1$, where $\chi \in C^\infty_c
([0, +\infty))$ satisfies the properties \eqref{e:chi properties}. Then $\chi_R$ has
the following pointwise properties: as $R \to + \infty$,
\[
     \chi_R \to \mathbf{1}_{\R^+}, \quad |\chi'_R| \le \frac{2}{R} \to 0^+ \quad\text{ and }\quad
     |\chi''_R| \le O(\frac{1}{R^2}) \to 0^+.
\]
For any $t \in (0,T]$, multiplying equation \eqref{e:tilde g eqn} by $2 \chi_R
\tilde{g}$, and then integrating over $[0,t] \times \TR$, we obtain, via
integration by parts,
%
%  eqn 6.15
\begin{align}
    & \; \iint \chi_R \tilde{g}^2 (t) \, dydx - \iint \chi_R \tilde{g}^2 |_{t=0}\,
    dydx \notag\\
    = & \; -2 \int^t_0 \iint \chi_R |\dy \tilde{g} |^2 - 2 \int^t_0 \int_{\T}
    \tilde{g} \dy \tilde{g} |_{y=0} \, dx - 4 \int^t_0 \iint \chi_R \tilde{g}
    \tilde{\w} \dy a_2 \notag\\
    & \; -2 \int^t_0 \iint \chi_R \tilde{g} \tilde{u} \{\tilde{u} \dx a_2
    + \tilde{v} \dy a_2 + 2 a_2 \dy a_2 \} + \mathcal{R}_1 + \mathcal{R}_2 \notag\\
    \le & \; -2 \int^t_0 \iint \chi_R |\dy \tilde{g} |^2 - 2 \int^t_0 \int_{\T}
    \tilde{g} \dy \tilde{g} |_{y=0} \, dx +4 \|\dy a_2\|_{L^\infty} \int^t_0
    \|\tilde{g}\|_{L^2} \|\tilde{\w}\|_{L^2} \notag\\
    & \; + 2 \|(1+y) \{\tilde{u} \dx a_2 + \tilde{v} \dy a_2 + 2 a_2 \dy a_2 \}
    \|_{L^\infty} \int^t_0 \|\tilde{g}\|_{L^2} \left\| \frac{\tilde{u}}{1+y}
    \right\|_{L^2} + \mathcal{R}_1 + \mathcal{R}_2, \label{e:chi tilde g}
\end{align}
where the remainder terms $\mathcal{R}_i$ are defined by
\[
    \mathcal{R}_1 := \int^t_0 \iint \chi'_R v_1 \tilde{g}^2 \quad \text{ and }
    \quad \mathcal{R}_2 := \int^t_0 \iint \chi''_R \tilde{g}^2.
\]

Now, the first technical problem is to deal with the boundary integral
$\int^t_0 \int_{\T} \tilde{g} \dy \tilde{g} |_{y=0} \, dx$. Since $\dy \tilde{g}
|_{y=0} = - a_2 \tilde{g} |_{y=0}$, we have, after applying the simple trace estimate
\eqref{e:trace est}, \
%
%  eqn 6.15
\begin{equation} \label{e:int tilde g}
    \left| \int^t_0 \int_{\T} \tilde{g} \dy \tilde{g} |_{y=0} \, dx \right|
    \le \frac{1}{2} \int^t_0 \iint \chi_R |\dy \tilde{g}|^2
    + C \{ \|a_2\|_{L^\infty} + \|a_2\|^2_{L^\infty} + \|\dy a_2\|_{L^\infty} \}
    \int^t_0 \iint \tilde{g}^2.
\end{equation}

Furthermore, since $\w_2 \in \Hs$, it follows from the weighted $L^\infty$ bounds on $\w_2$ that
%
%  eqn 6.16
\begin{equation}  \label{e:Linfty a2}
\left\{\begin{aligned}
    \|(1+y) a_2 \|_{L^\infty} & \le \delta^{-2}\\
    \|(1+y) \dx a_2 \|_{L^\infty}, \|(1+y)^2 \dy a_2\|_{L^\infty}
    & \le \delta^{-2} + \delta^{-4},
\end{aligned}\right.
\end{equation}
so by proposition \ref{p:L2Linfty uvwgk},
%
%  eqn 6.18
\begin{equation}  \label{e:Linfty a2 2}
%\begin{split}
    \|(1+y) \{\tilde{u} \dx a_2 + \tilde{v} \dy a_2 + 2 a_2 \dy a_2\}
    \|_{L^\infty}
    \le  C_{\gamma, \sigma, \delta} \{1+ \|\w_1\|_{H^{4,\gamma}_g}
    + \|\w_2\|_{H^{4,\gamma}_g} + \|\dx^4 U\|_{L^2 (\T)} \}.
%\end{split}
\end{equation}

Substituting \eqref{e:int tilde g} - \eqref{e:Linfty a2 2} into
\eqref{e:chi tilde g}, we obtain
%
%  eqn 6.19
\begin{equation} \label{e:tilde g2}
\begin{split}
    & \; \iint \chi_R \tilde{g}^2 (t) \, dydx - \iint \chi_R \tilde{g}^2
    |_{t=0}\, dydx\\
    \le & \; - \int^t_0 \iint \chi_R |\dy \tilde{g} |^2 + C_\delta \int^t_0
    \| \tilde{g} \|^2_{L^2} \\
    & \; + C_{\gamma, \sigma, \delta, \|\w_1\|_{H^{4,\gamma}_g},
    \|\w_2\|_{H^{4,\gamma}_g}, \|\dx^4 U\|_{L^2}} \int^t_0 \|\tilde{g}\|_{L^2}
    \left\{ \|\tilde{\w}\|_{L^2} + \left\| \frac{\tilde{u}}{1+y} \right\|_{L^2}
    \right\} + \mathcal{R}_1 + \mathcal{R}_2.
\end{split}
\end{equation}

Next, we emphasize that both $\tilde{\w}$ and $\frac{\tilde{u}}{1+y}$ can be
controlled by $\tilde{g}$, namely,
%
%
%   claim 6.4
\begin{claim} \label{claim:tilde w and u}
%
%  eqn 6.20
\begin{equation} \label{e:tilde w and u}
    \|\tilde{\w}\|_{L^2}, \left\| \frac{\tilde{u}}{1+y} \right\|_{L^2}
    \le C_{\sigma, \delta} \|\tilde{g}\|_{L^2}.
\end{equation}
\end{claim}

The proof of claim \ref{claim:tilde w and u} is very similar to that of lemma
\ref{l:L2 dxk u-U dxk w gk}, so we will only outline it at the end of this
subsection. Assuming claim \ref{claim:tilde w and u} for the moment, we can
substitute \eqref{e:tilde w and u} into \eqref{e:tilde g2} to obtain
%
%   eqn 6.21
\begin{equation} \label{e:chi tilde g2}
\begin{split}
    & \; \iint \chi_R \tilde{g}^2 (t) \,dydx - \iint \chi_R \tilde{g}^2 |_{t=0}\,
    dydx \\
    \le & \; - \int^t_0 \iint \chi_R |\dy \tilde{g}|^2
    + C_{\gamma, \sigma, \delta, \|\w_1\|_{H^{4,\gamma}_g},
    \|\w_2\|_{H^{4,\gamma}_g}, \|\dx^4 U\|_{L^2}} \int^t_0 \|\tilde{g}\|^2_{L^2}
    + \mathcal{R}_1 + \mathcal{R}_2.
\end{split}
\end{equation}

Finally, both integrands of $\mathcal{R}_1$ and $\mathcal{R}_2$ can be
controlled by a multiple of $\tilde{g}^2$, which belongs to
$L^1 ([0,T]; \TR)$, so applying Lebesgue's dominated convergence theorem, we
have
%
%
%  eqn 6.22
\begin{equation} \label{e:limit Ri}
    \lim_{R \to + \infty} \mathcal{R}_i = 0 \qquad\qquad \text{ for } i=1,2.
\end{equation}

Using monotone convergence theorem and \eqref{e:limit Ri}, we can pass to the limit $R \to + \infty$ in
\eqref{e:chi tilde g2} to obtain \eqref{e:tilde g est}.

Lastly, we will justify claim \ref{claim:tilde w and u} as follows.
%
%
%  proof claim 6.4
\begin{proof}[Proof of claim \ref{claim:tilde w and u}]
Using triangle inequality and $\eqref{e:Linfty a2}_1$, we have
$\|\tilde{\w}\|_{L^2} \le \|\tilde{g}\|_{L^2} + \delta^{-2}
\left\| \frac{\tilde{u}}{1+y} \right\|_{L^2}$,
so it suffices to control $\tilde{u}$.

Since $\delta\le (1+y)^\sigma \w_2 \le \delta^{-1}$ and $\tilde{u} |_{y=0} \equiv 0$,
applying part (ii) of lemma \ref{l:Hardy}, we obtain
\[
    \left\| \frac{\tilde{u}}{1+y} \right\|_{L^2}
     \le  \delta^{-1} \left\| (1+y)^{-\sigma-1} \frac{\tilde{u}}{\w_2}
    \right\|_{L^2}
     \le  C_{\sigma, \delta} \left\| (1+y)^{-\sigma} \dy
    \left(\frac{\tilde{u}}{\w_2} \right) \right\|_{L^2}
     \le  C_{\sigma, \delta} \| \tilde{g} \|_{L^2}
\]
because $\tilde{g}=\w_2\dy (\frac{\tilde{u}}{\w_2})$.
\end{proof}

\end{proof}

%=================================================================
%
%    7.    Existence for the Regularized Prandtl Equations
%
%==================================================================
%
\section{Existence for the Regularized Prandtl Equations} \label{s:exist RPE}
The aim of this section is to solve the regularized Prandtl equations \eqref{e:RPE},
or equivalently its vorticity system \eqref{e:Rvort} - \eqref{e:Ruv}. In other words,
we will prove proposition \ref{p:loc sol RPE} according to the plan described in
section \ref{s:approx}. However, we will only sketch our proof because
 the method for solving intermediate approximate systems
\eqref{e:Rvort} - \eqref{e:Ruv}, \eqref{e:TRPE} - \eqref{e:TRuv} and
\eqref{e:LTRPE} - \eqref{e:LTRuv} is  standard.
% \end{itemize}
%
Before we proceed, it should be also remarked that we will solve the approximate systems
\eqref{e:LTRPE} - \eqref{e:LTRuv}, \eqref{e:TRPE} - \eqref{e:TRuv} and \eqref{e:Rvort} -
\eqref{e:Ruv} with a decreasing order of regularities. The main reason of this technical
arrangement is to derive our estimates in a rigorous way so that we can differentiate
the intermediate equations pointwisely and have enough pointwise decay at $y = + \infty$
according to prop \ref{p:decay Hs+3}.

%------------------------------------------------------------------------------------------
%
%    7.1     Solving Inhomogenous Heat Equations
%
%----------------------------------------------------------------------------------
\subsection{Solvability of Inhomogenous Heat Equation} \label{ss:inhomo heat}
In this first subsection we will solve an inhomogenous heat equation in the
weighted space $\Hs$. This existence result will be applied to solve the linearized,
truncated and regularized vorticity system \eqref{e:LTRPE} - \eqref{e:LTRuv} in the
next subsection.

Let us consider the following inhomogenous heat equation: for any $\e > 0$,
%
%  eqn 7.1
\begin{equation} \label{e:inhomo heat}
\left\{\begin{aligned}
    \dt W + F_R & = \e^2 \dxx W + \dyy W && \text{ in } [0,T] \times \TR\\
    W|_{t=0} & = W_0 && \text{ on } \TR\\
    \dy W |_{y=0} & = \dx \pep && \text{ on } [0,T] \times \T
\end{aligned}\right.
\end{equation}
where $W$ is an unknown, $W_0$ and $\dx \pep$ are given initial and boundary data,
$F_R$ is a given inhomogenous term with compact support in $[0,T] \times \T
\times [0,2R]$. Since \eqref{e:inhomo heat} is just a standard inhomogenous
heat equation, we can solve it by classical methods and obtain the following
%
%  prop 7.1
\begin{prop}[Existence of Inhomogenous Heat Equation] \label{p:inhomo heat}
Let $s \ge 4$ be an even integer, $\gamma \ge 1, \sigma > \gamma + \frac{1}{2},
\delta \in (0, \frac{1}{2})$ and $\e \in (0,1]$. If $W_0 \in
H^{s+12,\gamma}_{\sigma,2\delta}$ and $\operatorname{supp} F_R \subseteq [0,T]
\times \T \times [0,2R]$, then there exist a time $T := T(s,\gamma,\sigma,\delta,
R, \|W_0\|_{H^{s+8,\gamma}}, \|F_R\|_{C^2}) > 0$ and a solution $W \in
C([0,T]; H^{s+8,\gamma}_{\sigma,\delta}) \cap C^\infty ((0,T] \times \TR)$
to the inhomogenous heat equation \eqref{e:inhomo heat}.

Furthermore, we have the following pointwise decay at $y = + \infty$: for any $l=0,1,\cdots,
\frac{s}{2}+4$, for any $|\alpha| \le s-2l+9$,
%
% eqn 7.2
\begin{equation} \label{e:DW ptwise}
    \dt^l \Dalpha W =
    \begin{cases}
        O ((1+y)^{-\sigma-\alpha_2}) & \text{ if } |\alpha| + 2l \le 2\\
        O ((1+y)^{-\frac{\sigma+(2^{|\alpha|+2l-2}-1)\gamma}{2^{|\alpha|+2l-2}}
        - \alpha_2} ) & \text{ if } 2 \le |\alpha| + 2l \le s+9,
    \end{cases}
\end{equation}
and
energy estimate:
%
%  eqn 7.3
\begin{equation} \label{e:W est}
\begin{split}
    & \; \dfrac{d}{dt} \|| W |\|^2_{s+8,\gamma} + \e^2 \|| \dx W |\|^2_{s+8,\gamma}
    + \||\dy W |\|^2_{s+8,\gamma} \\
    \le & \; C_{s,\gamma} \||W|\|^2_{s+8,\gamma} + C_s \||W|\|_{s+8,\gamma} \||F_R|\|_{s+8,\gamma}
    + C_s \|| F_R |\|^2_{s+7,\gamma} + C_s \|| \dx \pep|\|^2_{s+8},
\end{split}
\end{equation}
where the norms $\||\cdot|\|_{s',\gamma}$ and $\|| \cdot |\|_{s'}$ are defined in
definition \ref{defn:weighted norms}.
\end{prop}

\begin{proof}[Outline of the Proof]
Using the method of reflection and Duhamel's principle, one may express the unique
global-in-time $C([0,T] \times \TR) \cap C^\infty ((0,T] \times \TR)$ solution
to \eqref{e:inhomo heat} by an explicit solution formula
%
%  eqn 7.4
\begin{equation} \label{e:heat sol}
    W = K_{W_0} + K_{\dx \pep} + K_{F_R}
\end{equation}
where the terms $K_{W_0}, K_{\dx \pep}$ and $K_{F_R}$ can be written explicitly
by using the Gaussian (i.e., the heat kernel), and depend on $W_0, \dx \pep$ and
$F_R$ respectively. Since the solution formula \eqref{e:heat sol} is explicit,
based on the properties of the Gaussian, one may prove the following two facts:
as $y \to + \infty$,
\begin{itemize}
\item[(i)] both $K_{\dx\pep}$ and $K_{F_R}$ decay exponentially fast;
\item[(ii)] the term $\Dalpha K_{W_0} \lesssim ($or $\gtrsim) (1+y)^{-b_\alpha}$
provided that $\Dalpha W_0$ does.
\end{itemize}
Using quantitative versions of facts (i) and (ii), one can justify by
\eqref{e:heat sol} that $W$ fulfills all weighted $L^\infty$ controls for
$H^{s+8,\gamma}_{\sigma,\delta}$ within a short time interval
$[0,T_{s,\gamma,\sigma,\delta, R, \|W_0\|_{H^{s,\gamma}}, \|F_R\|_{C^2}}]$.

Furthermore, applying proposition \ref{p:decay Hs+3} with $s' = s+8$, we know that
$W_0$ satisfies the pointwise decay \eqref{e:DW ptwise} for $l=0$ and $|\alpha|
\le s+9$, and hence, $W$ does. Using the heat equation $\eqref{e:inhomo heat}_1$ repeatedly,
we also obtain the pointwise decay \eqref{e:DW ptwise} in our desired ranges of
$l$ and $\alpha$.

Finally, it remains to show the energy estimate \eqref{e:W est}, but its proof
just follows from standard energy methods, so we will omit the proof here.
However, during the estimation, one requires to apply an integration by parts
in the $y$-direction to deal with the operator $\dyy$, so we would like to give
the following two remarks on the boundary values of $W$:
\begin{itemize}
\item[(I)] (Boundary Values at $y=0$) The boundary values of $W$ as well as
its derivatives can be reconstructed by using the boundary reduction
formula
\[
    \dy^{2k+1} W|_{y=0} = (\dt - \e^2\dxx)^k \dx\pep + \sum^{k-1}_{j=0}
    (\dt - \e^2 \dxx)^{k-j-1} \dy^{2j+1} F_R |_{y=0},
\]
which reduces the order of the boundary terms so that we can control the
boundary integral at $y=0$ via the simple trace estimate \eqref{e:trace est};
\item[(II)] (Boundary Values at $y = +\infty$) All boundary
terms of $W$ as well as its derivatives required for deriving
energy estimate \eqref{e:W est} actually vanish fast enough at $y=+\infty$
because of the pointwise decay estimate \eqref{e:DW ptwise}. Thus, all required
boundary integrals at $y = +\infty$ are zero.

\end{itemize}

\end{proof}

%----------------------------------------------------------------------------
%
%   7.2   Solving Linearized, Truncated and Regularized Vorticity System
%
%-----------------------------------------------------------------------------
%
\subsection{Solvability of Linearized, Truncated and Regularized Vorticity System}
\label{ss:solve LTRPE}
Based on the $\Hs$ solutions to the inhomogeneous heat equation \eqref{e:inhomo heat}
derived in subsection \ref{ss:inhomo heat}, we will construct a sequence of solutions
to the linearized, truncated and regularized vorticity system \eqref{e:LTRPE} -
\eqref{e:LTRuv} with uniform bounds in this subsection. This sequence of solutions as
well as their uniform bounds will be the foundation for solving the truncated and
regularized vorticity system \eqref{e:TRPE} - \eqref{e:TRuv}  in the next subsection.

Let us begin by defining an iterative sequence $\{(u^n, v^n, \w^n)\}_{n \in \N}$
as follows:
\begin{itemize}
    \item[(i)] $\w^0 (t,x,y) := \w_0 (x,y)$;
    \item[(ii)] $(u^n, v^n)$ is defined by formulae \eqref{e:LTRuv} for all $n \in \N;$
    \item[(iii)] $\w^{n+1}$ is defined to be the $C([0,T]; H^{s+8,\gamma}_{\sigma,\delta})
    \cap C^\infty ((0,T] \times \TR)$ solution to the linearized, truncated and regularized
    vorticity system \eqref{e:LTRPE} - \eqref{e:LTRuv} for all $n \in N$.
\end{itemize}

The natural question is whether the iterative sequence $\{(u^n, v^n, \w^n)\}_{n \in \N}$
is well-defined, and the answer is affirmative because of the following
%
%  prop 7.2
\begin{prop}[Existence of Linearized, Truncated and Regularized Vorticity System]\label{p:Solve LTRPE}
Let $s \ge 4$ be an even integer, $\gamma \ge 1, \sigma > \gamma + \frac{1}{2}, \delta \in
(0,\frac{1}{2}), \e \in (0,1]$ and $R \ge 1$. If $\w_0 \in H^{s+12,\gamma}_{\sigma, 2\delta}$
and $\sup_t \||U|\|_{s+9, \infty} < + \infty$ where $\||\cdot|\|_{s',\infty}$ is defined
in definition \ref{defn:weighted norms}, then there exist a uniform life-span
$T := T(s,\gamma,\sigma,\delta, \e,\chi, R, \|\w_0\|_{H^{s+8,\gamma}},
\sup_t \||U|\|_{s+9,\infty}) > 0$ which is independent of $n$, and a sequence
of solutions $\{\w^n\}_{n\in \N} \subseteq C([0,T];
H^{s+8,\gamma}_{\sigma,\delta}) \cap C^\infty ((0,T] \times \TR)$ to the linearized, truncated
and regularized vorticity system \eqref{e:LTRPE} - \eqref{e:LTRuv}.

Furthermore, the pointwise decay estimate \eqref{e:DW ptwise} holds for $W := \w^n$ for
all $n \in \N$ and we have the following uniform (in $n$) energy estimate: for all
$n \in \N$ and for all $t \in [0,T]$,
%
% eqn 7.5
\begin{equation} \label{e:wn est}
    \|| \w^n|\|^2_{s+8,\gamma} + \e^2 \int^t_0 \||\dx\w^n|\|^2_{s+8,\gamma} +
    \int^t_0 \|| \dy \w^n|\|^2_{s+8,\gamma} \le \mathcal{Q}_{s+10}
    (\|\w_0\|_{H^{s+8},\gamma})
\end{equation}
where the norm $\||\cdot|\|_{s',\gamma}$ is defined in definition \ref{defn:weighted norms}
and $\mathcal{Q}_l$ is a degree $l$ polynomial with non-negative coefficients which depends
on $C_{s,\gamma,\chi}$ and $\||U|\|_{s+8,\infty}$ only.
\end{prop}
\begin{proof}[Outline of the proof]
For a given $\w^n \in C([0,T]; H^{s+8,\gamma}_{\sigma,\delta}) \cap
C^\infty ((0,T] \times \TR)$, the local in time solvability of $\w^{n+1}$ in the same
function space follows directly by applying proposition \ref{p:inhomo heat} with
$W := \w^{n+1}$ and $F_R := \chi_R \{u^n \dx \w^n + v^n \dy \w^n\}$, although the life-span
$T$ may depend on $n$ a priori. However, the unform (in $n$) energy estimate
\eqref{e:wn est} guarantee the unform (in $n$) life-span $T$ by the standard continuous
induction argument, so it suffices to prove \eqref{e:wn est}.

In order to derive the energy estimate \eqref{e:wn est}, we have to control
$\||F_R|\|_{s+7,\gamma}$ and $\||F_R|\|_{s+8,\gamma}$ for $F_R := \chi_R
\{u^n \dx\w^n + v^n \dy \w^n\}$. Using the triangle inequality, proposition
\ref{p:basic ineq weighted} and proposition \ref{p:cntl uv}, one may check that
%
%  eqn 7.6
\begin{equation} \label{e:FR 78}
\left\{\begin{aligned}
    \||F_R|\|_{s+7,\gamma} & \le C_{s,\gamma,\chi} \||\w^n|\|^2_{s+8,\gamma}
    + \||U|\|^2_{s+8,\infty}\\
    \||F_R|\|_{s+8,\gamma} & \le C_{s,\gamma,\chi,\R} \{ \||\w^n|\|_{s+8,\gamma}
    +\||U|\|_{s+9,\infty}\} \cdot \{\||\dx\w^n|\|_{s+8,\gamma} + \||\dy \w^n|\|_{s+8,\gamma}\}
\end{aligned}\right.
\end{equation}
where the norms $\||\cdot|\|_{s',\infty,\gamma}$ and $\||\cdot|\|_{s',\infty}$ are defined in
definition \ref{defn:weighted norms}.

Applying inequalities \eqref{e:W est} and \eqref{e:FR 78}, one can easily
show that as long as $\||\w^n|\|_{s+8,\gamma} |_{t=0} \le L$, there exists a uniform
(in $n$) time interval $[0,T_{s,\gamma,\sigma, \delta,\e,\chi,R,\sup_t \||U|\|_{s+9,\infty},
L}]$ such that
%
%  eqn 7.7
\begin{equation}  \label{e:wn le L}
    \||\w^n|\|^2_{s+8,\gamma} + \e^2 \int^t_0 \||\dx\w^n|\|^2_{s+8,\gamma}
    + \int^t_0 \||\dy \w^n |\|^2_{s+8,\gamma} \le 4 L^2 \qquad \qquad \text{ for all }
    n \in \N
\end{equation}
because both \eqref{e:W est} and \eqref{e:FR 78} are independent of $n$. Therefore, it
remains to derive a uniform (in $n$) control on the initial data
$\||\w^n|\|_{s+8,\gamma}|_{t=0}$.

To estimate $\||\w^n|\|_{s+8,\gamma}|_{t=0}$, let us first state without proof the
following fact:
%
%  eqn 7.8
\begin{equation} \label{e:initial wn}
    \|\dt^l \w^n\|_{H^{s-2l+8,\gamma}} |_{t=0} \le \mathcal{P}_{l+1}
    (\|\w_0\|_{H^{s+8,\gamma}})
\end{equation}
where $\mathcal{P}_{l+1}$ is a degree $l+1$ polynomial defined by
\[
    \mathcal{P}_1 (Z) := Z \qquad \text{ and } \qquad
    \mathcal{P}_{l+1} := \mathcal{P}_l + C_{s,\gamma,\chi} \sum^{l-1}_{j=0}
    (\mathcal{P}_{j+1} + \||U|\|_{s+8,\infty}) \mathcal{P}_{l-j} \qquad
    \text{ for all } l \ge 1.
\]

The fact \eqref{e:initial wn} can be proved by induction on $(n,l)$ together with the
following estimate:
%
%  eqn 7.9
\begin{equation} \label{e:Y n+1}
    Y_{n+1, l+1} \le Y_{n+1,l} + C_{s,\gamma,\chi} \sum^l_{j=0}
    (Y_{n,j} + \|| U|\|_{s+8,\infty} ) Y_{n, l-j}
\end{equation}
where $Y_{n,l} := \|\dt^l \w^n\|_{H^{s-2l+8, \gamma}} |_{t=0}$. The derivation
of \eqref{e:Y n+1}, which is based on $\eqref{e:LTRPE}_1$, proposition \ref{p:basic ineq weighted} and proposition
\ref{p:cntl uv}, will be left for the reader.

Combining estimates \eqref{e:wn le L} and \eqref{e:initial wn}, we show the
uniform energy estimate \eqref{e:wn est} for $\mathcal{Q}_{s+10} :=
4 \sum^{\frac{s}{2}+4}_{l=0} \mathcal{P}_{l+1}^2$.

\end{proof}

%-----------------------------------------------------------------------
%
%   7.3  Solving Truncated and Regularized Vorticity System
%
%-----------------------------------------------------------------------
\subsection{Solvability of Truncated and Regularized Vorticity System} \label{ss:solve TRPE}
The aim of this subsection is to construct a solution to the truncated and
regularized vorticity system \eqref{e:TRPE} - \eqref{e:TRuv} by passing to
the limit in its linearized version \eqref{e:LTRPE} - \eqref{e:LTRuv},
which was solved with uniform bounds in subsection \ref{ss:solve LTRPE}.

In other words, we will prove the following
%
%
%  prop 7.3
\begin{prop} [Existence of Truncated and Regularized Vorticity System] \label{p:solve TRPE}
Let $s \ge 4$ be an even integer, $\gamma \ge 1, \sigma > \gamma + \frac{1}{2},
\delta \in (0,\frac{1}{2}), \e\in (0,1]$ and $R \ge 1$. If $\w_0 \in
H^{s+12,\gamma}_{\sigma, 2\delta}$ and $\sup_t \||U|\|_{s+9,\infty} < + \infty$
where $\||\cdot|\|_{s',\infty}$ is defined in definition \ref{defn:weighted norms},
then there exist a time $T := T(s,\gamma,\sigma, \delta,\e,\chi,R,
\|\w_0\|_{H^{s+8,\gamma}}, \sup_t \||U|\|_{s+9,\infty}) > 0$ and a solution
$\w_R \in C([0,T]; H^{s+8,\gamma}_{\sigma,\delta}) \cap \bigcap^{\frac{s}{2}+4}_{l=1}
C^l ([0,T]; H^{s-2l+8, \gamma})$ to the truncated and regularized vorticity system
\eqref{e:TRPE} - \eqref{e:TRuv}.

Furthermore, we have the following uniform (in $R$) weighted energy estimate:
%
%
%  eqn 7.10
\begin{equation}  \label{e:wR est}
\begin{split}
    & \; \dfrac{d}{dt} \||\w_R|\|^2_{s+4,\gamma} + \e^2 \|| \dx \w_R|\|^2_{s+4,\gamma}
    + \|| \dy \w_R |\|^2_{s+4,\gamma} \\
    \le & \; C_{s,\gamma,\e,\chi} \{1+\||\w_R|\|_{s+4,\gamma} +\||U|\|_{s+6} \}^2
    \|| \w_R|\|^2_{s+4,\gamma} + C_s \{1+ \||U|\|_{s+6} \}^2 \||U|\|^2_{s+6}
\end{split}
\end{equation}
and the following weighted $L^\infty$ estimates:
%
%
%  eqn 7.11
\begin{equation} \label{e:IR Linfty}
    \|I_R (t)\|_{L^\infty (\TR)} \le
    ( \|I_R (0)\|_{L^\infty (\TR)} + C_{s,\gamma} \Lambda  (t) \sup_{[0,t]}
    \|\w_R\|_{H^{s+4,\gamma}}t) e^{C_\sigma \Lambda (t) t}
\end{equation}
%
%
%  eqn 7.12
\begin{equation} \label{e:min wR}
    \min_{\TR} (1+y)^\sigma \w_R (t)
    \ge (1-\Lambda (t) t e^{\Lambda(t)t}) (\min_{\TR} (1+y)^\sigma \w_0
    -C_\sigma \Lambda (t) \sup_{[0,t]} \|\w_R\|_{H^{s+4,\gamma}} t ),
\end{equation}
where the norms $\||\cdot|\|_{s',\gamma}$ and $\||\cdot|\|_{s'}$ are defined in definition \ref{defn:weighted norms}, and the quantities $I_R$ and $\Lambda$ are defined by
\[
%\left\{\begin{aligned}
    I_R (t) := \sum_{|\alpha| \le 2} |(1+y)^{\sigma + \alpha_2} \Dalpha \w_R (t)|^2
    \quad \text{ and } \quad \Lambda (t) := 1 + \sup_{[0,t]} \|\w_R\|_{H^{s+4,\gamma}} +
    \sup_{[0,t]} \|U\|_{C^3 (\T)}.
%\end{aligned}\right.
\]

\end{prop}
\begin{proof} [Outline of the proof]
According to proposition \ref{p:Solve LTRPE}, the sequence of solutions $\{\w^n\}_{n\in\N}$
to the linearized, truncated and regularized vorticity system \eqref{e:LTRPE} - \eqref{e:LTRuv}
has a uniform (in $n$) life-span $[0, T_{s,\gamma,\sigma, \delta, \e, \chi, R, \w_0, U}]$,
in which $\||\w^n|\|_{s+8,\gamma}$ is uniformly bounded by estimate \eqref{e:wn est}. Based on
this uniform bound, one may apply the standard energy methods to $\w^{n+1} -\w^n$ to prove
that the approximate sequence $\{\w^n\}_{n \in \N}$ is indeed Cauchy in the norm
$\sup_{t \in [0,T]} \||\cdot|\|_{s+6,\gamma}$
where the time $T := T(s,\gamma, \sigma, \delta, \e, \chi, R,
\|\w_0\|_{s+8,\gamma}, \sup_t \||U|\|_{s+9, \infty}) > 0$ is independent of $n$. As a result,
we can pass to the limit $n$ goes to $+\infty$ in \eqref{e:LTRPE} - \eqref{e:LTRuv} to
obtain a solution $\w_R := \lim_{n \to + \infty} \w^n$ to the truncated and regularized
vorticity system \eqref{e:TRPE} - \eqref{e:TRuv}. Moreover, $\w_R$ belongs to $C([0,T];
H^{s+8, \gamma}_{\sigma, \delta} ) \cap \bigcap^{\frac{s}{2}+4}_{l=1} C^l
([0,T]; H^{s-2l+8,\gamma})$ because $\w^n$ does.

The uniform energy estimate \eqref{e:wR est} follows from the standard energy methods,
so its proof will be omitted here. It is noteworthy to mention that unlike the estimates in
subsection \ref{ss:solve LTRPE}, all constants in \eqref{e:wR est} are independent of $R$.
This improvement is based on applying integration by parts appropriately to the integral
involving the convection term $\chi_R \{u_R \dx \w_R + v_R \dy \w_R\}$, but it
does not exist in \eqref{e:LTRPE} - \eqref{e:LTRuv} because the linearization destroys this structure.

The weighted $L^\infty$ controls \eqref{e:IR Linfty} and \eqref{e:min wR} can be derived by
the classical maximum principle (see lemmas \ref{l:max parabolic} and \ref{l:min parabolic}
for instance) as in subsection \ref{ss:Linfty lower}. We leave this for the interested reader.

\end{proof}

%-----------------------------------------------------------------------------
%
%    7.4  Solving Regularized Vorticity System and Regularized Prandtl Equations
%
%----------------------------------------------------------------------------
\subsection{Solvability of Regularized Vorticity System and Regularized Prandtl Equations}
\label{ss:solve RPE}
In this subsection we will construct a solution $\wep$ to the regularized vorticity
system \eqref{e:Rvort} - \eqref{e:Ruv} by passing to the limit in its truncated version
\eqref{e:TRPE} - \eqref{e:TRuv}, whose local-in-time solvability and uniform bounds are
shown in subsection \ref{ss:solve TRPE}. Furthermore, we will also justify that $\wep$
solves the regularized Prandtl equations \eqref{e:RPE}.

More precisely, we will complete the proof of proposition \ref{p:loc sol RPE} as follows.
%
%
%  Outline of the Proof of prop 5.1
\begin{proof}[Outline of the proof of proposition \ref{p:loc sol RPE}]
To solve the regularized vorticity system \eqref{e:Rvort} - \eqref{e:Ruv}, we
first pick any function $\chi$ with the properties \eqref{e:chi properties}
in the truncated and regularized vorticity system \eqref{e:TRPE} - \eqref{e:TRuv}.
Then by proposition \ref{p:solve TRPE}, we have a local-in-time solution $\w_R$ to
\eqref{e:TRPE} - \eqref{e:TRuv} and uniform bounds \eqref{e:wR est} - \eqref{e:min wR}
on $\w_R$. Since the estimates \eqref{e:wR est} - \eqref{e:min wR} are independent
of $R$, one can show that there exists a uniform time $T := T(s,\gamma,\sigma,
\delta, \e, \|\w_0\|_{H^{s+4, \gamma}}, U) > 0$ such that $\{\w_R\}_{R \ge1} \subseteq
C([0,T]; H^{s+4,\gamma}_{\sigma,\delta}) \cap C^1 ([0,T]; H^{s+2,\gamma})$ and
$\|\w_R\|_{H^{s+4,\gamma}} \le C_{s,\gamma,\sigma,\delta,\e,\|\w_0\|_{H^{s+4,\gamma}}, U}$
for all $R \ge 1$. Therefore, by the standard compactness argument, there exist a
function $\wep \in C([0,T]; H^{s+4,\gamma}_{\sigma,\delta})
\cap C^1 ([0,T]; H^{s+2,\gamma})$ and a subsequence $\{R_k\}_{k\in \N}$ with
$\lim_{k \to +\infty} R_k = + \infty$ such that $\w_{R_k}$ converges to $\wep$ in
$C([0,T]; H^{s+2}_{loc})$ as $R_k \to + \infty$. As a result, we can pass to the
limit $R_k \to +\infty$ in \eqref{e:TRPE} - \eqref{e:TRuv}, and prove that $\wep$
solves the regularized vorticity system \eqref{e:Rvort} - \eqref{e:Ruv} in a
classical sense.

Finally, we will justify that $(\uep,\vep)$ defined by \eqref{e:Ruv} satisfies the
regularized Prandtl equations \eqref{e:RPE} as follows.

First of all, the matching condition $\eqref{e:RPE}_5$, the Dirichlet boundary
condition $\vep|_{y=0}$ and the initial condition $\eqref{e:RPE}_3$ follows
immediately from the formulae \eqref{e:Ruv}, $\w_0 := \dy u_0$ and the
compatibility condition \eqref{e:u0 conditions}. Then by direct differentiations
on \eqref{e:Ruv}, we also have the incompressibility condition $\eqref{e:RPE}_2$
and $\wep = \dy \uep$.

To justify equation $\eqref{e:RPE}_1$, we substitute $\wep =
\dy \uep$ into $\eqref{e:Rvort}_1$ and obtain, via using $\eqref{e:RPE}_2$,
%
%
%  eqn 7.13
\begin{equation} \label{e:wep Rvort}
    \dy \{\dt\uep + \uep \dx\uep + \vep \wep\} = \dy \{\e^2 \dxx \uep + \dy \wep\}.
\end{equation}
Then one may derive $\eqref{e:RPE}_1$ by integrating \eqref{e:wep Rvort} with
respect to $y$ over $[y, +\infty)$ and using the following pointwise convergences:
as $y \to + \infty$,
%
%
%  eqn 7.14
\begin{equation} \label{e:vw w converges}
\left\{\begin{aligned}
    \vep \wep, \dy \wep & \to 0\\
    \dt \uep + \uep \dx \uep - \e^2 \dxx \uep & \to \dt U +U\dx U - \e^2 \dxx U
    = : - \dx \pep.
\end{aligned}\right.
\end{equation}
The pointwise convergences \eqref{e:vw w converges} can be shown easily as long as
$\wep \in C([0,T]; H^{s+4,\gamma}_{\sigma, \delta}) \cap C^1 ([0,T]; H^{s+2,\gamma})$,
so we leave this to the interested reader.

Lastly, it remains to show the Dirichlet boundary condition $\uep |_{y=0} = 0$. To
prove this, we evaluate the evolution equation $\eqref{e:RPE}_1$ at $y=0$, and apply
the boundary conditions $\vep|_{y=0} = 0$ and $\eqref{e:Rvort}_3$ to obtain that
 $\uep |_{y=0}$ satisfies the viscous Burger's equation:
%
%
%  eqn 7.15
\begin{equation} \label{e:uep Burger}
    \dt (\uep|_{y=0}) + (\uep |_{y=0}) \dx (\uep|_{y=0}) = \e^2 \dxx (\uep|_{y=0}).
\end{equation}
It follows from the classical uniqueness result for the viscous Burger's equation
\eqref{e:uep Burger} that $\uep|_{y=0} \equiv 0$ since it does initially according to
the compatibility condition \eqref{e:u0 conditions}.

\end{proof}

%===========================================================
%
%        Appendix A : Equivalence of Weighted Norms
%
%===========================================================
\appendix
\section{Almost Equivalence of Weighted Norms}\label{s:appendixA}
The purpose of this appendix is to justify the almost equivalence relation \eqref{e:wHg}.
In other words, we will prove the following
%
%  prop A1
\begin{prop}[Almost Equivalence of Weighted $H^s$ Norms] \label{p:almost equiv}
Let $s \ge 4$ be an integer, $\gamma \ge 1, \sigma > \gamma + \frac{1}{2}$ and
$\delta \in (0,1)$. For any $\w \in \Hs (\TR)$, we have the following inequality:
there exist constants $c_\delta$ and $\constC >0$ such that
%
% eqn A1
\begin{equation}\label{e:almost equiv}
    c_\delta \|\w\|_{H^{s,\gamma}_g} \le \|\w\|_{H^{s,\gamma}} + \|u-U\|_{H^{s,\gamma-1}}
    \le \constC \{ \|\w\|_{H^{s,\gamma}_g} + \| \dx^s U\|_{L^2} \}
\end{equation}
provided that $\w = \dy u$, $u|_{y=0} = 0$ and $ \lim_{y \to +\infty} u =U$, where
the weighted $H^s$ norms $\|\cdot\|_{H^{s,\gamma}}$ and $\|\cdot\|_{H^{s,\gamma}_g}$
are defined by \eqref{e:w2Hsgamma} and \eqref{e:Hgsgamma} respectively.
\end{prop}
%
% proof prop A1
\begin{proof}
Without loss of generality, we only need to prove inequality \eqref{e:almost equiv} for
the smooth $\w$ because the full version can be recovered by the standard density argument.
First of all, it follows from the definition of
$\|\w\|_{H^{s,\gamma}}$ and $\|u-U\|_{H^{s,\gamma-1}}$ that
%
% eqn A2
\begin{equation} \label{e:defn w u-U}
\begin{split}
    &\|\w\|_{H^{s,\gamma}} + \sum^s_{k=0} \|(1+y)^{\gamma-1} \dx^k (u-U)\|_{L^2}
    \le \|\w\|_{H^{s,\gamma}} + \|u-U\|_{H^{s,\gamma-1}} \\
    &\qquad\qquad\qquad\le 2 \{\|\w\|_{H^{s,\gamma}} + \sum^s_{k=0} \|(1+y)^{\gamma-1}
    \dx^k (u-U) \|_{L^2}\}.
\end{split}
\end{equation}
Furthermore, applying Wirtinger's inequality in the variable $x$ repeatedly and part (i)
of lemma \ref{l:Hardy}, we have
\[\left\{\begin{aligned}
    \sum^s_{k=1} \|(1+y)^{\gamma-1} \dx^k (u-U)\|_{L^2} & \le \frac{1+\pi^{-2s}}{1-\pi^{-2}}
    \| (1+y)^{\gamma-1} \dx^s (u-U)\|_{L^2} \\
    \|(1+y)^{\gamma-1} (u-U) \|_{L^2} & \le \frac{2}{2\gamma-1} \|(1+y)^\gamma \w\|_{L^2} \le \frac{2}{2\gamma-1} \|\w\|_{H^{s,\gamma}},
\end{aligned}\right.\]
and hence, there exists a constant $C_{s,\gamma} > 0$ such that
%
% eqn A3
\begin{equation} \label{e:Csgamma}
\begin{split}
    &\|\w\|_{H^{s,\gamma}} +  \|(1+y)^{\gamma-1} \dx^s (u-U)\|_{L^2}
    \le \|\w\|_{H^{s,\gamma}} + \sum^s_{k=0} \|(1+y)^{\gamma-1} \dx^k
    (u-U)\|_{L^2}\\
    &\qquad\qquad\qquad\le C_{s,\gamma} \{\|\w\|_{H^{s,\gamma}} +
    \|(1+y)^{\gamma-1} \dx^s (u-U) \|_{L^2}\}.
\end{split}
\end{equation}
Therefore, according to inequalities \eqref{e:defn w u-U} and \eqref{e:Csgamma}, it suffices
to prove
%
% eqn A4
\begin{equation}\label{e:cC ineq}
    c_\delta \|\w\|_{H^{s,\gamma}_g} \le \|\w\|_{H^{s,\gamma}} + \|(1+y)^{\gamma-1} \dx^s
    (u-U)\|_{L^2}
    \le C_{\gamma,\sigma,\delta} \{\|\w\|_{H^{s,\gamma}_g} + \|\dx^s U\|_{L^2} \}
\end{equation}
for some constants $c_\delta$ and $C_{\gamma, \sigma, \delta} > 0$.

The key idea of proving \eqref{e:cC ineq} is the following
%
% lemma A2
\begin{lem}[$L^2$ Comparison of $\dx^k (u-U), \dx^k\w$ and $g_k$]\label{l:L2 dxk u-U dxk w gk}
Let $s \ge 4$ be an integer, $\gamma \ge 1, \sigma > \gamma + \frac{1}{2}$ and
$\delta \in (0,1)$. If $\w\in \Hs (\TR)$, then for any $k =1,2, \cdots, s$,
%
% eqn A5
\begin{equation} \label{e:L2 dxk u-U dxk w gk}
    \|(1+y)^\gamma g_k\|_{L^2} \le \|(1+y)^\gamma \dx^k \w\|_{L^2} + \delta^{-2}
    \| (1+y)^{\gamma-1} \dx^k (u-U)\|_{L^2}
\end{equation}
where $g_k := \dx^k \w - \frac{\dy \w}{\w} \dx^k (u-U)$. In addition, if $u|_{y=0}=0$,
then for any $k=1,2, \cdots, s$,
%
% eqn A6
\begin{equation}\label{e:L2 dxk u-U dxk w gk uy0=0}
    \|(1+y)^\gamma \dx^k \w\|_{L^2} + \|(1+y)^{\gamma-1} \dx^k (u-U)\|
    \le C_{\gamma, \sigma, \delta} \{\|(1+y)^\gamma g_k \|_{L^2} + \|\dx^k U\|_{L^2(\T)} \}
\end{equation}
where $C_{\gamma, \sigma, \delta}$ is a constant depending on $\gamma, \sigma$ and $\delta$ only.
\end{lem}

Assuming lemma \ref{l:L2 dxk u-U dxk w gk}, which will be shown at the end of this appendix,
for the moment, we can show inequality \eqref{e:cC ineq} as follows.

Applying lemma \ref{l:L2 dxk u-U dxk w gk} for $k=s$, we obtain, from \eqref{e:L2 dxk u-U dxk w gk}
and \eqref{e:L2 dxk u-U dxk w gk uy0=0},
%
% eqn A7
\begin{equation} \label{e:lemA2 k=s A5 A6}
\begin{split}
    \frac{1}{2} \delta^4 \|(1+y)^\gamma g_s\|^2_{L^2} \le & \|(1+y)^\gamma \dx^s \w\|^2_{L^2}
    + \|(1+y)^{\gamma-1} \dx^s (u-U) \|^2_{L^2} \\
    & \qquad \qquad \le C_{\gamma, \sigma, \delta} \{ \|(1+y)^\gamma g_s\|^2_{L^2}
    + \|\dx^s U\|^2_{L^2(\T)} \}.
\end{split}
\end{equation}
Adding $\sum_{\substack{|\alpha| \le s \\ \alpha_1 \le s-1}} \| (1+y)^{\gamma+\alpha_2}
D^\alpha \w\|^2_{L^2}$
to \eqref{e:lemA2 k=s A5 A6}, we have
\[
    \frac{1}{2}\delta^4 \|\w\|^2_{H^{s,\gamma}_g} \le \|\w\|^2_{H^{s,\gamma}}
    + \| (1+y)^{\gamma-1} \dx^s (u-U)\|^2_{L^2}
    \le C_{\gamma, \sigma, \delta} \{\|\w\|^2_{H^{s,\gamma}_g}
    + \|\dx^s U\|^2_{L^2 (\T)}\}
\]
which implies inequality \eqref{e:cC ineq}.

\end{proof}

Finally, in order to complete the justification of the almost equivalence relation
\eqref{e:almost equiv}, we will prove the lemma \eqref{l:L2 dxk u-U dxk w gk} as
follows.
%
% pf lemma A2
\begin{proof}[Proof of lemma \ref{l:L2 dxk u-U dxk w gk}]
To prove \eqref{e:L2 dxk u-U dxk w gk}, let us first recall from the
definition of $\Hs$ that
%
% eqn A8
\begin{equation}\label{e:defn Hs}
\left\{\begin{aligned}
    \delta (1+y)^{-\sigma} \le  \w  & \le \delta^{-1} (1+y)^{-\sigma}\\
    \dy \w & \le \delta^{-1} (1+y)^{-\sigma-1},
\end{aligned}\right.
\end{equation}
so $\frac{\dy \w}{\w} \le \delta^{-2} (1+y)^{-1}$, and hence, for any $k =1,2,\cdots, s$,
\begin{align*}
    \|(1+y)^{\gamma} g_k\|_{L^2} & \le \| (1+y)^\gamma \dx^k \w\|_{L^2} +
    \|(1+y) \frac{\dy \w}{\w} \|_{L^\infty} \|(1+y)^{\gamma-1} \dx^k (u-U)\|_{L^2}\\
    & \le \|(1+y)^\gamma \dx^k \w\|_{L^2} + \delta^{-2} \|(1+y)^{\gamma-1}
    \dx^k (u-U)\|_{L^2}
\end{align*}
which is inequality \eqref{e:L2 dxk u-U dxk w gk}.

Next, we are going to show \eqref{e:L2 dxk u-U dxk w gk uy0=0}. The main observation is
that we can also rewrite $g_k :=\w\dx \left(\frac{\dx^k (u-U)}{\w}\right)$. Thus, applying
$\eqref{e:defn Hs}_1$ and part (ii) of lemma \ref{l:Hardy}, we have
%
% eqn A9
\begin{align}
    \|(1+y)^{\gamma-1} \dx^k (u-U)\|_{L^2}
    & \le \delta^{-1} \|(1+y)^{\gamma-\sigma-1} \frac{\dx^k (u-U)}{\w}\|_{L^2} \notag\\
    & \le C_{\gamma, \sigma} \delta^{-1} \left\{ \left\| \frac{\dx^k U}{\w|_{y=0}}
    \right\|_{L^2(\T)} + \left\|(1+y)^{\gamma-\sigma} \dy \left(\frac{\dx^k (u-U)}{\w}
    \right) \right\|_{L^2}  \right\} \notag\\
    & \le C_{\gamma, \sigma} \delta^{-2} \{\|\dx^k U\|_{L^2 (\T)}
    + \|(1+y)^\gamma g_k \|_{L^2} \}. \label{e:pf A6}
\end{align}
Now, using triangle inequality, \eqref{e:defn Hs} and \eqref{e:pf A6}, we also have
%
% eqn A10
\begin{align}
    \|(1+y)^\gamma \dx^k \w\|_{L^2} & \le \|(1+y)^\gamma g_k \|_{L^2}
    + \delta^{-2} \|(1+y)^{\gamma-1} \dx^k (u-U)\|_{L^2} \notag\\
    & \le C_{\gamma, \sigma, \delta} \{\|(1+y)^\gamma g_k\|_{L^2} + \|\dx^k U\|_{L^2 (\T)} \}.
    \label{e:pf A6 2}
\end{align}
Summing up \eqref{e:pf A6} and \eqref{e:pf A6 2}, we prove \eqref{e:L2 dxk u-U dxk w gk uy0=0}.

\end{proof}

% ==============================================================
%
%    B Calculus Inequalities
%
%==================================================================
\section{Calculus Inequalities}\label{s:appendixB}
In this appendix we will introduce several calculus inequalities for the incompressible
velocity field $(u,v)$, the vorticity $\w$ and the quantity $g_k$. These inequalities
are not related to any equations; they hold just because of elementary computations.

%------------------------------------------------------------
%
%   B.1    Basic Inequalities
%
%----------------------------------------------------------------
\subsection{Basic Inequalities}\label{ss:Basic Ineq}
In this subsection we will state without proof two elementary inequalities (i.e.,
lemma  \ref{l:Hardy} and lemma \ref{l:Sobolev} below).

Let us first introduce Hardy's type inequalities.
%
% lemma B.1
\begin{lem}[Hardy's Type Inequalities] \label{l:Hardy}
Let $f : \TR \to \R$. Then
\begin{itemize}
\item[(i)] if $\lambda > - \frac{1}{2}$ and $ \lim_{y \to +\infty} f(x,y) = 0$, then
%
%  eqn B.1
\begin{equation} \label{e:Hardy1}
    \|(1+y)^\lambda f\|_{L^2 (\TR)} \le \frac{2}{2\lambda +1}
    \| (1+y)^{\lambda +1} \partial_y f\|_{L^2 (\TR)};
\end{equation}
\item[(ii)] if $\lambda < - \frac{1}{2}$, then
%
% eqn B.2
\begin{equation}\label{e:Hardy2}
    \|(1+y)^\lambda f\|_{L^2 (\TR)} \le \sqrt{- \frac{1}{2\lambda +1}}
    \| f|_{y=0}\|_{L^2 (\T)} - \frac{2}{2\lambda + 1}
    \| (1+y)^{\lambda + 1} \partial_y f \|_{L^2 (\TR)}.
\end{equation}
\end{itemize}
\end{lem}

The proof of lemma \ref{l:Hardy} is elementary, so we leave it for the reader.

Next, we will state the following Sobolev's type inequality.
%
% lemma B.2
\begin{lem}\label{l:Sobolev}
Let $f: \TR \to \R$. Then there exists a universal constant $C > 0$ such that
%
% eqn B.3
\begin{equation}\label{e:Sobolev}
    \|f\|_{L^\infty (\TR)}  \le C \{\|f\|_{L^2 (\TR)} +\|\partial_x f\|_{L^2 (\TR)}
    + \|\partial^2_y f\|_{L^2 (\TR)} \}
\end{equation}
\end{lem}

To prove lemma \ref{l:Sobolev}, one may extend the domain of $f$ to $\R^2$ via the standard
extension argument. Then inequalities \eqref{e:Sobolev} follows easily by the Fourier's
inversion formula. We leave this to the reader as well.

%-------------------------------------------------------------------
%
%  B.2 Estimates for H Functions
%
%------------------------------------------------------------------
\subsection{Estimates for $\Hs$ Functions}\label{ss:est H}
In this subsection we will use the weighted norm $\|\cdot\|_{H^{s,\gamma}_g}$ to control
certain $L^2$ and $L^\infty$ norms of $u,v,\w,g_k$ and their derivatives. To derive
these estimates, we shall apply lemma \ref{l:Hardy} and lemma \ref{l:Sobolev}, which was
introduced previously in subsection \ref{ss:Basic Ineq}.

Our aim is to prove the following
%
% prop B.3
\begin{prop}[$L^2$ and $L^\infty$ Controls on $u,v,\w$ and $g_k$]\label{p:L2Linfty uvwgk}
Let the vector field $(u,v)$ defined on $\TR$ satisfy the incompressibility condition
$\dx u + \dy v =0$, the Dirichlet boundary condition $u|_{y=0} = v|_{y=0} \equiv 0$ and
$\lim_{y \to +\infty} u = U$. If the vorticity $\w := \dy u\in \Hs$ for some constants
$s \ge 4, \gamma \ge 1, \sigma > \gamma + \frac{1}{2}$ and $\delta \in (0,1)$, then we
have the following estimates: there exists a constant $\constC > 0$ such that\\
\hfill\\
{\bf Weighted $L^2$ Estimates}:
\begin{itemize}
\item[(i)] for all $k = 0,1, \cdots, s$,
%
% eqn B.4
\begin{equation} \label{e:B3 L2 u}
    \|(1+y)^{\gamma-1} \dx^k (u-U)\|_{L^2} \le \constC \{\|\w\|_{H^{s,\gamma}_g}
    + \|\dx^s U\|_{L^2} \},
\end{equation}
\item[(ii)] for all $k=0,1,\cdots,s-1$,
%
% eqn B.5
\begin{equation} \label{e:B3 L2 v}
    \left\| \frac{\dx^k v + y \dx^{k+1} U}{1+y} \right\|_{L^2}
    \le \constC \{\|\w\|_{H^{s,\gamma}_g} + \|\dx^s U\|_{L^2} \},
\end{equation}
\item[(iii)] for all $|\alpha| \le s$,
%
% eqn B.6
\begin{equation} \label{e:B3 L2 w}
    \|(1+y)^{\gamma+\alpha_2} D^\alpha \w\|_{L^2} \le
    \begin{cases}
        \constC \{\|\w\|_{H^{s,\gamma}_g} + \|\dx^s U\|_{L^2}\} & \text{ if } \alpha = (s,0)\\
        \|\w\|_{H^{s,\gamma}_g} & \text{ if } \alpha \ne (s,0).
    \end{cases}
\end{equation}
\item[(iv)] for all $k = 1,2, \cdots,s$,
%
% eqn B.7
\begin{equation} \label{e:B3 L2 gk}
    \|(1+y)^\gamma g_k \|_{L^2} \le
    \begin{cases}
        \constC \{\|\w\|_{H^{s,\gamma}_g} + \|\dx^s U\|_{L^2} \} & \text{ if }
        k=1,2,\cdots,s-1\\
        \|\w\|_{H^{s,\gamma}_g} & \text{ if } k=s,
    \end{cases}
\end{equation}
where the quantity $g_k := \dx^k \w - \frac{\dy \w}{\w} \dx^k (u-U)$.
\end{itemize}
\hfill\\
{\bf Weighted $L^\infty$ Estimates}:
\begin{itemize}
\item[(v)] for all $k=0,1,\cdots, s-1$,
%
% eqn B.8
\begin{equation} \label{e:B3 Linfty u}
    \|\dx^k u\|_{L^\infty} \le \constC \{\|\w\|_{H^{s,\gamma}_g} + \|\dx^s U\|_{L^2} \},
\end{equation}
\item[(vi)] for all $k=0,1,\cdots, s-2$,
%
% eqn B.9
\begin{equation} \label{e:B3 Linfty v}
    \left\| \frac{\dx^k v}{1+y} \right\|_{L^\infty}
    \le \constC \{\|\w\|_{H^{s,\gamma}_g} + \|\dx^s U\|_{L^2} \},
\end{equation}
\item[(vii)] for all $|\alpha| \le s-2$,
%
% eqn B.10
\begin{equation} \label{e:B3 Linfty w}
    \|(1+y)^{\gamma+\alpha_2} D^\alpha \w\|_{L^\infty}
    \le C_{s,\gamma} \|\w\|_{\Hg}.
\end{equation}
\end{itemize}

\end{prop}
%
% proof prop B.3
\begin{proof}
\hfill\\
(i) It follows from the definition of $\|\cdot\|_{H^{s,\gamma-1}}$ that $\|(1+y)^{\gamma-1}
\dx^k (u-U)\|_{L^2} \le \|u-U\|_{H^{s, \gamma-1}}$, so inequality \eqref{e:B3 L2 u} is a direct consequence of the almost equivalence inequality \eqref{e:almost equiv}.
\hfill\\
(ii) Applying part (ii) of lemma \ref{l:Hardy} and inequality \eqref{e:B3 L2 u}, we have
\[
    \left\|\frac{\dx^k v + y \dx^{k+1} U}{1+y} \right\|_{L^2}
    \le 2 \| \dx^{k+1} (u-U)\|_{L^2}
    \le \constC \{\|\w\|_{H^{s,\gamma}_g} + \|\dx^s U\|_{L^2} \}
\]
which is inequality \eqref{e:B3 L2 v}.
\hfill\\
(iii) Inequality \eqref{e:B3 L2 w} follows directly from the definition of
$\|\w\|_{H^{s,\gamma}}$ and inequality \eqref{e:almost equiv}.
\hfill\\
(iv) Since $\w\in\Hs$, we know that $\|(1+y) \frac{\dy w}{\w}\|_{L^\infty} \le \delta^{-2}$,
so using triangle inequality, inequalities \eqref{e:B3 L2 u} and \eqref{e:B3 L2 w}, we have
\[
    \|(1+y)^\gamma g_k\|_{L^2}
    \le \|(1+y)^\gamma \dx^k \w\|_{L^2} + \delta^{-2} \|(1+y)^{\gamma-1} \dx^k
    (u-U)\|_{L^2}
    \le \constC \{ \|\w\|_{H^{s,\gamma}_g} + \|\dx^s U\|_{L^2}\}
\]
which is inequality \eqref{e:B3 L2 gk} for $k=1,2,\cdots, s-1$. When $k=s$, the better upper bound in \eqref{e:B3 L2 gk} follows directly from the definition of $\|\w\|_{H^{s,\gamma}_g}$.
\hfill\\
(v) For any $k=1,2,\cdots, s-1$, applying lemma \ref{l:Sobolev}, inequalities
\eqref{e:B3 L2 u} and \eqref{e:B3 L2 w}, we have
\begin{align*}
    \|\dx^k (u-U)\|_{L^\infty}
    & \le C \{\|\dx^k (u-U)\|_{L^2} + \|\dx^{k+1} (u-U)\|_{L^2} + \|\dx^k\dy \w\|_{L^2}\}\\
    & \le \constC \{\|\w\|_{H^{s,\gamma}_g} + \|\dx^s U\|_{L^2} \},
\end{align*}
and hence, by the triangle inequality and $\|\dx^k U\|_{L^\infty} \le C_s \|\dx^s U\|_{L^2}$,
we justify \eqref{e:B3 Linfty u}.

For the case $k=0$, let us first recall from the hypothesis that $\w:=\dy u>0$, so
%
% eqn B.11
\begin{equation} \label{e:recall w}
    0\le u \le U = \int^{+\infty}_0 \w\,dy.
\end{equation}
Thus, using Cauchy-Schwarz's inequality and estimate \eqref{e:B3 L2 w}, we have
\[
    \|U\|_{L^2}^2
     = \int_\T \left| \int^{+\infty}_0 \w\,dy \right|^2\,dx
     \le \frac{1}{2\gamma-1} \|(1+y)^\gamma \w\|_{L^2}^2
     \le \frac{1}{2\gamma-1} \|\w\|_{H^{s,\gamma}_g}^2,
\]
and hence, by \eqref{e:recall w}, Sobolev inequality and
$\|\dx U\|_{L^2} \le C_s \|\dx^s U\|_{L^2}$, we obtain
\[
    \|u\|_{L^\infty}
     \le \|U\|_{L^\infty} \le C\{\|U\|_{L^2} + \|\dx U\|_{L^2}\}
     \le \constC \{\|\w\|_{H^{s,\gamma}_g} + \|\dx^s U\|_{L^2}\}
\]
which is inequality \eqref{e:B3 Linfty u}.
\hfill\\
(vi) Applying triangle inequality, lemma \ref{l:Sobolev},
$\dx u + \dy v = 0$ and $\w = \dy u$, we have
\begin{align*}
    \left\|\frac{\dx^k v}{1+y}\right\|_{L^\infty}
    & \le \left\| \frac{y \dx^{k+1} U}{1+y} \right\|_{L^\infty}
    + \left\| \frac{\dx^k v + y \dx^{k+1} U}{1+y} \right\|_{L^\infty}\\
    & \le C \Big\{ \| \dx^{k+1} U\|_{L^2 (\T)} + \|\dx^{k+2} U\|_{L^2 (\T)}
    + \left\| \frac{\dx^k v + y \dx^{k+1} U}{1+y} \right\|_{L^2}
    + \left\| \frac{\dx^{k+1} v + y \dx^{k+2} U}{1+y} \right\|_{L^2}\\
    & \qquad +  \left\| \frac{\dx^k v + y \dx^{k+1} U}{(1+y)^3} \right\|_{L^2}
    +  \left\| \frac{\dx^{k+1} (u-U)}{(1+y)^2} \right\|_{L^2}
    +  \left\| \frac{\dx^{k+1} \w}{1+y} \right\|_{L^2} \Big\}
\end{align*}
which implies inequality \eqref{e:B3 Linfty v} because of Wirtinger's inequality
and \eqref{e:B3 L2 u} - \eqref{e:B3 L2 w}.
\hfill\\
(vii) Inequality \eqref{e:B3 Linfty w} follows directly from lemma \ref{l:Sobolev} and
inequality \eqref{e:B3 L2 w}.

\end{proof}

%--------------------------------------------------------------------
%
%   B.3  Estimates for Functions Vanishing at Infinity
%
%--------------------------------------------------------------------
\subsection{Estimates for Functions Vanishing at Infinity}\label{ss:est vanish infty}
In this subsection we will first define certain weighted norms involving time and
spatial derivatives. Then we will state two basic inequalities about these norms, see
proposition \ref{p:basic ineq weighted} below. Finally, in proposition \ref{p:cntl uv} we will
control the weighted norms of $u$ and $v$ by that of $\w$ provided that
$u-U$ and its derivatives vanish at $y = +\infty$. The vanishing hypotheses
(i.e., the decay rates) are usually guaranteed by proposition \ref{p:decay Hs+3}
in applications.

Let us begin by defining the weighted norms.
%
%  def B4
\begin{defn} \label{defn:weighted norms}
For any $s' \in \N$ and $\gamma \in \R$, we define
\[\begin{aligned}
    \||\cdot|\|^2_{s',\gamma} & := \sum^{[\frac{s'}{2}]}_{l=0} \| \dt^l
    \cdot\|^2_{H^{s'-2l,\gamma} (\TR)}, \qquad
    \||\cdot|\|^2_{s'} := \sum^{[\frac{s'}{2}]}_{l=0} \|\dt^l \cdot
    \|^2_{H^{s'-2l}(\T)}, \\
    \||\cdot|\|^2_{s',\infty,\gamma} & := \sum^{[\frac{s'}{2}]}_{l=0}
    \sum_{|\alpha| \le s-2l} \|(1+y)^{\gamma + \alpha_2} \dt^l \Dalpha \cdot
    \|^2_{L^\infty (\TR)} \qquad \text{ and }\\
    \||\cdot|\|^2_{s',\infty} & := \sum^{[\frac{s'}{2}]}_{l=0}
    \| \dt^l \cdot \|^2_{W^{s'-2l,\infty} (\T)}
\end{aligned}\]
where $[\frac{s'}{2}]$ denotes the largest integer which is less than or
equal to $\frac{s'}{2}$.
\end{defn}

Using H\"{o}lder inequality and Sobolev inequality, one can easily show the
following
%
%  prop B5
\begin{prop}[Basic Inequalities for Weighted Norms] \label{p:basic ineq weighted}\hfill
\begin{itemize}
\item[(i)] For any $s' \in \N$ and $\gamma, \gamma_1, \gamma_2 \in \R$ with $\gamma
= \gamma_1 + \gamma_2$,
\[
    \||F_1 F_2|\|_{s',\gamma} \le C_{s'} \|| F_1 |\|_{s',\infty,\gamma_1}
    \||F_2|\|_{s',\gamma_2}.
\]
\item[(ii)] For any $s' \ge 5$ and $\gamma, \gamma_1, \gamma_2 \in \R$ with $\gamma
= \gamma_1 +\gamma_2$,
\[
    \||F_1 F_2|\|_{s',\gamma} \le C_{s', \gamma_1, \gamma_2} \{\|| F_1 |\|_{s'-1, \gamma_1}
    \||F_2|\|_{s',\gamma_2} + \|| F_1 |\|_{s', \gamma_1} \||F_2|\|_{s'-1,\gamma_2} \}.
\]
\end{itemize}
\end{prop}

Next, we will state the weighted controls on $u$ and $v$ as follows.
%
%
%  prop B6
\begin{prop}[Weighted Controls on $u$ and $v$] \label{p:cntl uv}
For any $s' \ge 4$ and $\gamma \ge 1$, let the vector field $(u,v)$ defined
on $\TR$ satisfying the incompressibility condition $\dx u + \dy v = 0$, the
Dirichlet boundary condition $v|_{y=0} = 0$ and $\lim_{y \to + \infty}
\dt^l \dx^k u = \dt^l \dx^k U$ for all $l=0,1,\cdots, [\frac{s'}{2}]$ and
$k = 0,1,\cdots, s'-2l+1$. Denote the vorticity $\w := \dy u$. Then there
exists a universal constant $C > 0$ such that
%
%
%  eqn B.12
\begin{equation} \label{e:norm u-U 0}
    \||u-U|\|_{s',0} \le C \||\w|\|_{s',\gamma} \qquad \text{ and }
    \qquad  \|| v+ y \dx U|\|_{s',-1} \le C \|| \dx \w|\|_{s',\gamma}.
\end{equation}
\end{prop}
\begin{proof}[Outline of the proof]
The hypotheses of proposition \ref{p:cntl uv} allow us to apply lemma \ref{l:Hardy} to
$\dt^l \dx^k (u-U)$ and $\dt^l \dx^k (v+y \dx U)$ provided that $2l + k \le s'$,
so we obtain $\|| u-U|\|_{s',0} \le C\||\w|\|_{s',1}$ and $\||v+y \dx U|\|_{s',-1} \le
C \|| \dx \w|\|_{s',1}$ which imply \eqref{e:norm u-U 0} since $\gamma \ge 1$.
\end{proof}

%========================================================================
%
% C. Decay Rates for \Hs Functions
%
%=======================================================================
\section{Decay Rates for $\Hs$ Functions}\label{s:appendixC}
The aim of this appendix is to prove that the actual pointwise decay rates of $\Hs$
functions at $y = + \infty$ are better than the decay rates obtained by the Sobolev
embeddings. The proof relies on a pointwise interpolation argument (see lemma
\ref{l:ptwise interpolation} below), which is a direct
consequence of the Taylor's series expansion.

More specifically, we will prove the decay property of $D^\alpha \w$ as $y$ goes to
$+ \infty$ as follows.
%
% prop C.1
\begin{prop}[Decay Rates for $\Hs$ Functions]\label{p:decay Hs+3}
Let $s' \ge 4$ be an integer, $\gamma \ge 1, \sigma > \gamma + \frac{1}{2}$ and $\delta \in
(0,1)$. If $\w \in H^{s'+4,\gamma}_{\sigma, \delta}$, then $\w$ is $s'+2$ times differentiable and there exists a
constant $\constCw > 0$ such that for all $|\alpha| \le s'+2$,
%
% eqn C.1
\begin{equation}\label{e:decay Hs+3}
    |D^\alpha \w| \le \constCw (1+y)^{-b_\alpha} \qquad\qquad \text{ in } \TR
\end{equation}
where the exponent
%
% eqn C.2
\begin{equation}\label{e:decay Hs+3 e}
    b_\alpha :=
    \begin{cases}
        \sigma + \alpha_2 & \text{ if } |\alpha| \le 2\\
        \frac{\sigma + (2^{|\alpha|-2}-1) \gamma}{2^{|\alpha|-2}} + \alpha_2 & \text{ if }
        2 \le |\alpha| \le s'+1\\
        \gamma + \alpha_2 & \text{ if } |\alpha| = s'+2.
    \end{cases}
\end{equation}
\end{prop}
%
% rmk C.2
\begin{rem}[Decay Rates from Sobolev Embeddings] \label{r:decay sobolev}
Using the standard Sobolev embedding argument and the definition of
$H^{s'+4, \gamma}_{\sigma, \delta}$, one may prove that if
$\w \in H^{s'+4, \gamma}_{\sigma, \delta}$, then $\w$ is $s'+2$ times differentiable
and there exists a constant $C_{s,\gamma} > 0$ such that
%
% eqn C.3
\begin{equation}\label{e:Hs+3 sobolev}
    |D^\alpha \w| \le
    \begin{cases}
        \delta^{-1} (1+y)^{-\sigma-\alpha_2} & \text{ if } |\alpha| \le 2\\
        C_{s,\gamma} \|\w\|_{H^{s'+4,\gamma}} (1+y)^{-\gamma-\alpha_2} & \text{ if }
        2 \le |\alpha| \le s'+2.
    \end{cases}
\end{equation}
Thus, the interesting part of proposition \ref{p:decay Hs+3} is the decay rate of
\eqref{e:decay Hs+3} is better than that of \eqref{e:Hs+3 sobolev}. This
slightly better pointwise decay will help us to deal with the boundary terms at $y = + \infty$
while we are integrating by parts in the $y$-direction (cf remark \ref{r:bdry infty}).

\end{rem}
%
% pf of prop C.1
\begin{proof}[Proof of proposition \ref{p:decay Hs+3}]
According to remark \ref{r:decay sobolev}, we are only required to justify the inequality
\eqref{e:decay Hs+3} with the decay rate defined in \eqref{e:decay Hs+3 e}.

First of all, let us state without proof the following calculus lemma.
%
% lemma C.3
\begin{lem}[Pointwise Interpolation]\label{l:ptwise interpolation}
Let $f :\TR \to \R$ be a twice differentiable function. Then we have the following:
\begin{itemize}
\item[(i)] If there exist constants $C_0, C_2, b_0$ and $b_2$ such that
$|\dx^i f|  \le C_i (1+y)^{-b_i}$ for all $i = 0, 2$, then
\[
    |\dx f| \le 2 \sqrt{C_0C_2} (1+y)^{-\frac{b_0+b_2}{2}} \qquad\qquad \text{ in } \TR.
\]
\item[(ii)] If there exist non-negative constants $C_0, C_2, b_0$ and $b_2$ such that
$|\dy^i f|  \le C_i (1+y)^{-b_i}$ for all $i = 0, 2$, then
\[
    |\dy f| \le 2 \sqrt{C_0C_2} (1+y)^{-\frac{b_0+b_2}{2}} \qquad\qquad \text{ in } \TR.
\]
\end{itemize}
\end{lem}

The proof of lemma \ref{l:ptwise interpolation} is based on the standard Taylor's series
expansion technique, and will be omitted here.

Now, applying lemma \ref{l:ptwise interpolation} to $D^\alpha \w$ inductively on $|\alpha|
= 3, 4, \cdots, s'+1$ with the inequality \eqref{e:Hs+3 sobolev}, we prove \eqref{e:decay Hs+3}
with the exponent $b_\alpha$ defined in \eqref{e:decay Hs+3 e}.
\end{proof}
%
% rmk C.4
\begin{rem}[Further Improvement on the Decay Rate]
The decay rate $b_\alpha$ defined in \eqref{e:decay Hs+3 e} is obviously not optimal
because one can apply the pointwise interpolation lemma \ref{l:ptwise interpolation}
again to further improve it. However, we do not intend to optimize it here. Indeed,
repeatedly applying the pointwise interpolation lemma \ref{l:ptwise interpolation}, one may
improve the decay rate $b_\alpha$ as
\[
    b_\alpha :=
    \begin{cases}
        \sigma + \alpha_2 & \text{ if } |\alpha| \le 2\\
        \frac{(s'+2-|\alpha|) \sigma + (|\alpha| - 2) \gamma}{s'} + \alpha_2^- &
        \text{ if } 3 \le |\alpha| \le s'+1\\
        \gamma + \alpha_2 & \text{ if } |\alpha| = s'+2.
    \end{cases}
\]
We leave this proof for the interested reader.
\end{rem}

%==============================================================
%
% D. Equations for $\ae$ and $g_s^\e$
%
%===============================================================
\section{Equations for $\aep$ and $\gep_s$}\label{s:appendixD}
In this appendix we will derive the evolution equations for $\aep := \frac{\dy \wep}{\wep}$
and $\gep_s := \dx^s \wep - \aep \dx^s (\uep -U)$ provided that $\wep > 0,
(\uep, \vep, \wep)$ and $(\pep, U)$ satisfy \eqref{e:RPE} - \eqref{e:Ruv}. These
derivations just follow from direct computations.
\\
\\
\underline{Equation for $\aep$:}

Differentiating the vorticity equation $\eqref{e:Rvort}_1$ with respect to $y$ once,
we obtain
%
% eqn D.1
\begin{equation}\label{e:dy wep}
    (\dt + \uep\dx + \vep\dy) \dy \wep = \e^2 \dxx \dy \wep + \dy^3 \wep - \wep
    \dx \wep + \dx\uep \dy\wep.
\end{equation}
Using \eqref{e:dy wep} and the vorticity equation $\eqref{e:Rvort}_1$, we can compute
%
% eqn D.2
\begin{align}
    (\dt + \uep\dx+ \vep \dy) \aep \notag
    & = \;\frac{(\dt + \uep \dx + \vep \dy) \dy \wep}{\wep} -
    \frac{\dy \wep (\dt + \uep \dx + \vep \dy)\wep}{(\wep)^2} \notag\\
    & = \;\e^2 \left\{ \frac{\dxx \dy \wep}{\wep} -\aep \frac{\dxx \wep}{\wep} \right\}
    + \left\{ \frac{\dy^3 \wep}{\wep} - \aep \frac{\dyy \wep}{\wep} \right\}
    - \dx \wep + \aep \dx \uep. \label{e:d aep}
\end{align}
On the other hand, by direct differentiations only, one may check that
%
% eqn  D.3
\begin{equation} \label{e:dxx dyy aep}
\left\{\begin{aligned}
    \dxx \aep & = \frac{\dxx\dy\wep}{\wep} - \aep \frac{\dxx \wep}{\wep}
    - 2 \frac{\dx \wep}{\wep} \dx\aep\\
    \dyy \aep & = \frac{\dy^3\wep}{\wep} - \aep \frac{\dyy \wep}{\wep}
    - 2 \aep \dy\aep.
\end{aligned}\right.
\end{equation}
Substituting \eqref{e:dxx dyy aep} into \eqref{e:d aep}, we obtain an equation for
$\aep$:
%
% eqn D.4
\begin{equation} \label{e:aep}
    (\dt+\uep\dx+\vep\dy-\e^2\dxx-\dyy)\aep
    = 2\e^2 \frac{\dx\wep}{\wep}\dx\aep + 2 \aep\dy \aep - \gep_1 +\aep\dx U,
\end{equation}
where $\gep_1 := \dx\wep - \aep \dx (\uep-U)$.
\\
\\
\underline{Equation for $\gep_s$:} (Derivation of equation \eqref{e:gs evlt})

Differentiating the evolution equations for $\wep$ and $\uep-U$ (i.e., equations
\eqref{e:evlt eqn}) with respect to $x\;s$ times, we have
%
% eqn D.5
\begin{equation} \label{e:d evlt s}
\left\{\begin{aligned}
    &(\dt + \uep\dx +\vep\dy - \e^2 \dxx-\dyy) \dx^s \wep + \dx^s \vep \dy\wep\\
    &\qquad \qquad = - \sum^{s-1}_{j=0} \binom{s}{j} \dx^{s-j} \uep \dx^{j+1} \wep
    - \sum^{s-1}_{j=1} \binom{s}{j} \dx^{s-j} \vep \dx^j \dy \wep\\
    &(\dt + \uep\dx +\vep\dy - \e^2 \dxx-\dyy) \dx^s (\uep-U) + \dx^s \vep \wep\\
    &\qquad \qquad = - \sum^{s-1}_{j=0} \binom{s}{j} \dx^{s-j} \uep \dx^{j+1} (\uep-U)
    - \sum^{s-1}_{j=1} \binom{s}{j} \dx^{s-j} \vep \dx^j  \wep\\
    & \qquad\qquad\qquad - \sum^s_{j=0} \binom{s}{j} \dx^j (\uep-U) \dx^{s-j+1} U.
\end{aligned} \right.
\end{equation}
To eliminate the problematic term $\dx^s \vep$, we subtract $\aep \times
\eqref{e:d evlt s}_2$ from $\eqref{e:d evlt s}_1$, and obtain
%
% eqn D.6
\begin{equation} \label{e:eliminate dxsvep}
\begin{split}
    & (\dt+\uep\dx+\vep\dy-\e^2\dxx-\dyy) \gep_s + \{(\dt+\uep\dx+\vep\dy-\e^2\dxx-\dyy)
    \aep\} \dx^s (\uep-U)\\
    = & \; 2 \e^2 \dx^{s+1}(\uep-U) \dx \aep + 2 \dx^s \wep \dy \aep - \sum^{s-1}_{j=0} \binom{s}{j}
    \gep_{j+1} \dx^{s-j} \uep \\
    & \; - \sum^{s-1}_{j=1} \binom{s}{j} \dx^{s-j} \vep \{\dx^j \dy\wep-\aep\dx^j\wep\}
    + \aep \sum^s_{j=0} \binom{s}{j} \dx^j (\uep-U) \dx^{s-j+1} U.
\end{split}
\end{equation}
Substituting \eqref{e:aep} into \eqref{e:eliminate dxsvep}, we obtain equation \eqref{e:gs evlt}.

%====================================================================
%
%   Appendix E : Classical Max. Principle
%
%=====================================================================
\section{Classical Maximum Principles} \label{s:appendixE}
The main purpose of this appendix is to state two classical maximum principles,
which are useful in subsection \ref{ss:Linfty lower}, for parabolic equations.

The first lemma is the maximum principle for bounded solutions to parabolic
equations.
%
%
%   lemma E.1   Max Principle for Parabolic Equations
\begin{lem}[Maximum Principle for Parabolic Equations] \label{l:max parabolic}
Let $\e \ge 0$. If $H \in C([0,T]; C^2 (\TR) \cap C^1 ([0,T]; C^0 (\TR))$ is a
bounded function which satisfies the differential inequality:
\[
    \{\dt+b_1\dx+b_2\dy-\e^2\dxx - \dyy\} H \le f H \qquad\qquad \text{ in }
    [0,T] \times \TR
\]
where the coefficients $b_1, b_2$ and $f$ are continuous and satisfy
%
%
% eqn E.1
\begin{equation} \label{e:b2 f eqn}
    \left\| \frac{b_2}{1+y} \right\|_{L^\infty ([0,T] \times \TR)}  < + \infty
    \qquad \text{ and } \qquad \|f\|_{L^\infty ([0,T] \times \TR)}  \le \lambda,
\end{equation}
then for any $t \in [0,T]$,
%
%  eqn E.2
\begin{equation} \label{e:sup H}
    \sup_{\TR} H(t) \le \max \{e^{\lambda t} \|H(0)\|_{L^\infty (\TR)},
    \max_{\tau \in [0,t]} \{e^{\lambda (t-\tau)} \|H(\tau) |_{y=0}
    \|_{L^\infty (\T)} \} \}.
\end{equation}
\end{lem}

The proof of lemma \ref{l:max parabolic} is a direct application of the classical
maximum principle. For the reader's convenience, we will outline its proof as
follows.
%
%
%  outline of proof of lemma E.1
\begin{proof}[Outline of the proof of lemma \ref{l:max parabolic}]
For any $\mu > 0$, let us define $\mathcal{H} := e^{-\lambda t} H - \mu
\left\| \frac{b_2}{1+y} \right\|_{L^\infty} t - \mu \ln (1+y)$. Then one may
check that for any $\tilde{t} \in (0,T]$,
\[
    \{\dt+b_1\dx+b_2 \dy -\e^2 \dxx - \dyy + (\lambda - f) \} \mathcal{H} < 0
    \qquad\qquad \text{ in } [0, \tilde{t}] \times \TR,
\]
so by the classical maximum principle for parabolic equations, we have
\[
    \max_{[0,\tilde{t}] \times \T \times [0,R]} \mathcal{H}
    \le \max \{\|H(0)\|_{L^\infty (\TR)}, \max_{\tau \in [0, \tilde{t}]}
    \{e^{-\lambda \tau} \|H(\tau) |_{y=0} \|_{L^\infty (\T)} \} \}
\]
provided that $R \ge exp \left(\frac{1}{\mu} \|H\|_{L^\infty} \right) -1$.
Therefore, for any $(x,y) \in \TR$, we have
\[\begin{split}
    & \; H(\tilde{t},x,y) - \mu  \left\| \frac{b_2}{1+y}\right\|_{L^\infty} e^{\lambda
    \tilde{t}}  \tilde{t} - \mu e^{\lambda \tilde{t}} \ln (1+y)\\
    \le &\; \max \{e^{\lambda \tilde{t}} \|H(0)\|_{L^\infty (\TR)}, \max_{\tau \in [0,\tilde{t}]}
    \{ e^{\lambda (\tilde{t} -\tau)} \|H(\tau) |_{y=0} \|_{L^\infty (\T)} \} \}
\end{split}\]
which implies \eqref{e:sup H} if we take the limit $\mu \to 0^+$ and
replace the arbitrary time $\tilde{t}$ by $t$.

\end{proof}

The second lemma is the lower bound estimate on bounded solutions for
parabolic equations.
%
%
%   lemma E.2
\begin{lem}[Minimum Principle for Parabolic Equations] \label{l:min parabolic}
Let $\e \ge 0$. If $H \in C([0,T]; C^2 (\TR) ) \cap C^1 ([0,T]; C^0 (\TR))$ is a bounded
function with
\[
    \kappa (t) := \min \{\min_{\TR} H(0), \min_{[0,t] \times \T} H|_{y=0} \} \ge 0
\]
and satisfies
\[
    \{\dt + b_1\dx + b_2\dy -\e^2\dxx - \dyy\} H = f H
\]
where the coefficients $b_1, b_2$ and $f$ are continuous and satisfy \eqref{e:b2 f eqn},
then for any $t \in [0,T]$,
%
%
%  eqn E.3
\begin{equation} \label{e:min H}
    \min_{\TR} H(t) \ge (1 - \lambda t e^{\lambda t} ) \kappa (t).
\end{equation}

\end{lem}

The proof of lemma \ref{l:min parabolic} is also standard and very similar to that
of lemma \ref{l:max parabolic}. We will outline it here for the reader's
convenience as well.
%
%
%   outline proof of lemma E.2
\begin{proof} [Outline of the proof of lemma \ref{l:min parabolic}]
For any fixed $\tilde{t} \in [0,T]$ and $\mu > 0$, let us define $h := e^{- \lambda t}
\{H - \kappa (\tilde{t})\} + \left\{\lambda \kappa (\tilde{t}) + \mu \left\|
\frac{b_2}{1+y} \right\|_{L^\infty} \right\} t + \mu \ln (1+y)$. Then one may check that
\[
    \{\dt + b_1 \dx + b_2 \dy - \e^2 \dxx - \dyy + (\lambda - f) \} h > 0
    \qquad\qquad \text{ in } [0, \tilde{t}] \times \TR,
\]
so by the classical maximum principle for parabolic equations, we have
\[
    \min_{[0,\tilde{t}]\times \T \times [0,R]} h \ge 0
\]
provided that $R \ge exp \left(\frac{1}{\mu} \{ \| H \|_{L^\infty} + \kappa (\tilde{t})\}
\right) - 1$. Taking the limit $R \to + \infty$, and then $\mu \to 0^+$, we obtain
\[
    H(\tilde{t}) \ge (1 - \lambda \tilde{t} e^{\lambda \tilde{t}} ) \kappa (\tilde{t})
\]
which implies inequality \eqref{e:min H} if we replace the arbitrary time
$\tilde{t}$ by $t$.

\end{proof}

\subsection*{Acknowledgements}
Nader Masmoudi was partially supported by an NSF-DMS grant  0703145. Tak Kwong Wong
was partially supported by The Croucher Foundation.

%==============================================================
%
%       Bibliography
%
%===============================================================

 % \bibliographystyle{amsalpha}
%bibliography{../biblio}
 % \bibliography{PrandtlReferences}
%\end{document}
\def\cprime{$'$}

\end{document}